\documentclass[12pt,a4paper]{amsart}

\usepackage{amsfonts,amsmath,amssymb}
\usepackage{latexsym,fullpage,amsfonts,amssymb,amsmath,amscd,graphics}

\usepackage[all]{xy}
\usepackage{amssymb,amsthm,amsxtra}
\usepackage{dsfont}
\usepackage[dvipsnames]{xcolor}
\usepackage{amscd}
\usepackage{tkz-euclide}
\usepackage{stmaryrd}
\usepackage{mathrsfs}
\usepackage{url}
\usepackage{bbm}
\usepackage{graphicx}
\usepackage{caption}
\usepackage{subcaption}
\usepackage{wasysym}
\usepackage{enumitem}
\setenumerate{label=\alph*),leftmargin=.6cm}
\usepackage[vcentermath]{youngtab}%
\usepackage{framed}
\usepackage{extarrows}
\usepackage[normalem]{ulem}

\usepackage{upgreek}

\usepackage[breaklinks=true]{hyperref}
\hypersetup{ unicode, bookmarksnumbered, linktoc = all, hidelinks
}
\usepackage[capitalise]{cleveref}

\usepackage[citestyle=numeric,bibstyle=numeric, giveninits=true,doi=false,url=false,isbn=false,
parentracker=false, backend=biber, maxnames=10, date=year, sorting=nty]{biblatex} 

\input{hacks}

\usepackage{tikz}
\usetikzlibrary{decorations.markings}
\usetikzlibrary{cd}
\usetikzlibrary{patterns}
\usetikzlibrary{shapes}

\tikzset{anchorbase/.style={baseline={([yshift=-0.5ex]current bounding box.center)}}}
\tikzset{wipe/.style={white,line width=4pt}}

\tikzset{->-/.style={decoration={ markings, mark=at position #1 with
  {\arrow{>}}},postaction={decorate}}} \tikzset{-<-/.style={decoration={ markings, mark=at position
  #1 with {\arrow{<}}},postaction={decorate}}}

 \tikzset{darkg/.style={green!70!black}}

\theoremstyle{plain}
\newtheorem*{theorem*}{Theorem}
\newtheorem*{remark*}{Remark}
\newtheorem*{example*}{Example}
\newtheorem{lemma}{Lemma}[subsection]
\newtheorem{proposition}[lemma]{Proposition}
\newtheorem{corollary}[lemma]{Corollary}
\newtheorem{theorem}[lemma]{Theorem}

\newtheorem*{conjecture*}{Conjecture}

\theoremstyle{definition}
\newtheorem{definition}[lemma]{Definition}

\theoremstyle{remark}

\newcommand{\tr}{\operatorname{tr}} 
\newcommand{\HC}{\operatorname{HC}}

\newcommand{\Hom}{\operatorname{Hom}} 
\newcommand\Homp{\operatorname{Hom^{\prime}}}

\newcommand{\diag}{\operatorname{diag}} 

\newcommand{\sdim}{\operatorname{sdim}} 
\newcommand{\str}{\operatorname{str}}

\newcommand{\ad}{\operatorname{ad}} 
\newcommand{\at}{\operatorname{at}}
\newcommand{\defect}{\operatorname{def}} 
\renewcommand{\Im}{\operatorname{im}}
\newcommand{\Ker}{\operatorname{ker}} 
\newcommand{\Coker}{\operatorname{{coker}}}

\newcommand{\WCl}{\operatorname{WCl}}

\newcommand{\ann}{\operatorname{ann}}

\def\Weyl{{\mathscr W(\odd)}} 
 
\def\CC{{\mathbb C}} 
\def\RR{{\mathbb R}}
\def\QQ{{\mathbb Q}}
\def\PP{{\mathbb P}}
\def\UE{{\mathfrak U}} 
\def\ZC{{\mathfrak Z}} 
\def\SU{{\mathrm{SU}}} 
\def\Uone{{\mathrm{U}}}
\def\HH{{\mathcal{H}}}
\def\hh{{\mathfrak h}} 
 
\def\kk{{\mathfrak k}} 
\def\uu{{\mathfrak u}}
 
\def\Oss{{\mathfrak{O}}} 
\def\slmn{{\mathfrak{sl}(m\vert n)}}
\def\supqn{{\mathfrak{su}(p,q\vert n)}} 
\def\glmn{{\mathfrak{gl}(m\vert n)}} 

\def\LL{{\mathcal{L}}}

\def\YY{{\mathcal Y}} 
\def\qq{{\mathfrak q}}

\def\bb{{\mathfrak b}} 
\def\even{{\mathfrak{g}_{0}}} 
\def\odd{{\mathfrak{g}_{1}}}

\def\tt{{\mathfrak t}} 
\def\NN{{\mathbb N}} 
\def\nn {{\mathfrak{n}}} 

\def\pr{{\mathrm{pr}}} 
\def\gg{{\mathfrak{g}}} 
\def\S{{\mathrm{S}}} 
 
\def\DS{{\operatorname{DS}}}
\def\rform{{\gg_{0}^{\RR}}}
\def\rformg{{G_{0}^{\RR}}}
\def\sccg{{\widetilde{G}_{0}^{\RR}}}
\def\reg{{\it reg}}

\def\ZZ{{\mathbb Z}} 

\def\gsmod{{\mathfrak{g}}\textbf{-smod}} 
\def\ugsmod{({\mathfrak{g}},\omega)\textbf{-usmod}} 
\def\ugsmodp{({\mathfrak{g}},\omega)\textbf{-usmod$^{\prime}$}} 
\def\gxsmod{{\mathfrak{g}_{x}}\textbf{-smod}}
 
\def\gmod{{\mathfrak{g}_{0}}\textbf{-mod}}
\def\evsmod{{\mathfrak{g}_{0}}\textbf{-smod}} 
 
\def\calC{{\mathcal C}} 
\def\calS{{\mathcal S}} \def\calI{{\mathcal I}}

\newcommand{\Span}{{\operatorname{Span}}}

\newcommand{\len}{{\text{len}}}

\newcommand{\rk}{{\operatorname{rk}}}
\newcommand{\End}{\operatorname{End}} 
\newcommand{\gr}{\operatorname{gr}}

\newcommand{\gl}{{\mathfrak{gl}}}

\newcommand{\Ad}{\operatorname{Ad}} 

\newcommand{\abs}[1]{\left|{#1}\right|}

\newcommand{\bil}{B}
\newcommand{\Dirac}{\operatorname{D}}

\newcommand{\ev}{\mathrm{ev}}

\newcommand{\comment}[1]{}

\def\psl{{\mathfrak{psl}}}
\def\psu{{\mathfrak{psu}}} 
\def\su{{\mathfrak{su}}} 

\def\dd{{\mathfrak{d}}}
\def\sqrttwo{{\textstyle 
\frac{1}{\sqrt{2}}}}
 
\def\onehalf{{\textstyle \frac 12}}

\def\sl{{\mathfrak{sl}}} 
\def\u{{\mathfrak{u}}} 
\def\so{{\mathfrak{so}}} 
\def\spo{{\mathfrak{spo}}}

\newcommand\twedge{\mbox{\text{\Large$\wedge$}}} 
\def\id{{\rm id}}

\newcommand\topa[2]{\genfrac{}{}{0pt}{2}{\scriptstyle #1}{\scriptstyle #2}} 

\DeclareFontFamily{U}{stix2bb}{} 
\DeclareFontShape{U}{stix2bb}{m}{n} {<-> stix2-mathbb}{}

\begin{document}

\title{An index for unitarizable $\slmn$-supermodules}

\author{Steffen Schmidt}
\address{S.S.: Institut für Mathematik, Heidelberg University, Germany}
\email{stschmidt@mathi.uni-heidelberg.de}

\author{Johannes Walcher}
 \address{J.W.: Institut für Mathematik and Institut f\"ur Theoretische Physik, Heidelberg
 University, Germany}
 \email{walcher@mathi.uni-heidelberg.de}
 
\begin{abstract}
The ``superconformal index'' is a character-valued invariant attached by theoretical physics to unitary 
representations of Lie superalgebras, such as $\su(2,2\vert n)$, that govern certain
quantum field theories. The index can be calculated as a supertrace over Hilbert space, and is
constant in families induced by variation of physical parameters. This is because the index receives
contributions only from ``short'' irreducible representations such that it is invariant
under recombination at the boundary of the region of unitarity.

The purpose of this paper is to develop these notions for unitarizable supermodules over the special 
linear Lie superalgebras $\slmn$ with $m\ge 2$, $n\ge 1$. To keep it self-contained, we include a 
fair amount of background material on structure theory, unitarizable supermodules, the Duflo-Serganova
functor, and elements of Harish-Chandra theory. Along the way, we provide a precise dictionary 
between various notions from theoretical physics and mathematical terminology. Our final result is a kind
of ``index theorem'' that relates the counting of atypical constituents in a general unitarizable
$\slmn$-supermodule to the character-valued $Q$-Witten index, expressed as a supertrace over the
full supermodule. The formal superdimension of holomorphic discrete series $\slmn$-supermodules
can also be formulated in this framework.
\end{abstract}

\maketitle

\setcounter{tocdepth}{1}

\tableofcontents

\setcounter{section}{-1}

\newpage

\hfill\parbox[b]{9.6cm}{{\it
Mathematicians are kind of like Frenchmen: when you talk to them, they translate it into their 
language and then it is soon something rather different.} (Goethe)}

\section{Introduction}

\subsection{Elementary ideas}
\label{sec::elements}

Let $\bigl(\HH,\langle\cdot,\cdot\rangle\big)$ be a separable Hilbert space, $(-1)^F\in\LL(\HH)$ a 
self-adjoint involution, and $Q\in \LL(\HH)$ a closed and densely defined operator with $(-1)^F Q = 
- Q (-1)^F$ and $Q^2=0$. Letting $Q^\dagger$ denote the adjoint of $Q$ with respect to $\langle\cdot,
\cdot\rangle$, the operators $Q^\dagger Q$ and $Q Q^\dagger$, and \emph{a fortiori} $\Xi:=Q^\dagger Q + 
Q Q^\dagger$ are self-adjoint and non-negative (von Neumann's Theorem), and $[\Xi,Q]=0$. Under the 
assumption that $e^{-\beta \Xi}$ is trace class for some $\beta>0$, the \emph{Witten index} of $Q$ 
(also known as the ``$Q$-Witten index'') is 
defined as \cite{WittenMorse}
\begin{equation}
\label{definedas}
I^W_\HH(Q;\beta) 
:= \tr_{\HH} (-1)^F e^{-\beta (Q Q^\dagger + Q^\dagger Q)}
= \tr_{\HH} (-1)^F e^{-\beta \Xi}.
\end{equation}
The ``chiral halves'' of $Q$, defined by $Q_\pm := \onehalf (Q+Q^\dagger) \bigl(1\pm(-1)^F\bigr)$ 
verify $Q_\pm^2=0$, $Q_-=Q_+^\dagger$, and $Q_+Q_-+Q_-Q_+=\Xi$. If, when viewed as operators from 
$\HH_\pm:=\onehalf (1\pm (-1)^F)\HH$ to $\HH_\mp$, respectively, they satisfy the 
Fredholm property,  that is, $\Im(Q_\pm)$ are closed, and both $\Ker Q_+$ and $\Ker Q_-= \Coker(Q_+)$ are 
finite-dimensional, the Witten index is in fact independent of $\beta$, which is thence omitted from the 
notation, and
\begin{equation}
\label{groundstate}
I^W_\HH(Q) = \dim \Ker(Q_+) - \dim \Ker(Q_-) = \tr_{\ker Q\cap\Ker(Q^\dagger)} (-1)^F.
\end{equation}
The Witten index is also stable under a certain class of deformations of the data $(Q,\Xi)$ for
fixed $(\HH,(-1)^F)$. In particular, under the above assumptions, 
\begin{equation}
\label{invariance}
I^W_\HH(Q+\delta Q)=I_\HH^W(Q) 
\end{equation}
if $\delta Q\in\LL(\HH)$ is \emph{relatively $Q$-compact} in the sense that $\delta Q (Q-i)^{-1}$ 
is compact \cite{CFKS,GesztesySimon}. In fact, $I^W$ satisfies all desirable properties of an index.

A rich supply of such data is expected to be provided by constructions of supersymmetric quantum field
theories (QFT) \cite{iaslectures}. In this context, $(-1)^F$ is the ``fermion number'' and $Q$ any 
``nilpotent complex supercharge''. The operator $\Xi$ can be viewed as ``energy'', although in general 
it differs from the physical Hamiltonian by generators of internal, so-called R-symmetries (which 
themselves commute with the Hamiltonian). According to \eqref{groundstate}, the Witten index counts 
the number of bosonic, minus the number of fermionic, ``ground states'' (states of zero energy, also
referred to as ``BPS states''). The ``topological invariance'' \eqref{invariance} says that 
this number is independent of continuous physical parameters, such as gauge couplings, that enter the 
putative mathematical definition of the QFT. This was the original motivation for the definition of the 
Witten index \cite{WittenIndex}.

Over the years, such indices have been of great interest in mathematical physics to explore and test 
predictions concerning properties of quantum field theories with \emph{superconformal symmetry} (SCFTs) 
\cite{Minwalla,romelsberger2006counting}. In this context, $\bigl(\HH,\langle\cdot,\cdot\rangle\bigr)$ 
is the ``super Hilbert space in radial quantization'' equipped with its unitary representation of the
superconformal algebra (SCA) of super Minkowski space, and $Q$ any nilpotent odd element of its 
complexification. The (in general, infinitely degenerate!) ground states organize in representations of 
the centralizer of $Q$,
as captured by the ``grand canonical'' refinement of the standard $Q$-Witten index \eqref{groundstate}. 
Non-vanishing contributions arise only from so-called ``short'' irreducible representations of the SCA, 
and one can show that this contains all protected information that can be obtained without 
invoking any dynamical assumptions. Moreover, the ``operator-state correspondence'' paired with topological 
invariance allow for a semi-rigorous path-integral evaluation of this ``superconformal index'' in many 
interesting examples via {\it supersymmetric localization}. This has led to a plethora of insights into SCFTs, 
their dualities, and related fields of mathematics \cite{RastelliReview,Gadde:2020yah}.

This paper grew out of attempts to explain the representation theoretic underpinnings of the superconformal
index (namely, the content of the preceding two paragraphs) in a mathematically intelligible language, as 
a way to exchange insight and extract lessons for possible applications in other areas. The natural context
for this are unitarizable supermodules over the special linear Lie superalgebras $\slmn$ with respect to 
real forms $\supqn$. This does not yet cover all superconformal algebras but shares all essential features, 
including ``atypical highest weight supermodules'' as mathematical counterparts to ``short supermultiplets'' 
and the delicate dependence of unitarizability on the weights of the two-dimensional center of the maximal 
compact subalgebra (which take the role of conformal dimension and superconformal R-charge). 

In order to provide precise translations for all these concepts, we have found it useful to review
various aspects of the representation theory of Lie superalgebras, which is otherwise covered in many 
excellent review articles (see, \emph{e.g.,} \cite{musson2012lie,serganova2017representations}), if not quite 
in the selection that we need here.\footnote{The expert or busy reader might wish to glance at section 
\ref{sec::leitfaden} and skip from there directly to section \ref{sec::anindex}, which contains our main 
results.} The most prominent insight from physics that did not seem to have an adequate 
mathematical counterpart is the ``continuity of the fragmentation/recombination process at the boundary of the 
region of unitarity''. While we do not claim that our ad hoc approach to this issue is entirely satisfying, we 
hope that it justifies the interest in the definition of indices for Lie superalgebras in general. Such 
indices come in two flavors, one topological, the other analytic, and one the natural questions, which we
establish in the special case, is the equivalence between the two. In the other direction (from mathematics 
to physics), we uncover a relation between the Witten index (of analytic flavor) and the notion of
formal superdimension that was studied recently by one of us \cite{SchmidtGeneralizedSuperdimension}.

\subsection{Basic conventions} 
\label{sec::basic}

In this article, $\ZZ_{\ge 0}$ denotes the set of non-negative integers, and $\ZZ_{+}$, the positive.
$\ZZ/2:= \ZZ/2\ZZ$ is the ring of integers modulo $2$. The $\ZZ/2$ grading of various construction of 
supermathematics is often left implicit. When needed, the induced sign (or ``parity'') and 
sign rule are symbolized ruthlessly as ``exponentiated Fermion number''. On any super vector space 
$V=V_0\oplus V_1$, $(-1)^F:= \id_{V_0} \oplus (-\id_{V_1})$. A formula containing terms such as $(-1)^{vw}$ is
parsed by first restricting to homogeneous $v, w\in V$, replacing them in the exponent by their grading, 
and then extending linearly. For example, if $A=A_0\oplus A_1$ is an (associative) superalgebra, the
definition of the bracket that turns $A$ into a Lie superalgebra\footnote{Although one should, we
will not derive any meaning from the position occupied by the prefix ``super'' with respect to other
epithets. A Lie superalgebra is the same thing as a super Lie algebra.} reads $[a,b] = ab -
(-1)^{ab} ba$.
Morphisms of vector spaces and modules, (bi-)linear forms, etc., are implicitly assumed to be even, 
\emph{i.e.}, respect the grading. The exception are endomorphisms of a super vector space, which are 
taken to be the ``internal hom'', and viewed as a Lie superalgebra. 
The ground field is $\CC$ unless otherwise stated. Real structures are usually conceived as
``conjugate-linear {\it graded} anti-automorphisms'' \cite{deligne1999classical}, except in the
context of Hermitian forms, where we give preference to positivity. See section \ref{unitarizable} for
details and verify that one may mediate between the two conventions by the ``amendment''
(here and throughout, $i=\sqrt{-1}=e^{\pi i/2}$)
\begin{equation}
\label{amendment}
i^a := \begin{cases} 
1 & \text{ if $a$ is even} 
\\
i & \text{ if $a$ is odd}
\end{cases}
\quad
\Longrightarrow 
\quad
i^{ab} = (-1)^{ab} i^{a} i^b.
\end{equation}
We have no need to internalize signs of homological algebra. 

If $\gg=\gg_0\oplus \gg_1$ is a Lie superalgebra, a (left) $\gg$-supermodule is a super vector space $M$ 
with a graded (left) action of $\gg$ that induces a morphism of Lie superalgebras. In other words, $[X,Y]m= 
X(Ym)-(-1)^{XY} Y(Xm)$ for all $X,Y \in \gg$ and $m \in M$. We will reserve the notation $\rho:\gg\to \End(M)$ 
for the structure map when needed for clarity. 
A morphism $f:M \to N$ of $\gg$-supermodules is a linear map such that $f(M_{i})\subset N_{i}$, and 
$f(Xm)=Xf(m)$ for all $X\in \gg$ and $m\in M$. 
We write $\gsmod$ for the category of all (left) $\gg$-supermodules. This is a $\CC$-linear 
abelian category. It is equipped with an endofunctor $\Pi$, the \emph{parity reversing functor}, defined 
by $\Pi(M)_{0}=M_{1}$, $\Pi(M)_{1}=M_{0}$, and viewed as a $\gg$-supermodule via $X\cdot m := (-1)^X Xm$ 
for any $X \in \gg$ and $m \in M$. Note that a $\gg$-supermodule $M$ is 
not necessarily isomorphic to $\Pi(M)$. 
If $\gmod$ denotes the category of ordinary (even) (left) $\even$-modules, and we consider $\even$ as a 
purely even Lie superalgebra, the category $\evsmod$ is a direct sum of two copies of $\gmod$, since any 
$\even$-module can be viewed as a $\even$-supermodule concentrated in either one of two single degrees. 
The ordinary $\even$-module that is obtained from a $\gg$-supermodule $M$ by restriction and forgetting 
the grading is denoted by $M_{\ev}$.
The universal enveloping algebra of $\gg$ is denoted by $\UE(\gg)$, the universal enveloping algebra of $\even$ 
by $\UE(\even)$, and their centers by $\ZC(\gg)$ and $\ZC(\even)$, respectively.
In the context of real Lie superalgebras, Lie groups, or physics, we (tend to) use the notion of unitary 
(irreducible) representations synonymously with unitarizable (simple) supermodules over the complexification, 
see section \ref{unitarizable} for details. In a similar vein, elements of the odd part of a Lie superalgebra 
are sometimes addressed as ``supercharges''.

\subsection{Standing assumptions}
\label{subsec::assumptions} 

The thrust of the paper being representation theoretic, it seems worthwhile to clarify early on that the main
interest is \emph{analytic}, although we will not develop all technical details here. Generally speaking, 
the Hilbert space $\bigl(\HH,\langle\cdot,\cdot\rangle\bigr)$ of a physical theory provides a module over its symmetry 
(super)algebra, $\gg$, that is unitarizable with respect to the relevant real structure, but, quite importantly, 
\emph{not at all irreducible}, 
lest the dynamics be rather trivial. Instead, the composition of $\HH$ from simple $\gg$-modules is, on
the one hand, \emph{constrained} by further physical requirements (such as locality) and, on the other, \emph{varies
continuously} as a function of external parameters. 

To formalize the latter notion, we think of a topological space $\mathds{H}$ of ``coupling constants'' and of
a ``family of physical theories'' as providing (among other things) a continuous map
$T : \mathds{H} \to \Hom(\gg^\RR, \uu(\HH))$ into the space of unitarizable $\gg$-supermodule structures on a 
fixed $\ZZ/2$-graded Hilbert space $\HH$, which is identified with the representations of the underlying real
Lie algebra $\gg^\RR$ in the space $\mathfrak{u}(\HH)$ of skew-Hermitian operators on $\HH$.
Equivalently, we can think geometrically of the datum of a bundle 
of unitarizable $\gg$-supermodules over $\mathds{H}$ whose underlying Hilbert space bundle is trivial.

Concretely, with $\gg=\sl(4\vert 4)$, $\gg^\RR=\su(2,2\vert 4)$, we may think of $\tau\in\mathds{H}$, the complex upper 
half-plane, as representing the complexified gauge coupling of maximally supersymmetric Yang-Mills theory with 
fixed compact gauge group (e.g., ${\SU}(N)$ for some $N\in\NN$), and of $T(\tau)$ as being defined by the supersymmetric 
path-integral over fields on a suitable four-dimensional Lorentzian manifold with appropriate boundary conditions.
Here unitarity and continuity are guaranteed by standard, but non-rigorous(!), physical arguments.

From our perspective, the main issue in this approach is that the space $\Hom(\gg^\RR, \mathfrak{u}(\HH))$ 
can be endowed with various topologies, depending on the choice of topology for $\uu(\HH)$. A natural topology 
on $\uu(\HH)$ is not known, and it is not entirely clear which topology is actually preferred by physics. 
To circumvent this problem (in an admittedly rather ad hoc fashion), we add the technical assumption that the 
unitarizable supermodules $T(\tau)\in \Hom(\gg^\RR,\uu(\HH))$ decompose as direct sums of simple $\gg$-supermodules,
as follows for example from a discrete spectrum condition. This condition is, in particular, expected to be 
satisfied in ``radial quantization'' of super Yang-Mills theory or more pertinently when the four-manifold 
in question has compact spatial slices. This is natural from the physics point of view.%
\footnote{The direct sum might be (countably)
infinite, necessitating a precise definition of the \emph{Hilbert space direct sum}. Given a family 
$(\mathcal{H}_{i})_{i \in I}$ of Hilbert spaces indexed by a countable set $I$, we define the Hilbert 
space direct sum as
\begin{equation}
\bigoplus_{i \in I} \mathcal{H}_{i} := \biggl\{(v_{i})_{i\in I} \in \prod_{i 
\in I}\mathcal{H}_{i} \bigg\vert \sum_{i \in I} \langle v_{i},v_{i}\rangle < \infty \biggr\},
\end{equation}
which forms a Hilbert space with inner product
\begin{equation}
\langle (v_{i})_{i\in I},(w_{i})_{i\in I}\rangle := \sum_{i \in I}\langle v_{i},w_{i}\rangle_{\mathcal{H}_{i}}.
\end{equation}
Each $\mathcal{H}_{j}$ embeds as the subspace $\{(v_{i})_{i \in I} \mid v_{i} = 0 \ \text{for all $i \neq j$}\}$.}

As recounted in section \ref{sec::HW}, the unitarizable simple $\gg$-supermodules in turn decompose as direct sums of 
weight spaces, joint eigenspaces of the Cartan subalgebra  $\hh$ of $\gg$. This yields a map 
$\Theta:\Hom(\gg^\RR,\uu(\HH))\to(\hh^*)^{\ZZ_+}/S_\infty$, 
where $(\hh^*)^{\ZZ_{+}}$ consists of all sequences $(\lambda_{1}, \lambda_{2}, \dots)$ with $\lambda_{i} \in \hh^*$, 
$S_\infty$ is the infinite symmetric group, and $(\hh^*)^{\ZZ_+}/S_\infty$ is equipped with the quotient topology.
Our assumption is that the topology on $\Hom(\gg^\RR,\uu(\HH))$ is fine enough to allow $\Theta$ to be continuous.
This will guarantee that the indices we construct in section \ref{sec::anindex} are continuous functions on
the space of unitarizable $\gg$-supermodules, and can hence be used as invariants as intended, assuming of course that physics
is strong enough to make $T$ continuous. In other words, we use $(\hh^*)^{\ZZ_+}/S_\infty$ as a (topological) 
surrogate for $\Hom(\gg^\RR,\uu(\HH))$.

\subsection{Simple example} 
\label{sec::SQM}

We have found it instructive to illustrate the main ideas and terminology in the simplest non-trivial example. The
reader is invited to work through the following formulas, and reflect on the final result \eqref{evenbetter}, 
peeking ahead for precise definitions as necessary.

The symmetry algebra of superconformal quantum mechanics 
\cite{Fubini:1984hf,okazaki2015superconformal} is $\su(1,1\vert 1)$. This is a real Lie superalgebra 
consisting of complex $(2\vert 1)\times (2\vert 1)$ supermatrices with 
vanishing supertrace that are anti-self-adjoint with respect to an appropriately conceived indefinite 
Hermitian inner product on $\CC^{2\vert 1}$. To cut to the chase, 
let's let, for any
\begin{equation}
X= \begin{pmatrix} A & \Psi \\ \bar{\Psi} & D \end{pmatrix}\in \sl(2\vert 1;\CC)\,,
\qquad
\omega(X) := 
\eta^{-1}
\begin{pmatrix} A^\dagger & i \bar{\Psi}^\dagger \\ i \Psi^\dagger &  D^\dagger \end{pmatrix}
\eta \,,
\end{equation}
where $\cdot^\dagger$ is conjugate transpose, and $\eta$ symmetric even real of signature $(1,1\vert
1)$. Then, $\omega(X Y ) = (-1)^{XY} \omega(Y) \omega(X)$ $\forall X,Y\in \sl(2\vert 1;\CC)$, hence
$\omega ([X,Y]) = - [\omega(X),\omega(Y)]$ with respect to the superbracket $[\cdot,\cdot]$, and
$\su(1,1\vert 1) := \bigl\{ X \,|\, \omega(X) = - X \bigr\}$ with the induced superbracket.
$\su(1,1\vert 1)$ has real dimension $(4\vert 4)$, and choosing $\eta=\left(\begin{smallmatrix} 0 &
1 & 0 \\ 1 & 0 & 0 \\ 0 & 0 & 1 \end{smallmatrix}\right)$ for convenience, a basis given by
\begin{equation}
\label{matrices}
\begin{split}
D = 
\begin{pmatrix} 
1 & 0 & 0  \\ 0 & -1 & 0  \\ 0 & 0 & 0 
\end{pmatrix}
,\;
K = 
\begin{pmatrix}
 0 & i & 0 \\ 0 & 0 & 0 \\ 0 & 0 & 0 
\end{pmatrix}
,\;
P = 
\begin{pmatrix}
 0 & 0 & 0 \\ -i & 0 & 0 \\ 0 & 0 & 0 
\end{pmatrix}
,\;
R = 
\begin{pmatrix}
i & 0 & 0 \\ 0 & i & 0 \\ 0 & 0 & 2i
\end{pmatrix}
\\
Q_1 = 
\begin{pmatrix}
0 & 0 & 0 \\ 0 & 0 & 1 \\ -i & 0 & 0
\end{pmatrix}
,\;
Q_2 = 
\begin{pmatrix}
0 & 0 & 0 \\ 0 & 0 & -i \\ 1 & 0 & 0
\end{pmatrix}
,\;
S_1 = 
\begin{pmatrix}
0 & 0 & -1 \\ 0 & 0 & 0 \\ 0 & i & 0
\end{pmatrix}
,\;
S_2 = 
\begin{pmatrix}
0 & 0 & i \\ 0 & 0 & 0 \\ 0 & -1 & 0
\end{pmatrix}
\end{split}
\end{equation}
with non-trivial superbrackets
\begin{equation}
\label{superbrackets}
\begin{gathered}
\relax 
[D,K] = 2K,\; [D,P]= - 2P,\; [D,Q_{1,2}] = - Q_{1,2},\; [D,S_{1,2}] = S_{1,2},\;
\\
[R,Q_1] = Q_2 ,\; [R,Q_2] = - Q_1 ,\; [R,S_1] = S_2 ,\; [R,S_2] = -S_1,
\\
[K,P] = D 
,\; 
[Q_i,Q_j] = 2\delta_{ij} P
,\;
[S_i,S_j] = -2 \delta_{ij} K,
\\
[Q_1,S_1] = [Q_2,S_2] = R
,\;
[Q_1,S_2] = - [Q_2,S_1] = D.
\end{gathered}
\end{equation}
By way of abstract physics terminology (which we will eschew for most of the paper), $P$ is the
``usual'' (Schr\"odinger) Hamiltonian, generating translations in (world-line) time. $D$ is the
dilation operator, and $K$ the ``special conformal generator''. $Q_{1,2}$ are the ``real
supercharges'' of $N=2$ supersymmetric quantum mechanics, of scaling dimension $-1$, and $S_{1,2}$
the ``superconformal generators'', of scaling dimension $+1$. $R$ is the R-charge that transforms
the pairs $(Q_1,Q_2)$ and $(S_1,S_2)$ like an infinitesimal counterclockwise rotation a positively
oriented basis of the euclidean plane. 

Concretely, one is interested in representations of \emph{the complexification} $\su(1,1\vert
1)\otimes \CC \cong \sl(2\vert 1;\CC)$ that are unitary in the sense that all $X\in\su(1,1\vert 1)$
are anti-self-adjoint operators on a super Hilbert space in the same conception as on $\CC^{2\vert
1}$. To analyze such representations, one may start from the bosonic conformal subalgebra that is
spanned by $D,K,P$ and isomorphic to $\su(1,1) \cong \so(2,1;\RR)$, with a focus on the ``conformal
Hamiltonian'', $K-P$. It generates a non-contractible compact one-parameter subgroup inside of $
{\SU}(1,1)$, and hence a non-compact subgroup of the universal cover $\widetilde {\SU}(1,1)$. By
``Wick rotation'', $\tilde D=i(K-P)$, together with $\tilde K = \onehalf(K+P+i D)$ and $\tilde P =
\onehalf(K+P - i D)$ span a copy of the 1-dimensional ``Euclidean'' conformal algebra
$\so(1,2;\RR)$, coincidentally isomorphic to $\so(2,1;\RR)$. Inside the complexification, $E=\tilde
K$, $H=\tilde D$ and $F=\tilde P$ form a standard $\sl(2;\CC)$ triple, under which any desired
unitary representation of positive energy decomposes into highest weight representations that are
unitarizable with respect to $\omega(H)=H$, $\omega(E)=-\omega(F)$. This was explained and utilized
for example by Mack \cite{mack1977all} in the closely related 4-dimensional situation. 
This can be extended to the supercharges:
\begin{equation}
    \tilde Q_1 = \sqrttwo (Q_1-iS_2)
    ,\;
    \tilde Q_2 = \sqrttwo (Q_2 + i S_1)
    ,\;
    \tilde S_1 = \sqrttwo (S_1 + i Q_2)
    ,\;
    \tilde S_2 = \sqrttwo (S_2 - i Q_1)
\end{equation}
together with $\tilde D,\tilde K, \tilde P$, and $\tilde R= R$ span a subalgebra of $\sl(2\vert 1;\CC)$ 
again isomorphic to $\su(1,1\vert 1)$. The complex combinations
\begin{equation}
\begin{split}
 Q & = \sqrttwo (\tilde Q_1 - i \tilde Q_2) = \onehalf ( Q_1-i Q_2 + S_1 - i S_2 ),
\\
\bar Q &= \sqrttwo (\tilde Q_1 + i \tilde Q_2) = \onehalf (Q_1+iQ_2 - S_1 - i S_2),
\\
S & = \sqrttwo (\tilde S_2 + i \tilde S_1) = \onehalf( S_2 + i S_1 - Q_2 - i Q_1 ) ,
\\
\bar S &= \sqrttwo (-\tilde S_2 + i \tilde S_1) = \onehalf (-S_2+i S_1-Q_2 +i Q_1)
 \end{split}
\end{equation}
together with $J=-iR$, satisfy the standard $\sl(2\vert 1)$ relations
\begin{equation}
\label{standardRelations}
\begin{gathered}
\relax
[H,Q] = -Q,\; [H,\bar Q ] = -\bar Q,\; [H,S]= S,\; [H,\bar S]=\bar S,\;
    \\
    [J,Q]= Q,\; [J,\bar Q] = - \bar Q ,\;
    [J,S]= S,\; [J,\bar S] = - \bar S,
    \\
[Q,Q] = [S,S] = [\bar Q,\bar Q] = [ \bar S,\bar S] =  [Q,S] = [\bar Q,\bar S] = 0,
\\
[Q,\bar Q ] = 2F,\; [S,\bar S] = 2E,\;
 [Q,\bar S] = -H - J ,\;
 [S,\bar Q] = H - J,
\end{gathered}
\end{equation}
and the conditions for unitarizability ``satisfied by the Euclidean generators'' are 
\begin{equation}
    \label{crucialiter}
    \omega(J) = J,\; \omega(Q) = i\bar S,\; \omega(\bar Q)=-i S.
\end{equation}

A fundamental example of such a unitary representation is the \emph{oscillator supermodule},
$(\Oss,\langle \cdot, \cdot\rangle)$, also known as ``singleton supermultiplet''. Its construction
depends on the isomorphism $\spo(2\vert 2) \cong \sl(2\vert 1)$. Algebraically, $\Oss$ is the
natural simple supermodule over the Weyl-Clifford algebra $\WCl(W)$ attached to a $(2\vert
2)$-dimensional supersymplectic vector space $W$. The Hermitian form $\langle\cdot,\cdot\rangle$ is
such that the original $\su(1,1\vert 1)$ ``Minkowski signature'' subalgebra from \eqref{matrices},
\eqref{superbrackets} is unitary. Completion is implicit.

Specifically, let $V=\CC^{1\vert 1}$, and $W=V\oplus V^*$. The formula
$\Omega(a,b):=f(w)-(-1)^{ab}g(v)$ for all homogeneous $a=(v,f)$ and $b=(w,g)$ in $W$ defines an even
supersymplectic form $\Omega: W\times W\to \CC$. This is anti-symmetric on the even, and
symmetric on the odd part of $W$. The natural action of $W$ on the supersymmetric algebra $\S(V)$,
defined via $a=(v,f)\mapsto v \cdot + \iota_f \in \End(\S(V))$, satisfies $[a,b] = \Omega(a,b)$ with
respect to the canonical superbracket on $\End(\S(V))$. The image of $W$ in $\End(\S(V))$ generates the
Weyl-Clifford algebra $\WCl(W)$, and it is immediate to see that $\S(V)$ is a simple module over it.
The oscillator supermodule $\Oss$ is $\S(V)$ equipped with the restriction of $\WCl(W)$ to the
orthosymplectic Lie superalgebra $\spo(2\vert 2)\cong \S^2(W) \subset \End(\S(V))$.
With respect to a homogeneous basis $(x,\eta)$ of $V=\CC^{1\vert 1}$, we can think of $\sl(2\vert
1)$ as the Lie superalgebra spanned by the following set of super-differential operators acting on
$\CC[x,\eta]$
\begin{equation}
\label{supervectorfields}
\begin{gathered}
E = - \onehalf \partial_x^2,\; H = -x\partial_x -\onehalf,\;
F= \onehalf x^2 ,\; 
J = \eta\partial_\eta - \onehalf,
\\
Q = \eta x,\;
\bar Q = \partial_\eta x ,\;
S = - \eta\partial_x,\;
\bar S =  \partial_\eta\partial_x,
\end{gathered}    
\end{equation}
which one easily verifies satisfy the above superbrackets.

As an $\sl(2\vert 1)$ supermodule, $\Oss$ decomposes into two simple supermodules $\Oss
= \Oss^+\oplus \Oss^-$, according to the parity of the total polynomial degree,
\begin{equation}
    \Oss^+ = \CC[x^2, x\eta] ,\qquad 
    \Oss^- = \eta \CC[x^2] + x \CC[x^2].
\end{equation}
Both are unitarizable highest weight modules, which will serve to illustrate the general theory in 
later sections, see in particular Figs.\ \ref{fig:rootsystem} and \ref{fig:Example}. The Hermitian form 
on the supermodule $\Oss\cong \CC[x,\eta]$ is the sesqui\-linear extension of the formulas
\begin{equation}
\langle x^n,x^m\rangle = 
\begin{cases} n! & \text{if $n=m$} 
\\
0 & \text{otherwise}
\end{cases}
\qquad
\langle \eta,\eta\rangle = 1 \,,
\qquad 
\langle x^n,\eta\rangle =0.
\end{equation}
The adjoint map $a\mapsto a^\dagger$ defined with respect to $\langle\cdot,\cdot\rangle$ satisfies 
$(ab)^\dagger = b^\dagger a^\dagger$, independent of the parity of $a,b\in\End(\CC[x,\eta])$, 
with $x^\dagger = \partial_x$, $\eta^\dagger = \partial_\eta$. The ``amendment''
\begin{equation}
    \sigma(a) = 
    \begin{cases} 
    a^\dagger & \text{if $a$ is even,}
    \\
    i a^\dagger & \text{if $a$ is odd},
    \end{cases}
\end{equation}
(which is the superadjoint with respect to super Hermitian form $\psi(v,w) = i^v \langle
v,w\rangle$, see subsection \ref{unitarizable}) satisfies $\sigma(ab) = (-1)^{ab}\sigma(b)\sigma(a)$
for all $a,b\in \WCl(V)$, and is compatible with the conditions \eqref{crucialiter} on the vector
fields \eqref{supervectorfields}. Crucially,
\begin{equation}
\label{cruxy}
    Q^\dagger = \bar S ,\qquad 
    S^\dagger = - \bar Q
\end{equation}
and
\begin{equation}
[Q,Q^\dagger] = -H - J = x \partial_x - \eta\partial_\eta + 1
,\qquad 
[S,S^\dagger] = -H + J = x\partial_x + \eta\partial_\eta.
\end{equation}
These expressions are indeed non-negative operators on $\Oss$. The Witten indices \eqref{definedas}
of $Q$ and $S$ in the two simple supermodules $\Oss_{\pm}$ can be calculated algebraically.
\begin{equation}
\label{mainresult}
\begin{split}
I^W_{\Oss^+}(Q) &= \tr_{\Oss^+} (-1)^F e^{-\beta [Q,Q^\dagger]} = 0,
\\
I^W_{\Oss^-}(Q) &= \tr_{\Oss^-} (-1)^F e^{-\beta [Q,Q^\dagger]} = -1, 
\\
I^W_{\Oss^+}(S) &= \tr_{\Oss^+} (-1)^F e^{-\beta [S,S^\dagger]} = +1,
\\
I^W_{\Oss^-}(S) &= \tr_{\Oss^-} (-1)^F e^{-\beta [S,S^\dagger]} = 0.
\end{split}
\end{equation}
They can be viewed as a special case of the analytic result for the general square-zero supercharge
\begin{equation}
Q_{(r,t)} = r Q - t S = \eta (r x + t \partial_x) 
\end{equation}
as a function of $[r:t]\in \CC\PP^1$. Its adjoint is
\begin{equation}
Q_{(r,t)}^\dagger = \bar r\bar S + \bar t \bar Q = \partial_\eta (\bar r\partial_x + \bar t x)
\end{equation}
and
\begin{equation}
\Xi_{(r,t)} = [Q_{(r,t)},Q_{(r,t)}^\dagger] 
=
|r|^2 (J-H) - |t|^2 (J+H) + 2 r \bar t F -2 \bar r t E.
\end{equation}
Their formal joint kernel 
$\{f \in \CC[[x,\eta]] \ : \ Q_{(r,t)}f = Q_{(r,t)}^\dagger f = 0\} $ is two-dimensional, and spanned 
by
\begin{equation}
f_+ = \exp\Bigl(-\frac rt \frac{x^2}{2}\Bigr)\,,
\qquad
f_- = \eta \exp\Bigl(-\frac {\bar t}{\bar r} \frac{x^2}{2}\Bigr)\,.
\end{equation}
Now
\begin{equation}
\langle f_+,f_+\rangle = 
\sum_{n=0}^\infty \Bigl|\frac r{2t}\Bigr|^{2n} \frac{(2n)!}{(n!)^2}
= \Bigl(1-\Bigl|\frac rt\Bigr|^2\Bigr)^{-1/2}
\end{equation}
converges, \emph{i.e.}, $f_+\in \Oss^+$, iff $\bigl|\frac rt\bigr|<1$ (and complementarily for $f_-$).
As a result,
\begin{equation}
\label{evenbetter}
\begin{split}
I^W_{\Oss^+} (Q_{(r,t)}) = 
\begin{cases} 
+1 & \text{if $|r|<|t|$,}
\\
0 & \text{otherwise,}
\end{cases} 
\\
I^W_{\Oss^-} (Q_{(r,t)}) = 
\begin{cases} 
-1 & \text{if $|t|<|r|$,}
\\
0 & \text{otherwise.}
\end{cases} 
\end{split}
\end{equation}
We see in this example that the index is in fact not continuous as a function on the space of nilpotent 
supercharges. This observation does not seem to have been recorded in the literature so far, and will be
interesting to explore further.

\subsection{Leitfaden}
\label{sec::leitfaden}

We start in section \ref{sec:unitarizable} with a detailed review of special linear Lie superalgebras 
and their unitarizable supermodules. This includes the basic structure theory of $\glmn$ and 
$\gg:=\slmn=\even\oplus\odd$, the real forms $\supqn$, the construction of highest weight modules in $\gsmod$ via 
Kac induction from $\gmod$, and the relationship between atypicality of (highest) weights
and the degeneracy of the Kac-Shapovalov form. A central r\^ole is reserved for the two abelian
factors of the maximal compact subalgebra, which generalize what is referred to as ``dimension'', 
$\Delta$, and ``R-charge'', $r$, in the physics literature.

In section \ref{sec::redrecomb}, we present the extension of the ``algebraic Dirac operator'' to 
Lie superalgebras following \cite{SchmidtDirac} as a useful tool to characterize unitarizable 
$\gg$-supermodules, and to understand their decomposition under $\even$. We then give a geometric 
description of the ``region of unitarity'' in slices of weight space continuously parameterized by 
dimension and R-charge. We give a precise and complete description of the fragmentation/recombination 
phenomenon that occurs at the boundary of this region, in terms of the decomposition of Kac modules 
of atypical highest weight. We finally explain the relationship with the notions of ``protected/short 
supermultiplets'' that famously pervade the physics literature. The central statement here is that maximally 
atypical supermodules are absolutely protected.

In section \ref{SectionDS} attention turns to the Duflo-Serganova (DS) functor.\footnote{really, "foncteur
d\'eesse''} We give a lightning review of the self-commuting variety (also known as ``nilpotence variety'' 
in the physics literature), and its r\^ole in relating supermodules over Lie superalgebras of different 
rank in a cohomological fashion. A novel aspect of our discussion is the compatibility of the DS functor
with unitarity, wich is usually glossed over in the literature.

Section \ref{sec::formal_superdimension} aims for a brief (p)review of the formal superdimension introduced 
in \cite{SchmidtGeneralizedSuperdimension}. This notion generalizes the ordinary superdimension of 
finite-dimensional $\gg$-supermodules by assembling the formal (Harish-Chandra) degree of its 
$\even$-constituents, viewed as (relative) discrete series representations of $\SU(p,q)\times \SU(n)\times 
U(1)$ (or its universal cover). The main result is that the superdimension is non-zero only for maximally 
atypical unitarizable supermodules.

Section \ref{sec::anindex} contains our main results. Fist, following Kinney–Maldacena–Minwalla–Raju 
\cite{kinney2007index} we define the notion of a ``counting index'' as an additive and continuous function 
on the space of unitarizable $\gg$-supermodule structures (\emph{cf.} section \ref{subsec::assumptions}).
Second, we elaborate on the definition of the $Q$-Witten index for unitarizable $\gg$-supermodules as a 
``supertrace valued in supercharacters of the Duflo-Serganova twist of $\gg$ with respect to 
$Q+Q^\dagger$''. We then establish the $\RR$-linear equivalence between the KKMR and $Q$-Witten
index. Finally, we interpret the formal superdimensional as a (real) KMMR index, and establish the
relation between its calculation in terms of Harish-Chandra characters and the $Q$-Witten index.

\subsubsection{Acknowledgments} We extend special thanks to Rainer Weissauer for inquiring about 
the precise mathematical interpretation of the superconformal index, and for numerous conversations about 
Lie superalgebras. We thank Johanna Bimmermann, Fabian Hahner, Simone Noja, Ingmar Saberi, and in particular 
Raphael Senghaas for discussions and collaboration on related projects.
S.S.\ thanks the organizers of the 2024 HCM workshop ``Representations of supergroups'' for the opportunity to
present some of this work, and the Instituto Nacional de Matemática Pura e Aplicada (IMPA) for its hospitality 
during the final stages of this work.
J.W.\ thanks the International Centre for Mathematical Sciences, Edinburgh, for support and hospitality during 
the ICMS Visiting Fellows programme where much of this work was done. This work was supported by EPSRC grant 
EP/V521905/1.
This work is funded by the Deutsche Forschungsgemeinschaft (DFG, German Research Foundation) under
project number 517493862 (Homological Algebra of Supersymmetry: Locality, Unitary, Duality). 
This work is funded by the Deutsche Forschungsgemeinschaft (DFG, German Research Foundation) under
Germany’s Excellence Strategy EXC 2181/1 — 390900948 (the Heidelberg STRUCTURES Excellence Cluster).

\section{Unitarizable Supermodules over \texorpdfstring{$\slmn$}{sl(m|n)}}
\label{sec:unitarizable}

\subsection{Structure theory of \texorpdfstring{$\slmn$}{sl(m|n)}} 
\label{sec::structure_theory}

For given $m,n\in \ZZ_+$, the \emph{general linear Lie superalgebra} $\glmn$ is the complex vector space of 
block matrices 
\begin{equation}
X = \left(\begin{array}{@{}c|c@{}}
  A
  & \Psi \\
\hline
  \bar{\Psi} &
  D \tstrut
\end{array}\right),
\end{equation}
where $A$ is an $m\times m$ matrix, $\Psi$ an $m\times n$ matrix, $\bar{\Psi}$ an $n\times m$ matrix 
completely independent of $\Psi$, and $D$ an $n\times n$ matrix. The $\ZZ/2$-grading is 
$\mathfrak{gl}(m\vert n)=\glmn_{0}\oplus \glmn_{1}$ where $\mathfrak{gl}(m\vert n)_{0}$ consists of all 
``even'' block matrices with $\Psi=\bar{\Psi}=0$ and $\glmn_{1}$ consists of all ``odd'' block matrices 
with $A=D=0$. The bracket is the (canonical) super matrix commutator $[X,Y] := X\cdot Y- (-1)^{XY}Y\cdot
X$, where $X,Y \in \glmn$ are homogeneous elements, and $"\cdot"$ denotes usual matrix multiplication, 
extended by linearity (see also section \ref{sec::basic} for conventions). The \emph{supertrace} of
$X\in \glmn$ is $\str(X) := \tr(A)-\tr(D)\in\CC$. The \textit{special linear Lie superalgebra} 
$\slmn$ is the sub Lie superalgebra given by all matrices in $\glmn$ with vanishing supertrace. This is a 
codimension-$1$ ideal in $\glmn$. We write $\gg$ for $\slmn$ in the following. Its even part $\even$ is 
naturally isomorphic to a direct sum of complex special linear Lie algebras and an abelian factor, 
$\even \cong \sl(m)\oplus \sl(n)\oplus \CC E_{m\vert n}$ where $E_{m\vert n}$ denotes the even matrix with 
$A = n E_m$ and $D = m E_n$ in the representation above, and where $E_m$, $E_n$ are unit matrices of the
size indicated. When $m\neq n$ and $m+n \geq 2$, $\slmn$ is simple, and the extension to $\glmn$ splits (naturally, but 
non-canonically) via $\CC\ni 1 \mapsto E_{m+n}\in \glmn$. When $m=n$, we have $E_{2n} = \frac 1n E_{n|n}\in
\sl(n\vert n)$. As a consequence, $\gl(n\vert n)$ does not split, and $\sl(n\vert n)$ becomes reducible, 
although it also does not split. The codimension-1 simple quotient is the \emph{projective special linear Lie 
superalgebra} $\psl(n\vert n)=\sl(n\vert n)/\CC E_{2n}$.

We denote by $\dd := \{ H = \diag(h_1,\ldots,h_{m+n})\}$ the abelian Lie subalgebra of diagonal matrices 
in $\glmn$. We choose the subspace of diagonal matrices with vanishing supertrace, denoted by $\hh$, as 
Cartan subalgebra of $\gg=\slmn$. The standard basis of the dual space $\dd^\ast$ of $\dd$ is
$(\epsilon_{1},\dotsc,\epsilon_{m},\delta_{1},\ldots,\delta_{n})$ where $\epsilon_{i}(H)=h_{i},\
\delta_{k}(H)=h_{k+m}$  for $i=1,\dotsc,m$ and $k=1,\dotsc,n$ and for any $H\in\dd$. We (ab)use this basis 
also for the dual space $\hh^\ast$ of $\hh$. Namely, we identify weights $\lambda \in \hh^{\ast}$ for $\gg$ 
with tuples $(\lambda^{1},\dotsc,\lambda^{m}\vert \mu^{1},\dotsc,\mu^{n})$ via the expansion
$\lambda=\lambda^{1}\epsilon_{1}+\dotsc+\lambda^{m}\epsilon_{m}+\mu^{1}\delta_{1}+\dotsc+\mu^{n}\delta_{n}$,
keeping in mind that shifts by multiples of $(1,\ldots,1\vert {-}1,\ldots,-1)$ do not change the weight. The space 
of weights for $\psl(n\vert n)$ is the subquotient of tuples with $\sum_{i=1}^n (\lambda^i + \mu^i)=0$.

The Cartan subalgebra $\hh$ acts semisimply on $\even$, and $\odd$ is a completely reducible
$\even$-module under the adjoint action induced by the super matrix commutator $[\cdot,\cdot]$. This 
gives a $\ZZ/2$-graded root space decomposition of $\gg$,
\begin{equation}
\gg = \hh \oplus \bigoplus_{\alpha \in \hh^{\ast}\setminus\{0\}} \gg^{\alpha}, 
\qquad \gg^{\alpha}:=\{ X \in \gg \ : \ [H,X] = \alpha(H)X \ \text{for all} \ H \in \hh\}.
\end{equation}
The set of roots is $\Phi=\Phi_{0}\sqcup \Phi_{1}$ where
\begin{equation}
\label{allroots}
\begin{split}
\Phi_{0} &= \{\pm(\epsilon_{i}-\epsilon_{j}),\pm(\delta_{k}-\delta_{l}) \ : \ 1\leq i<j\leq m, \ 1\leq k <l \leq n\}, \\
\Phi_{1} &= \{\pm(\epsilon_{i}-\delta_{k}) \ : \ 1\leq i \leq m, \ 1\leq k\leq n\},
\end{split}
\end{equation}
are the \emph{even} and \emph{odd roots}, respectively. Note that each root space has superdimension
$(1\vert 0)$ or $(0\vert 1)$, and $\Phi_{0}$ is the disjoint union of root systems for $\sl(m)$ and
$\sl(n)$. On $\Phi_{0}$, we fix for the rest of the article the standard positive system 
\begin{equation}
\label{evenpositive}
\Phi_{0}^{+}:=\{\epsilon_{i}-\epsilon_{j}, \delta_{k}-\delta_{l} \ : \ 1 \leq i<j\leq m, \ 1\leq k <l \leq n\}
\end{equation}
such that the root vectors for $\epsilon_{i}-\epsilon_{j}$, $i<j$, are strictly upper triangular
matrices of $\sl(m)$, and the root vectors for $\delta_{k}-\delta_{l}$, $k<l$, are strictly upper
triangular matrices of $\sl(n)$, both diagonally embedded in $\even$. When extending this to the
odd part, we usually make the \emph{standard choice}
\begin{equation}
\label{oddpositive}
\Phi_{1}^{+} := \{\epsilon_{i}-\delta_{k} \ : \ 1\leq i \leq m, \ 1\leq k\leq n\}
\end{equation}    
such that the associated root vectors are the off-diagonal upper block matrices in $\odd$.
Then the set of positive roots is $\Phi^{+}:=\Phi_{0}^{+}\sqcup \Phi^{+}_{1}$. With respect to 
this choice of $\Phi^+$, the Lie superalgebra $\gg$ has the triangular decomposition
\begin{equation}
\gg = \nn^{-}\oplus \hh \oplus \nn^{+}, \qquad \nn^{\pm} := \bigoplus_{\alpha \in \Phi^{+}}\gg^{\pm
\alpha},
\end{equation}
where $\nn^{+}$ is the space of strictly upper triangular matrices and $\nn^{-}$ the 
space of strictly lower triangular matrices. The associated \emph{Borel subalgebra} is 
$\bb = \hh\oplus \nn^{+}$. The simple roots are 
\begin{equation}
\label{standardsimple}
\pi := (\epsilon_1-\epsilon_2,\ldots,\epsilon_{m-1}-\epsilon_m,\epsilon_m-\delta_1,
\delta_1-\delta_2,\ldots,\delta_{n-1}-\delta_n) \,.
\end{equation} 
The space of real weights, defined as the $\RR$-span of $\Phi$ inside of $\hh^*$, and denoted $\Phi_\RR$
or $\hh^*_\RR$, is a real vector space of dimension $n+m-1$. Its dual inside of $\hh$, denoted 
$\hh_\RR$, is the space of supertrace-less real diagonal matrices.

The general linear Lie superalgebra carries a natural \emph{even supersymmetric} and \emph{invariant} 
bilinear form $(\cdot,\cdot): \glmn \times \glmn \longrightarrow \CC$ defined by
\begin{equation}
\label{killing}
(X,Y):=\str(XY)
\end{equation}
for all $X,Y\in \glmn$. Here, even supersymmetric means that $(\cdot,\cdot)$ is symmetric on $\glmn_0$, 
skew-symmetric on $\glmn_1$, and $\glmn_0$ and $\glmn_1$ are orthogonal to each other. Invariant means 
that $([X,Y],Z)=(X,[Y,Z])$ for all $X,Y,Z\in \glmn$. The form $(\cdot,\cdot)$ is always non-degenerate
on $\glmn$. However, its restriction to $\slmn$ is non-degenerate only as long as $m\neq n$. When $m=n$, 
the one-dimensional center of $\sl(n\vert n)$ becomes the non-trivial radical of $(\cdot,\cdot)$.

The restriction of $(\cdot,\cdot)$ to $\dd$ is still non-degenerate, and identifies the subspace $\dd_\RR$ 
of real diagonal matrices with the pseudo-orthogonal space $\RR^{m,n}$. The bilinear form induced on 
the dual space is denoted by the same symbol. On the standard basis, we have for all $1\leq i,j \leq m$ and 
$ 1 \leq k,l \leq n$
\begin{align}
(\epsilon_{i},\epsilon_{j})=\delta_{ij}, \qquad
(\delta_{k},\delta_{l})=-\delta_{kl},\qquad
(\epsilon_{i},\delta_{k})=0.
\end{align}
When $n\neq m$, the restriction of $(\cdot,\cdot)$ to $\hh_\RR$ is non-degenerate of signature $(m-1,n)$ or 
$(m,n-1)$, depending on whether $m>n$ or $m<n$. We can then associate to any root $\alpha\in\Phi$ a unique
element $h_{\alpha}\in\hh_\RR$ through the condition $\alpha(H)=(H,h_{\alpha})$ for all $H \in \hh$. The 
element $h_{\alpha}$ is called the \emph{dual root} associated to $\alpha$. When $n=m$, we have 
$\hh_\RR/\RR E_{2n} \cong \RR^{m-1,n-1}$. We can fix the dual roots as elements of $\hh_\RR$ by requiring 
$\alpha(H) = (H,h_\alpha)$ for all $H\in\dd$. However, the bilinear form on $\hh^*_\RR$, which is equal to
the linear extension of $(\alpha,\beta) := (h_{\alpha},h_{\beta})$ for $\alpha,\beta\in\Phi$, remains 
non-degenerate only when $n\neq m$.
It follows from the above that the odd roots are all isotropic, \emph{i.e.}, 
$(\alpha,\alpha)=0$ for all $\alpha\in\Phi_1$.

The set of roots $\Phi$ carries a natural action of the Weyl group $W$ of $\even$. $W$ is isomorphic to the 
direct product of symmetric groups $S_{m}\times S_{n}$ on $m$ and $n$ letters, respectively. It is
generated by reflections in the even roots,
\begin{equation}
\label{evenreflection}
R_\alpha(\beta) := \beta - 2\frac{(\alpha,\beta)}{(\alpha,\alpha)} \alpha \,,\qquad \alpha\in \Phi_0
\,, \beta\in\Phi.
\end{equation}
This action can be extended linearly to $\hh^*_\RR$ (and to $\hh^*$), where it leaves the bilinear form
$(\cdot,\cdot)$ invariant. The \emph{Weyl group of $\gg$} is defined to be the Weyl group of $\even$. The \emph{Weyl 
vector} $\rho$ is $\rho=\rho_{0}-\rho_{1}$ where
\begin{equation}
\rho_{0} := \frac{1}{2}\sum_{\alpha \in \Phi_{0}^{+}}\alpha, \qquad \rho_{1} :=
\frac{1}{2}\sum_{\alpha \in \Phi_{1}^{+}}\alpha.
\end{equation}
Using the parameterization \eqref{evenpositive} and \eqref{oddpositive}, we find
\begin{equation}
\rho = \frac{1}{2}\biggl( \sum_{i=1}^{m}(m-n+1-2i)\epsilon_{i}+\sum_{k=1}^{n}(m+n-2k+1)\delta_{k}\biggr)
\end{equation}
In calculations, however, it is simpler to use a shifted representative of the Weyl vector,
\begin{multline}
\rho = \frac 12 (m-n-1,\ldots,-m-n-1 \vert m+n-1,\ldots,m-n-1)
+ \frac{m+n+1}{2}(1,\ldots,1\vert -1,\ldots,-1)\\ 
=(m,\ldots,2,1\vert {-}1,-2,\ldots,-n).
\qquad\qquad\qquad\qquad\qquad\qquad\qquad\qquad\qquad\qquad\qquad
\end{multline}
The \emph{dot action} on weight vectors, defined for $w\in W$ and $\lambda\in\hh^*$ by
$w\cdot \lambda = w(\lambda+\rho)-\rho$ as well as constructs such as $(\lambda+\rho,\epsilon_i-\delta_k)$, 
which will be important below, are of course independent of the representative by construction. 

In physics applications, it is sometimes convenient to work with different positive systems, which are
called \emph{non-standard} in the mathematical literature, see \emph{e.g.}, \cite{jakobsen1994full}. Generally 
speaking, different positive systems are related by a sequence of Weyl reflections \eqref{evenreflection} 
in the even roots, together with so-called \emph{odd reflections}, which are defined at the level of
simple roots by
\begin{equation}
\label{oddreflection}
R_\alpha(\beta) := \begin{cases} 
\beta + \alpha & \text{if $(\beta,\alpha)\neq 0$}
\\
\beta & \text{if $(\beta,\alpha) = 0$, $\beta\neq \alpha$}
\\
-\alpha & \text{if $\beta = \alpha.$}
\end{cases}
\,,\qquad \alpha\in\pi\cap\Phi_1\,,\beta\in \pi
\end{equation}
We will explain this for one such non-standard system, which is well-suited for the study of unitary 
representations of $\supqn$, at the end of the next subsection.

\subsection{Real forms \texorpdfstring{$\supqn$}{su(p,q|n)}}
\label{sec::realforms}

The even subalgebra of $\gg$ is $\even \cong \sl(m)\oplus \sl(n)\oplus \CC E_{m\vert n}$. We will be interested in 
those real forms of $\gg$ whose even subalgebras are the real forms $\su(p,q)\oplus \su(n) \oplus \u(1)$ of $\even$, 
where $p+q=m$. These are the real Lie superalgebras $\supqn$.
 
In general, real forms of $\gg$ are in one-to-one correspondence with conjugate-linear anti-involutions of $\gg$
\cite{ParkerRealForms, SerganovaRealForms}.
Here, an \emph{anti-involution} $\omega$ is an involution of the underlying super vector space such that 
$\omega([X,Y])=-[\omega(X),\omega(Y)]$ for all $X,Y \in \gg$. The real form $\gg^{\omega}$ of $\gg$ with respect to $\omega$ is 
then $\gg^{\omega} := \{X \in \gg : \omega(X) = -X\}$.
To identify real forms of $\gg$ with even part $\su(p,q)\oplus \su(n) \oplus \u(1)$, we consider the $\ZZ/2$-compatible 
$\ZZ$-grading $\gg =\gg_{-1}\oplus \gg_{0}\oplus \gg_{+1}$ with $\gg_{1} = \gg_{-1}\oplus
\gg_{+1}$, where
\begin{equation}
\label{intgrading}
\gg_{-1}:=\bigoplus_{\alpha \in \Phi_{1}^{+}}\gg^{-\alpha}=\left\{\left(\begin{array}{@{}c|c@{}}
  0
  & 0 \\
\hline
  \bar{\Psi} &
  0 \tstrut
\end{array}\right) \ : \ \bar{\Psi} \in \CC^{mn}\right\}, \ 
\gg_{+1}:=\bigoplus_{\alpha \in \Phi_{1}^{+}}\gg^{\alpha}=\left\{\left(\begin{array}{@{}c|c@{}}
  0
  & \Psi \\
\hline
  0 &
  0 \tstrut
\end{array}\right) \ : \ \Psi \in \CC^{mn}\right\}.
\end{equation}
Assuming temporarily that $p,q$ are both non-zero, we fix a maximal compact subalgebra
\begin{equation}
\label{maxcompact}
\kk=(\mathfrak{su}(p)\oplus \mathfrak{su}(q)\oplus \mathfrak{u}(1)) \oplus \mathfrak{su}(n) \oplus \u(1),
\end{equation}
diagonally embedded in $\gg$. Here, $\su(p)\oplus \su(q)\oplus \u(1)$ is a maximal compact subalgebra 
of $\su(p,q)$, and the final $\u(1)$ is the same as before. Then $\gg_{\pm 1}$ are $\kk$-modules under 
the action coming from the adjoint action of $\even$ on $\odd$. They each decompose in two simple 
$\kk$-modules,
\begin{equation}
\gg_{-1}=\bar{\mathfrak{q}}\oplus \bar{\mathfrak{s}}, \qquad
\gg_{+1}=\mathfrak{s} \oplus \mathfrak{q}\,.
\end{equation}
The constituents are tensor products of the defining representations of $\su(p)$, $\su(q)$, $\su(n)$
and their duals, and carry $\u(1)$ actions that we normalize momentarily. Writing a general element 
$X \in \gg$ as 
\begin{equation}
X=\left(\begin{array}{@{}cc|c@{}}
a & b & S 
\\ 
c & d & Q
\\
\hline
\bar{Q} & \bar{S}  & D 
\tstrut
\end{array}\right)
\end{equation}
we have $S \in \mathfrak{s}, Q \in \mathfrak{q}, \bar{Q} \in \bar{\qq},  
\bar{S} \in \bar{\mathfrak{s}}$, $D \in \mathfrak{su}(n)^{\CC}$
and $\begin{pmatrix} a & b \\ c & d \end{pmatrix}\in \mathfrak{su}(p,q)^{\CC}$. 
It is then immediate that there are exactly two conjugate-linear 
anti-involutions of $\gg$ compatible with the real forms $\su(p,q)\oplus \su(n) \oplus \u(1)$ of 
$\even$, namely 
\begin{equation}
\label{involutions}
\begin{split}
\omega_{+}\left(\begin{array}{@{}cc|c@{}}
a & b & S 
\\ 
c & d & Q
\\
\hline
\bar{Q} & \bar{S} & D 
\tstrut
\end{array}\right)
&=
\left(\begin{array}{@{}c@{\,\,}c@{\,}|c@{}}
a^{\dagger} & -c^{\dagger} & i\bar{Q}^{\dagger} 
\\ 
-b^{\dagger} & d^{\dagger} & -i\bar{S}^{\dagger}
\\
\hline
iS^{\dagger} & -iQ^{\dagger} &  D^{\dagger}
\tstrut
\end{array}\right), 
\\[.1cm]
\omega_{-}\left(\begin{array}{@{}cc|c@{}}
a & b & S 
\\ 
c & d & Q
\\
\hline
\bar{Q} & \bar{S} & D
\tstrut
\end{array}\right)
&=
\left(\begin{array}{@{}c@{\,\,}c@{\,}|c@{}}
a^{\dagger} & -c^{\dagger} & -i\bar{Q}^{\dagger} 
\\ 
-b^{\dagger} & d^{\dagger} & i\bar{S}^{\dagger}
\\
\hline
-iS^{\dagger} & iQ^{\dagger} &  D^{\dagger}
\tstrut
\end{array}\right), 
\end{split}
\end{equation}
where $\cdot^{\dagger}$ denotes complex conjugate transpose. It may be observed that the two 
possibilities can be mapped into each other by either Galois symmetry $i\mapsto -i$ or by
changing the ``relative signature'' between the even and odd part of the underlying super vector
space. This corresponds to the distinction between $\su(p,q\vert 0,n)$ and $\su(p,q\vert n,0)$,
which is however somewhat formal because it can also be brought about by conjugation with the 
overall grading operator (Fermion number $(-1)^F$), above and beyond the obvious isomorphism 
$\su(p,q\vert 0,n)\cong
\su(q,p\vert n,0)$. In our discussion of unitary representations, we will align the choice with 
our conventions for super Hermitian forms (see subsection \ref{unitarizable}). This will restrict 
the discussion to $\omega_-$, for which real form we will write $\mathfrak{su}(p,q\vert n)$ 
instead of $\mathfrak{su}(p,q\vert n,0)$.
The maximal compact subalgebra \eqref{maxcompact} of $\su(p,q\vert n)_{0}$ 
is referred to as a maximal compact subalgebra of $\su(p,q\vert n)$. The Cartan subalgebra of $\kk$ 
is the set of supertrace-less real diagonal matrices, which we denoted $\hh_\RR$ above.

If either $p=0$ or $q=0$, the maximal compact subalgebra is $\kk=\su(m)\oplus\su(n)\oplus\u(1)$ and the 
$\kk$-modules $\gg_{\pm 1}$ are simple. Still, there are two conjugate-linear anti-involutions
\begin{equation}
\omega_{\pm}\left(\begin{array}{@{}c|c@{}} 
A & \Psi \\
\hline
\bar{\Psi} & D 
\tstrut
\end{array}\right)
=
\left(\begin{array}{@{}c|c@{}} 
A^{\dagger} & \pm i\bar{\Psi}^{\dagger} \\
\hline
\pm i\Psi^{\dagger} & D^{\dagger} 
\tstrut
\end{array}\right)
\,, 
\qquad 
\left(\begin{array}{@{}c|c@{}} A &
\Psi \\
\hline
  \bar{\Psi} & D \tstrut
  \end{array}
  \right)\in \gg.
\end{equation}
In either case, we define as advertised
\begin{equation}
    \supqn := \{ X \in \gg\ : \ \omega_{-}(X) = -X\}.
\end{equation}

The root system $\Phi_{c} := \Phi(\kk^{\CC}; \hh)$ of $\kk^\CC$ with respect to the Cartan subalgebra inherited from 
$\even$ (and ultimately $\gg$) is a subset of $\Phi_0$ (and $\Phi$), naturally called the \emph{compact roots}. 
Elements of its complement $\Phi_{n} := \Phi_0 \setminus \Phi_{c}$ are called \emph{non-compact roots}. We 
denote $\Phi_{c}^{+} := \Phi_{c} \cap \Phi_0^{+}$ and $\Phi_{n}^{+} := \Phi_{n} \cap \Phi_0^{+}$. The ``compact'' 
Weyl vector is $\rho_c = \frac 12 \sum_{\alpha\in\Phi_c^+} \alpha$.

Of particular interest for the unitary representation theory of $\supqn$ are the abelian factors in the
maximal compact subalgebra. This is well-known also from the physics applications. A generator, 
$J$, of the $\u(1)$ that appears in the direct sum decomposition of $\supqn_{0}$ 
commutes with all other even generators, and is called \emph{superconformal R-charge}. The weights 
(eigenvalues) of $J$ in $\gg$-supermodules are known as \emph{R-charge}, and will be symbolized as $r$. 
When $n\neq m$, it is natural to normalize $J$ such that the R-charge induces the $\ZZ$-grading in eq.\ 
\eqref{intgrading}, \emph{i.e.}, to set
\begin{equation}
\label{inducesgrading}
J := \frac{1}{n-m} \left(\begin{array}{@{}c|c@{}} 
nE_{m} & 0 \\
\hline
 0 &
  mE_{n}
\end{array}\right).
\end{equation}
When $n=m$, the superconformal R-charge is central in $\sl(n\vert n)$, while the $\ZZ$-grading \eqref{intgrading} 
is not inner with respect to 
$\sl(n\vert n)$. We will simply use $J= E_{n+n}$ in this case.\footnote{The $\ZZ$-grading is still 
inner with respect to $\gl(n\vert n)$, generated for example by $\Sigma=\onehalf 
\begin{pmatrix} E_n & 0 \\ 0 & -E_n \end{pmatrix}$. Since however 
$\str\Sigma\neq 0$, $\Sigma \notin [\gl(n\vert n), \gl(n\vert n)]$ and does not participate in the 
representation theory at a level comparable to the superconformal R-charge. In contradistinction,
$J=E_{2n} \in [\sl(n\vert n), \sl(n\vert n)]$. As remarked before, $\sl(n\vert n)$, while not simple, is 
a non-trivial central extension.} 
The second $\u(1)$ appears when $p,q$ are both non-zero after the choice of a maximal compact subalgebra of
$\su(p,q)$. Following physics terminology, its generator, $H$, will be called \emph{conformal Hamiltonian}. Its
eigenvalues are known as \emph{(conformal) dimension} and will be symbolized by $\Delta$. It is natural to 
normalize $H$ such that the non-compact even positive roots have dimension $2$, \emph{i.e.}, to set
\begin{equation}
\label{ietoset}
H := \frac{2}{m} 
\left(\begin{array}{@{}cc|c@{}}
qE_{p} & 0 & 0 
\\ 
0 & -pE_{q} & 0 \\ 
\hline
0 & 0 & 0 
\end{array}\right)\,.
\end{equation}
These normalizations of $J$ and $H$ agree for $(p,q\vert n) = (1,1\vert 1)$ with the relations 
\eqref{standardRelations} from the introduction. The supercharges $Q$ and $\bar Q$ have dimension $-1$, the 
superconformal generators $S$ and $\bar S$ dimension $+1$. This is also true in other physically
interesting cases such as $\su(2,2\vert n)$, but not in general. On the occasion, we illustrate in
Fig.\ \ref{fig:rootsystem} the position of the root system of $\gg=\sl(2\vert 1)$ in the $(\Delta,r)$-plane, 
together with the weights of the two simple modules discussed in section \ref{sec::SQM}. The left panel 
includes the Weyl vector and the highest weights for the standard positive system \eqref{oddpositive}.
The simple roots are \eqref{standardsimple} $\pi= (\epsilon_1-\epsilon_2,\epsilon_2-\delta_1)$ with 
corresponding root vectors $E$ and $Q$.
The right panel is for the non-standard system with simple roots $(\epsilon_1-\delta_1,\delta_1-\epsilon_2)$,
corresponding to root vectors $S$ and $\bar S$. It can be obtained from the standard system by an odd 
reflection \eqref{oddreflection} in $\alpha = \epsilon_2-\delta_1$.
Similar positive systems with only odd simple roots exist in particular for all $\sl(n\vert n)$. Their 
advantage is that highest weight vector also maximizes the dimension in any highest weight supermodule.
This facilitates the identification of the ground states, and hence the evaluation of the superconformal 
index. This will be discussed at length in the following subsections, but may be instructive to verify
in the example.

\begin{figure}[ht]
\subfloat[Standard positive system]
{
\begin{tikzpicture}
    [
    x=1cm, y=1cm,  scale=1.2,
    font=\footnotesize,
    >=latex
    ]
\draw[->] (-2.2,0) -- (2.5,0) ;
\foreach \i in {-2,...,-1}
\draw (\i,0.25mm) -- ++ (0pt,-0.5mm) node[below] {$\i$};
\foreach \i in {1,...,2}
\draw (\i,0.25mm) -- ++ (0pt,-0.5mm) node[below] {$\i$};
\draw[->] (0,-2.2) -- (0,2.2);
\foreach \i in {-2,...,-1}
\draw   (0.025,\i) -- ++ (-0.5mm,0pt) node[left] {$\i$};
\foreach \i in {1,...,2}
\draw   (0.025,\i) -- ++ (-0.5mm,0pt) node[left] {$\i$};
\draw   (-1,1)  node[left] {$Q$};
\draw   (-1,-1)  node[left] {$\bar Q$};
\draw   (1,1)  node[right] {$S$};
\draw   (1,-1.05)  node[right] {$\bar S$};
\draw   (-2,0)  node[above] {${F}$};
\draw   (2,0)  node[above] {$E$};
\draw   (2.55,0)  node {\rlap{$\Delta=H$}};
\draw   (0,2.4)  node {$r=J$};
\fill[ForestGreen] (1,1) circle (0.6mm) (-1,1) circle (0.6mm) (2,0) circle (0.6mm);
\fill[red]  (-1,-1) circle (0.6mm) (1.05,-1.05) circle (0.6mm) (-2,0) circle (0.6mm);
\draw (-.5,-.47) node[left] {$1$} ;
\draw (-1.5,-.5) node[left] {$x$} ;
\draw (-2.5,-.45) node[left] {$x^2$} ;
\draw (-.5,.48) node[left] {$\eta$} ;
\draw (-1.5,.48) node[left] {$\eta x$} ;
\draw (-2.5,.53) node[left] {$\eta x^2$} ;
\draw[blue] (-.5,-.5) circle (0.4mm);
\draw[magenta] (-1.5,-.5) circle (0.4mm);
\draw[blue] (-2.5,-.5) circle (0.4mm);
\draw[magenta] (-.5,.5) circle (0.3mm);
\draw[magenta] (-.5,.5) circle (0.5mm);
\draw[blue] (-1.5,.5) circle (0.3mm);
\draw[blue] (-1.5,.5) circle (0.5mm);
\draw[magenta] (-2.5,.5) circle (0.4mm);
\draw (.95,-.85) node[left] {$\rho$} ;
\fill[black] (.95,-.95) circle (0.6mm);
\end{tikzpicture}
}
\qquad
\subfloat[Non-standard positive system]
{
\begin{tikzpicture}
    [
    x=1cm, y=1cm,  scale=1.2,
    font=\footnotesize,
    >=latex
    ]
\draw[->] (-2.2,0) -- (2.5,0) ;
\foreach \i in {-2,...,-1}
    \draw (\i,0.25mm) -- ++ (0pt,-0.5mm) node[below] {$\i$};
\foreach \i in {1,...,2}
    \draw (\i,0.25mm) -- ++ (0pt,-0.5mm) node[below] {$\i$};
\draw[->] (0,-2.2) -- (0,2.2);
\foreach \i in {-2,...,-1}
    \draw   (0.025,\i) -- ++ (-0.5mm,0pt) node[left] {$\i$};
\foreach \i in {1,...,2}
\draw   (0.025,\i) -- ++ (-0.5mm,0pt) node[left] {$\i$};
\draw   (-1,1)  node[left] {$Q$};
\draw   (-1,-1)  node[left] {$\bar Q$};
\draw   (1,1)  node[right] {$S$};
\draw   (1,-1)  node[right] {$\bar S$};
\draw   (-2,0)  node[above] {${F}$};
\draw   (2,0)  node[above] {$E$};
\draw   (2.55,0)  node {\rlap{$\Delta=H$}};
\draw   (0,2.4)  node {$r=J$};
\fill[ForestGreen]  (1,1) circle (0.6mm) (1,-1) circle (0.6mm) (2,0) circle (0.6mm);
\fill[red]  (-1,1) circle (0.6mm) (-1,-1) circle (0.6mm) (-2,0) circle (0.6mm);
\draw (-.5,-.47) node[left] {$1$} ;
\draw (-1.5,-.5) node[left] {$x$} ;
\draw (-2.5,-.45) node[left] {$x^2$} ;
\draw (-.5,.48) node[left] {$\eta$} ;
\draw (-1.5,.48) node[left] {$\eta x$} ;
\draw (-2.5,.53) node[left] {$\eta x^2$} ;
\draw[blue] (-.5,-.5) circle (0.3mm);
\draw[blue] (-.5,-.5) circle (0.5mm);
\draw[magenta] (-1.5,-.5) circle (0.4mm);
\draw[blue] (-2.5,-.5) circle (0.4mm);
\draw[magenta] (-.5,.5) circle (0.3mm);
\draw[magenta] (-.5,.5) circle (0.5mm);
\draw[blue] (-1.5,.5) circle (0.4mm);
\draw[magenta] (-2.5,.5) circle (0.4mm);
\draw (0,-.2) node[left] {$\rho$} ;
\fill[black] (0,0) circle (0.6mm);
    \end{tikzpicture}
\quad}
\caption{Root system of $\sl(2\vert 1)$ with respect to two canonical positive systems. The coordinates are 
the basis $(H,J)$ of the Cartan subalgebra used in the presentation in eq.\ \eqref{standardRelations}. The 
green dots represent the positive roots, and the red dots, the negative roots. The blue circles 
are (some of) the weights of the oscillator supermodule $\Oss^+$, the magenta circles, of $\Oss^-$. 
The double circles are the highest weights. The black dot is the Weyl vector.}
\label{fig:rootsystem}
\end{figure}
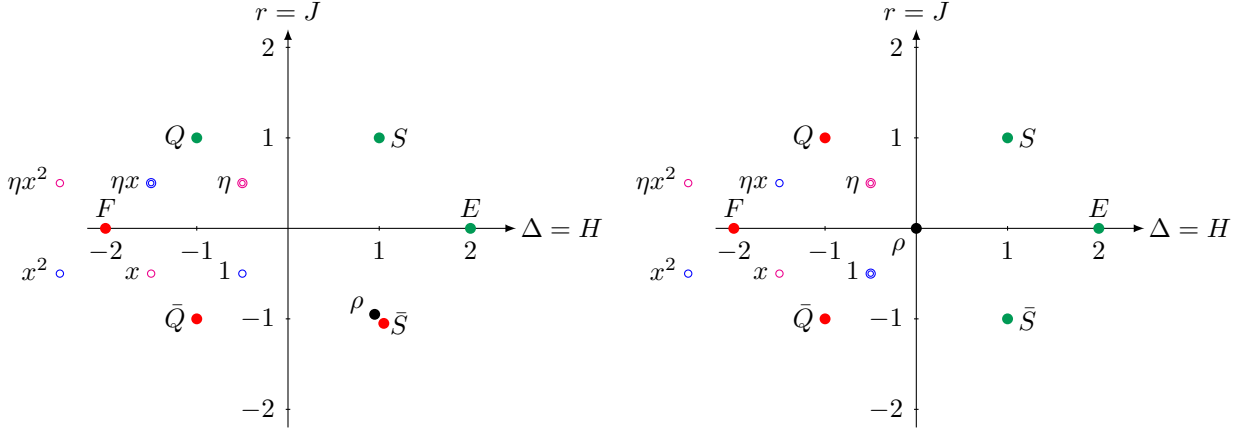

\subsection{Highest weight supermodules and atypicality} 
\label{sec::HW} 

Highest and lowest weight $\gg$-supermodules are a particularly important type of supermodules 
over $\gg$. On the one hand, any simple finite-dimensional $\gg$-supermodule is both a highest and 
lowest weight $\gg$-supermodule. On the other hand, as reviewed in detail in section \ref{unitarizable}, 
any unitarizable simple (in general, infinite-dimensional) supermodule is either a highest or a 
lowest weight $\gg$-supermodule.

\begin{definition} A $\gg$-supermodule $M$  is called \emph{highest (lowest) weight} $\gg$\emph{-supermodule
of highest (lowest) weight} $\Lambda \in \hh^{\ast}$ if there exists a nonzero $v_{\Lambda} \in M$ such that:
\begin{enumerate}
    \item[a)] $Xv_{\Lambda}=0$ for all $X \in \nn^{+}$ ($\nn^-$), 
    \item[b)] $Hv_{\Lambda}=\Lambda(H)v_{\Lambda}$ for all $H \in \hh$ and 
    \item[c)] $\mathfrak{U}(\gg)v_{\Lambda}=M$.
\end{enumerate}
We call $v_{\Lambda}$ \emph{highest (lowest) weight vector} of $M$. We call $M$ even/odd according
to the parity of $v_\Lambda$, denoted $(-1)^{v_\Lambda}=(-1)^M$ if necessary.
\end{definition}

The following discussion will be phrased in term of highest weight $\gg$-supermodules. All statements are
valid  \emph{mutatis mutandis} for lowest weight $\gg$-supermodules. We first record some elementary algebraic
properties that follow directly from the definition.

\begin{lemma}
\label{lemm::properties_HWM}
Let $M$ be a highest weight supermodule with highest weight $\Lambda$ and parity $(-1)^{v_\Lambda}$. 
Let $\alpha_{1}\ldots,\alpha_{k}$ and 
$\beta_{1},\ldots,\beta_{l}$ be an enumeration of the even and odd positive roots, respectively, where $k= 
\frac{m(m-1)+n(n-1)}{2}$ and $l= mn$. Choose root vectors $X_{i} \in 
\gg^{- \alpha_{i}}$ and $Y_{j} \in \gg^{-\beta_{j}}$. Then 
\begin{enumerate}
\item $M$ is spanned by the collection of vectors $X_{1}^{r_{1}}\cdots X_{k}^{r_{k}}Y_{1}^{s_{1}}\cdots 
Y_{l}^{s_{l}}v_{\Lambda}$ where $r_{1},...,r_{k} \in \ZZ_{+}$ and $s_{1},...,s_{l} \in \{0,1\}$, of weight
$\Lambda-\sum_{i=1}^{k}r_{i}\alpha_{i}-\sum_{j=1}^{l}s_{j}\beta_{j}$ and parity 
$(-1)^{\sum_{j=1}^l s_j + v_\Lambda}$.
\item $M$ is a \emph{weight supermodule}, that is, $M=\bigoplus_{\lambda \in \hh^{\ast}}M^{\lambda}$, where 
$M^{\lambda}:=\{m\in M : Hm=\lambda(H)m \text{ for all } H\in \hh\}$. The space $M^{\lambda}$ is 
the \emph{weight space} of weight $\lambda$, and $\dim(M^{\lambda})$ its \emph{multiplicity}. 
$\dim(M^{\lambda}) < \infty$ for all $\lambda$, with $\dim(M^{\Lambda})=1$.
\item Any nonzero quotient of $M$ is again a highest weight supermodule with highest weight $\Lambda$. 
$M$ has a unique maximal submodule and unique simple quotient. In particular, $M$ is indecomposable.
\item $M$ is a quotient of the (even or odd) \emph{Verma supermodule} $V(\Lambda)$, which 
is defined (up to parity) as the quotient of $\UE(\gg)$ by the left ideal generated by $\nn^+$ 
and $H-\Lambda(H)$ for all $H\in\hh$, and as a vector space isomorphic to $\UE(\nn^-)$ by 
Poincar\'e-Birkhoff-Witt.
\item All simple highest weight modules of fixed highest weight and parity are isomorphic.
\end{enumerate}
\end{lemma}

The peculiar consequences of the odd roots are brought to the fore by constructing 
the simple highest weight $\gg$-supermodule with highest weight $\Lambda$ not as a quotient of $V(\Lambda)$, but as
the (unique) simple top of a Kac supermodule. Kac supermodules can be described as generalized Verma 
supermodules with respect to $\even$ as Levi subalgebra. Concretely, one starts by extending any $M_0\in\gmod$, 
placed in either even or odd degree, trivially to a $\even\oplus \gg_{+1}$-supermodule. Then, the induced 
$\gg$-supermodule 
\begin{equation}
\label{kacinduct}
K(M_0) := \mathfrak{U}(\gg)\otimes_{\mathfrak{U}(\even\oplus \gg_{+1})}M_0,
\end{equation}
is called \emph{Kac supermodule}. The associated map $K(\cdot):\evsmod \longrightarrow \gsmod$ defines
an exact functor, the \emph{Kac induction functor}. In general, when $M_0$ is simple as a $\even$-module,
$K(M_0)$ is indecomposable as a $\gg$-supermodule, although it might not be simple. 

\begin{theorem}[{\cite[Theorem 4.1]{chen2021simple}}] 
\label{thm::simple_top} 
For any simple $\even$-supermodule $M_0$ the Kac supermodule $K(M_0)$ has a unique maximal proper sub supermodule. 
The unique simple top of $K(M_0)$ is denoted by $L(M_0)$. The map $M_0 \mapsto L(M_0)$ gives rise to a bijection 
between the set of isomorphism classes of simple $\even$-supermodules and the set of isomorphism classes of simple 
$\gg$-supermodules.\footnote{Equivalently, $L$ is a bijection from simple $\even$-modules to simple 
$\gg$-supermodules, up to parity reversal.} $L(M_0)$ is a simple highest weight $\gg$-supermodule iff 
$M_0$ is a simple highest weight $\even$-module, with the same highest weight.
\end{theorem}

We write $L_0(\Lambda)$ for the simple highest weight $\even$-module of highest weight $\Lambda\in\hh^*$,
which is unique up to isomorphism. We write $K(\Lambda)$ for the Kac supermodule $K(L_0(\Lambda))$,
and $L(\Lambda)$ for its simple quotient $L(L_0(\Lambda))$. The parity of $L(\Lambda)$ will
normally be clear from context, and otherwise is assumed to be \emph{even}.
We emphasize that in general $K(\Lambda)$ is not equal to the Verma supermodule of highest weight 
$\Lambda$, except when $L_0(\Lambda)$ is in fact the Verma module over $\even$. The size of the maximal
proper sub supermodule of $K(\Lambda)$ is encoded in the \emph{atypicality} of the highest weight, 
$\Lambda$, as we describe momentarily. Atypicality is a phenomenon only 
present in the theory of Lie superalgebras, which has no analog in the theory of complex semisimple 
Lie algebras.

In general, let $M$ be a highest weight $\gg$-supermodule with highest weight $\Lambda$. Any element
$z \in \ZC(\gg)$ in the center of the universal enveloping algebra acts on $M$ by a scalar 
$\chi_{\Lambda}(z)$. The map
$\chi_{\Lambda}:\ZC(\gg)\to \CC$ is an algebra homomorphism called the \emph{infinitesimal
character} of $M$. The infinitesimal character can be written explicitly as
\begin{equation}
\label{centralchar}
\chi_{\Lambda}(z):=(\Lambda+\rho)(\text{HC}(z)),
\end{equation}
where $\text{HC}:\ZC(\gg)\to \mathfrak{U}(\hh) \cong \S(\hh)$ is the Harish-Chandra
homomorphism, \emph{i.e.}, the restriction of the projection map
$\mathfrak{U}(\gg)=\mathfrak{U}(\hh)\oplus(\nn^{-}\mathfrak{U}(\gg) + \mathfrak{U}(\gg)\nn^{+})
\to \mathfrak{U}(\hh)$ to $\ZC(\gg)$ composed with the twist $\zeta: \S(\hh)\to \S(\hh)$ defined by 
$\lambda(\zeta(f)):=(\lambda-\rho)(f)$ for all $f\in\S(\hh)$, $\lambda\in\hh^*$. 
The Harish-Chandra homomorphism is an injective ring
homomorphism. To describe its image, we define 
$\S(\hh)^{W} := \{ f\in \S(\hh)\ :\ w (\lambda)(f) =
\lambda(f) \ \text{for all} \ w \in W, \lambda \in \hh^*\}$ and, for any $\lambda\in\hh^*$,
\begin{equation}
\label{Alambda}
A_{\lambda} := \{ \alpha \in \Phi_{1}^{+} : (\lambda+\rho,\alpha) = 0\} \,.
\end{equation}
Then, the image of $\HC$ is \cite{GorelikHC, KacHC}
\begin{equation}
\Im(\HC) = \bigl\{ f \in \S(\hh)^{W} : (\lambda+t\alpha)(f)=\lambda(f) \ 
\text{for all $t\in \CC$, $\lambda\in\hh^*$, and $\alpha \in A_{\lambda-\rho}$}\bigr\}.
\end{equation}
As a consequence, we obtain $\chi_{\Lambda}= \chi_{\Lambda'}$ whenever 
\begin{equation}
\Lambda'=w\biggl(\Lambda+\rho+\sum_{i=1}^{k}t_{i}\alpha_{i}\biggr)-\rho,
\end{equation}
where $w \in W$, $t_{i}\in \CC$, and $\alpha_{1},\dotsc,\alpha_{k}\in A_\Lambda$ are linearly independent 
odd isotropic roots that satisfy $(\Lambda+\rho,\alpha_{i}) = 0$. This leads to the definition of
typicality and atypicality.

\begin{definition}
\label{typatyp}
A weight $\Lambda \in \hh^{\ast}$ is called \emph{typical} if $A_\Lambda=\emptyset$, \emph{i.e.},
$(\Lambda+\rho,\alpha)\neq 0$ for
all $\alpha \in \Phi_{1}^{+}$. Otherwise $\Lambda$ is called \emph{atypical}. The
\emph{degree of atypicality} of $\Lambda$, denoted by $\at(\Lambda)$, is the maximal number of
linearly independent mutually orthogonal positive odd (in particular, isotropic) roots 
$\alpha \in \Phi_{1}^{+}$
such that $(\Lambda + \rho,\alpha) = 0$. In brief, $\at(\Lambda)$ is the dimension of a maximal
isotropic subspace of $\Span_\CC(A_{\Lambda})\subset \hh^{\ast}$. We call a highest weight 
$\gg$-supermodule $M$ with highest weight $\Lambda$ \emph{typical} if $\at(\Lambda) = 0$, and 
otherwise \emph{atypical}.
\end{definition}

We point out that the structure of highest weight supermodules over Lie superalgebras depends more
severely on the ordering of the roots than for ordinary complex Lie algebras. This will have been
illustrated by the example (see Fig.\ \ref{fig:rootsystem}). Still, any change of positive system
can be compensated by an appropriate action on the highest weight to yield an isomorphic 
$\gg$-supermodule. In particular, the degree of atypicality of a weight $\Lambda \in \hh^*$ is 
invariant under both even reflections, eq.\ \eqref{evenreflection} and odd reflections, 
eq.\ \eqref{oddreflection}, while (even the cardinality of) the set $A_\Lambda$ is not. As a 
consequence, the degree of atypicality is independent of the choice of positive root system. 

The \emph{defect} of $\gg$, denoted by $\defect(\gg)$, is the dimension of a maximal isotropic
subspace in the $\RR$-span of $\Phi$. Given $\gg=\slmn$, with $\Phi_\RR\cong\RR^{m,n}$, we have 
$\defect(\gg)=\min(m,n)$. The degree of atypicality satisfies $0\leq \at(\Lambda) \leq \defect(\gg)$ 
for any weight $\Lambda \in \hh^{\ast}$. A weight $\Lambda$ with $\at(\Lambda)=\defect(\gg)$ is
called \emph{maximally atypical}.

The physics literature knows various alternative quantifications of the reducibility of $K(\Lambda)$,
referred to as ``multiplet shortening'', ``BPS-ness'' or ``protectedness''. We will propose translations
of these concepts back into the mathematical language in due course.

\subsection{Unitarizability}
\label{unitarizable}

A \emph{super pre-Hilbert space} is a complex super vector space $\HH=\HH_{0} \oplus 
\mathcal{H}_{1}$ equipped with a positive definite even Hermitian form $\langle\cdot,\cdot\rangle$, that 
is, a complex valued sesquilinear form $\langle\cdot,\cdot\rangle:\HH\times \HH\to \CC$, which is 
conjugate-linear in the \textbf{first} and linear in the second variable, conjugate symmetric, such that 
$\HH_{0}$ and $\HH_{1}$ are mutually orthogonal subspaces, and $\langle v,v\rangle>0$ for all $v\neq 0$. 
The \emph{amendment} of $\langle \cdot,\cdot\rangle$, defined as
\begin{equation}
\psi: \HH\times \HH\to \CC, \qquad \psi(v,w) = i^v\langle v,w\rangle 
:=  \begin{cases}
i\langle v,w\rangle \ &\text{if} \ v,w \in \HH_{1}, \\
\langle v,w \rangle \ &\text{if} \ v,w \in \HH_{0}.
\end{cases}
\end{equation}
is an even \emph{super Hermitian form} on $\HH$, meaning that $\psi(v,w) =(-1)^{vw}\overline{\psi(w,v)}$ for 
all $v,w \in V$, which is \emph{super positive definite}, meaning that $\psi$ is positive definite on $\HH_{0}$ 
and $-i\psi$ is positive definite on $\HH_{1}$. Obviously, $\langle\cdot,\cdot\rangle$ and $\psi$ are
equivalent data, preference for which is dependent on taste and situation. Completion of $\HH$ to a 
$\ZZ/2$-graded Hilbert space is more natural with respect to $\langle\cdot,\cdot\rangle$. Given a 
(possibly only densely defined) endomorphism $T\in\End(\HH)$, the adjoint of $T$ is defined with respect 
to $\langle \cdot,\cdot\rangle$ by the formula
$\langle v, Tw\rangle=\langle T^\dagger v,w\rangle$
while the super-adjoint of $T$ is defined with respect to $\psi$ by
$\psi(v,Tw)=(-1)^{vT}\psi(\sigma(T) v,w)$. Both are involutive and, as before, related by the amendment
\begin{equation}
\label{amendagain}
\sigma(T) = i^T T^\dagger = \begin{cases}
T^{\dagger} \ & \text{if $T$ is even}, \\
iT^{\dagger} \ &\text{if $T$ is odd}.
\end{cases}
\end{equation}
The adjoint $\cdot^{\dagger}$ satisfies $(ST)^{\dagger}=T^{\dagger}S^{\dagger}$, independent of the
parity of $S,T \in \End(\HH)$. The super-adjoint $\sigma(\cdot)$ satisfies $\sigma(ST)=(-1)^{ST}\sigma(T)\sigma(S)$. 
Namely, $\cdot^{\dagger}$ defines a conjugate-linear anti-invo\-lu\-ti\-on on $\End(\HH)$, while $\sigma(\cdot)$ 
defines a conjugate-linear \emph{graded} anti-invo\-lu\-ti\-on. This makes $\sigma$ more natural for the definition of
unitarizability of supermodules over Lie superalgebras with respect to conjugate-linear anti-involutions.
See also sections \ref{sec::realforms} and \ref{sec::basic}.

\begin{definition} 
\label{unitarize}
Let $\omega$ be a conjugate-linear anti-involution of $\gg$, and let $(\rho,M)$ be a $\gg$-supermodule. Then $M$ is 
called $\omega$-\emph{unitarizable} if it is equipped with or admits a positive definite even Hermitian form 
$\langle\cdot,\cdot\rangle$ such that $\rho \circ \omega = \sigma \circ \rho$. 
\end{definition}

For clarity, the definition is formulated in terms of the $\ZZ/2$-graded representation $\rho: \gg \to \End(M)$. 
In what follows, however, we will continue to work with $\gg$-supermodules suppressing the explicit Lie superalgebra 
homomorphism $\rho$ (so as to avoid confusion with the Weyl vector). 
When $\langle\cdot,\cdot\rangle$ is specified or clear from context, we will prefer the symbol 
$\HH$ over $M$. We leave completion to a Hilbert space implicit as needed. Writing the defining condition as 
$\langle v, X w\rangle = i^X \langle \omega(X)v,w\rangle$ for all $v,w\in \HH$, $X\in\gg$, we will refer to 
$\langle\cdot,\cdot\rangle$ also as a contravariant Hermitian form. We will also think of the data as defining a 
``unitary representation'' of the real Lie superalgebra $\gg^\omega\to \gl(\HH)^\sigma$, although we will not be 
concerned with integrability to Lie supergroups. In the context of $\UE(\gg)$-supermodules, we extend $\omega$ 
to a graded anti-involution of $\UE(\gg)$ in the obvious way. When $\omega$ is to be implied from context, we 
just say ``unitarizable''.

\begin{definition}
\label{equivDef}
Two unitarizable $\gg$-supermodules $\HH_{1}$ and $\HH_{2}$ are called \emph{ equivalent} if 
there exists an isomorphism $f: \HH_{1} \to \HH_{2}$ in $\gsmod$ that is 
compatible with the Hermitian forms.
\end{definition}

Note that in the definition of equivalence of unitarizable $\gg$-supermodules, intertwining operators 
are assumed to preserve the grading. This means that in general a unitarizable $\gg$-supermodule is not 
necessarily equivalent to the one obtained by parity change.\footnote{We leave it as an exercise to 
verify that if $\psi$ is an invariant super Hermitian form on $\HH$, then $(-i)^{\Pi \cdot}
\psi(\Pi\cdot,\Pi\cdot)$ is an invariant super Hermitian form on $\Pi\HH$, while the associated 
contravariant Hermitian forms are the same.}

As mentioned in the introduction, constructions of superconformal field theories in theoretical physics are
expected to provide an interesting class of unitary representations of various superconformal algebras, restricted 
by a number of further conditions. In this context, an equivalence class of unitarizable simple $\gg$-supermodules 
is also referred to as a \emph{superconformal multiplet}.

\subsubsection{Implications of unitarity}
We record some basic properties of $\omega$-unitarizable $\gg$-super\-modules. The first result obtains for any 
Lie superalgebra and follows by a standard argumentation.

\begin{proposition} \label{CompletelyReducible} Let $\HH$ be an $\omega$-unitarizable $\gg$-supermodule. 
Then 
\begin{enumerate}
\item[a)] $\HH$ is completely reducible. 
\item[b)] $\HH_{\ev}$ (which we recall is the $\even$-module obtained by restriction and forgetting 
the $\ZZ/2$ grading) is a unitarizable $\even$-module with respect to the real form $\gg_{0}^{\omega}$. 
In particular, $\HH_{\ev}$ is completely reducible as a unitarizable $\even$-module.
\end{enumerate}
\end{proposition}

The second result, due to Neeb and Salmasian, concerns specifically $\gg=\slmn$ and justifies our
focus on the anti-involutions $\omega_-$ and $\omega_+$ given in eq.\ \eqref{involutions}.

\begin{theorem}[\cite{furutsu1991classification,neeb2011lie}] 
\label{involutionsOnly}
The special linear Lie superalgebra $\gg=\slmn$ admits non-trivial $\omega$-unitarizable supermodules 
if and only if the conjugate-linear anti-involution $\omega$ is associated to the real forms 
$\mathfrak{su}(p,q\vert n,0)$ or $\mathfrak{su}(p,q\vert 0,n)$ with $p+q=m$. Any simple such 
$\omega$-unit\-arizable supermodule is either a highest weight $\gg$-supermodule or a lowest weight 
$\gg$-supermodule.
\end{theorem}

Here, the notion of lowest/highest weight supermodule is defined with
respect to a positive system that is compatible with the $\su(p,q)\oplus \su(n)\oplus\uu(1)$ 
real subalgebra of $\even$ selected by $\omega$, such as the one we have fixed in section 
\ref{sec::structure_theory}. The physical significance of Theorem \ref{involutionsOnly} is 
that ``supersymmetry implies positive energy''. In the bosonic context, this has to be assumed
separately, \emph{cf.} \cite{mack1977all}.

We consider in the following only $\omega$-unitarizable \emph{highest} weight $\gg$-supermodules. The investigation 
of unitarizable lowest weight $\gg$-supermodules is similar, although it might require exchange of
$\omega_+$ with $\omega_-$. As remarked before, the real 
forms $\mathfrak{su}(p,q\vert 0,n)$ and $\mathfrak{su}(q,p\vert n,0)$ are isomorphic as real Lie 
superalgebras, so we may restrict our considerations to $\su(p,q\vert n,0)$, which is associated to 
$\omega_-$ from eq.~\eqref{involutions} and has unitarizable highest weight $\gg$-supermodules with 
respect to our standard choice of the positive system \cite[Theorem 5.2]{jakobsen1994full}. The 
$\omega$-unitarizable $\gg$-supermodules form a subcategory of $\gsmod$, where the homomorphisms are 
$\gg$-supermodule homomorphisms respecting the Hermitian forms. We will denote it as $\ugsmod$ when needed.
By Prop.\ \ref{CompletelyReducible} and
Thm.\ \ref{involutionsOnly}, the unitarizable highest weight $\gg$-supermodules are the simple 
objects in this category. By Thm.\ \ref{thm::simple_top}, they can be realized as simple tops of Kac 
supermodules. So the main question is for which $\Lambda\in \hh^*$ the simple top $L(\Lambda)$ of
the Kac supermodule $K(\Lambda)$ is unitarizable. An immediate observation, which follows from the
fact that $\omega$ preserves the Cartan subalgebra and exchanges positive with negative roots, is that
$\Lambda$ must be real, \emph{i.e.}, $\Lambda\in\hh^*_\RR$ in the standard notation. Before describing the 
rest of the answer, following \cite{jakobsen1994full}, we explain the Hermitian form.

\subsubsection{Kac--Shapovalov form} 
\label{subsubsec::KS} 

The basic point, which is familiar in physics since the early days of quantum mechanics \cite{Fock}, is 
that the Kac supermodule $K(\Lambda)$ has a unique (up to a real scalar) contravariant Hermitian form, 
the \emph{Kac--Shapovalov} form. This form induces a non-degenerate form on $L(\Lambda)$. The form on 
$L(\Lambda)$ is in general not positive definite, but if it is, then $L(\Lambda)$ is unitarizable. Namely, 
the classification of unitarizable simple highest weight $\gg$-supermodules reduces to the determination 
of all $\Lambda\in\hh^*_\RR$ with the property that the Kac--Shapovalov form on $L(\Lambda)$ is positive 
definite.

Concretely, the Kac--Shapovalov form can be defined with the help of the (Harish-Chandra) projection 
$\pr: \UE(\gg) \to \UE(\hh) \cong \S(\hh)$ on the first summand of the PBW decomposition 
$\UE(\gg) = \UE(\hh) \oplus (\nn^{-}\UE(\gg) + \UE(\gg)\nn^{+})$. We also use that $\omega$ preserves
this decomposition and provides a graded anti-isomorphism between $\UE(\nn^-)$ to $\UE(\nn^+)$. Then, 
the infinitesimal character $\chi_\Lambda$ that appeared in eq.~\eqref{centralchar} 
induces an even sesquilinear form on $\UE(\gg)$ via
\begin{equation}
(X,Y)_{\Lambda} := \chi_{\Lambda}(\pr(\omega(X)Y)) \,.
\end{equation}
This is graded conjugate symmetric, in the sense that $(X,Y)_{\Lambda} = (-1)^{XY}\overline{(Y,X)}_{\Lambda}$,
as a consequence of $\chi_\Lambda(Z) = \overline{\chi_\Lambda(\omega(Z))}$. ($\Lambda$ is real.)
On the (even) Verma supermodule $V(\Lambda)\cong \UE(\nn^-)$, we define the super 
Hermitian form $\psi$ via
\begin{equation}
\psi(X v_{\Lambda}, Y v_{\Lambda} ) = (X,Y)_{\Lambda} \,, \qquad X,Y\in \UE(\nn^-)
\end{equation}
The \emph{Kac--Shapovalov form} is the associated Hermitian form $\langle v,w 
\rangle= (-i)^v\psi(v,w)$. This is contravariant by construction, and descends to the Kac supermodule $K(\Lambda)$
(of either parity) by standard considerations. The following proposition records some of its
basic properties.

\begin{proposition} 
\label{PropertiesHermitianForm} 
Let $K(\Lambda)$ be the Kac supermodule with highest weight $\Lambda$, and let $\langle\cdot,\cdot\rangle$ be 
the Kac--Shapovalov form on $K(\Lambda)$. Then
\begin{enumerate}
\item Up to multiplication by a real scalar, $\langle\cdot,\cdot\rangle$ is the unique
contravariant form on $K(\Lambda)$. It is normalized to $\langle v_{\Lambda},v_{\Lambda}\rangle =1$.
\item Weight spaces $K(\Lambda)^{\lambda},K(\Lambda)^{\mu}$ of $K(\Lambda)$ with
$\lambda\neq \mu$ are orthogonal to each other, $\langle K(\Lambda)^{\lambda},K(\Lambda)^{\mu}\rangle=0$. 
\item The radical of $\langle\cdot,\cdot\rangle$ coincides with the maximal proper sub supermodule of 
$K(\Lambda)$. In particular, $\langle\cdot,\cdot\rangle$ induces a non-degenerate form on $L(\Lambda)$.
\end{enumerate}
\end{proposition}

Thus, a Kac supermodule is simple if and only if the radical of $\langle\cdot,\cdot\rangle$ is trivial. The
size of the radical is related to the degree of atypicality of $K(\Lambda)$ by construction of the 
Kac--Shapovalov form. This is essentially a version of Kac's criterion (\emph{cf.}\ 
\cite[Theorem 4.12]{chen2021simple}), according to which a highest weight Kac supermodule $K(\Lambda)$ 
is simple if and only if $\Lambda$ is typical. Echoing the remarks after Lemma \ref{lemm::properties_HWM}, 
the essence of Kac induction 
is that it elucidates the structure of the radical in terms of the constituents in the decomposition of
$K(\Lambda)$ as a $\even$-module. This decomposition will be summarized in subsection \ref{subsec::Dirac}, 
following \cite{SchmidtDirac}. In preparation, we describe the subset of the space $\hh^*_\RR\subset \RR^{m,n}$ 
of real weights that correspond to unitarizable highest weight $\even$-modules. 

\subsubsection{Parameterization of weight space} 

Indeed, it follows from the definition of Kac induction and Proposition \ref{CompletelyReducible} that a necessary 
condition for $L(\Lambda)$ to be unitarizable as a $\gg$-supermodule is that $L_0(\Lambda)$ be unitarizable as a 
$\even$-module. This imposes a classical sequence of standard conditions on the highest weight, which we recall 
is parameterized in terms of the standard coordinates on $\dd_\RR^*$ as
\begin{equation}
\label{standardCoordinates}
\Lambda = (\lambda^1,\ldots,\lambda^m\vert\mu^1,\ldots,\mu^n),
\end{equation}
modulo shifts by $(1,\ldots,1\vert {-}1,\ldots,-1)$ (see page \pageref{allroots}).
First, we consider the restriction to the maximal compact subalgebra $\kk$ in eq.\ \eqref{maxcompact}. If 
$L_0(\Lambda)$ is unitarizable as a $\even$-module, then as a $\kk^\CC$-module it is semisimple with 
finite multiplicities. In particular, $\Lambda$ is the highest weight of a unitarizable simple 
(hence finite-dimensional) $\kk^\CC$-module, which appears with multiplicity one. Namely, $\Lambda$ must be 
integral and dominant with respect to the positive system induced from $\gg$. On the simple $\kk^\CC$-roots
this means
\begin{equation}
\label{dominance}
\begin{split}
(\Lambda,\epsilon_i-\epsilon_{i+1}) = 
\lambda^i-\lambda^{i+1} & \in \ZZ_{\ge 0} \quad \text{for $i=1,2,\ldots,p-1$}
\\
(\Lambda,\epsilon_j-\epsilon_{j+1}) =
\lambda^j-\lambda^{j+1} & \in \ZZ_{\ge 0} \quad \text{for $j=p+1,\ldots,m-1$}
\\
-(\Lambda,\delta_k-\delta_{k+1}) = 
\mu^k-\mu^{k+1} & \in \ZZ_{\ge 0} \quad \text{for $k=1,\ldots,n-1$}
\end{split}
\end{equation}
Second, for the extension to a $\even$-module to be unitarizable, we require that the non-compact
roots belonging to $\su(p,q)$ be positive. This imposes the constraints
\begin{equation}
\label{second}
- (\Lambda,\epsilon_i-\epsilon_j) = \lambda^j - \lambda^i \ge 0
\quad
\text{for $i=1,\ldots,p-1$ and $j=p+1,\ldots,m$}
\end{equation}
Taken together, we have the inequalities
\begin{equation}
\label{inequalities}
\lambda^{p+1}\ge \cdots \ge \lambda^m \ge \lambda^1 \cdots \ge\lambda^p
\,,
\qquad
\mu^1\ge \cdots \ge \mu^n
\end{equation}
and the differences in the three chains have to be integral. Note that at this stage there is no independent 
constraint on the weights of the two abelian factors in $\kk^\CC$, and no relation between the $\lambda$'s 
and $\mu$'s. Following \cite{jakobsen1994full}, we parameterize the solution to these constraints by writing
\begin{multline}
\label{jakobsenpara}
\Lambda = (0,a_{2},\ldots,a_{m-1},0\vert b_{1},\ldots,b_{n-1},0) + \frac{\uplambda}{2} 
(1,\ldots,1,-1,\ldots,-1\vert 0,\ldots,0) \\ + \frac{\upalpha}{2}(1,\ldots,1\vert 1,\ldots,1),
\end{multline}
with integers $a_i$, $a_j$ satisfying $a_{p+1}\ge \cdots a_{m-1} \ge 0 \ge a_2 \ge \cdots \ge a_p$, integers 
$b_k$ satisfying $b_1\ge \cdots\ge b_{n-1}$ and real numbers $\upalpha$ and $\uplambda$, the second of which 
is non-positive. (To achieve this parameterization, we exploit the shift-invariance to enforce 
$\lambda^1 + \lambda^m = 2\mu^n=: \upalpha$, and define $\uplambda := \lambda^1-\lambda^m$. The 
constraints on the $a$'s and $b$'s then follow from above.)

Finally, in lieu of integrality for the non-compact roots, the requirement that the Kac-Shapovalov form be 
positive semi-definite on the entire Verma supermodule is expressed in terms of polynomials
(``Kac-Shapovalov determinants'') that measure the norm of a certain finite set of vectors in $\UE(\nn^-)v_\Lambda$. 
One finds \cite{enright1983classification, jakobsen1994full} that for fixed $a_i$'s and $b_k$'s satisfying 
the above constraints, the weight $\Lambda$ in \eqref{jakobsenpara} is the highest weight of a unitarizable 
simple highest weight $\even$-module iff
\begin{multline}
\label{thisset}
\uplambda \in \bigl(-\infty,-m+\max(i_0,j_0) + 1 \bigr) \\ \cup 
\bigl\{-m+\max(i_0,j_0) + 1, -m + \max(i_0,j_0) + 2  ,\ldots,-m+ i_0+j_0\bigr\},
\end{multline}
where $i_{0}$ is the largest index for which $a_{i}=0$, and $j_{0}$ is the smallest index for which 
$a_{m-j}\neq 0$ (if $a_{p+1}=0$ then $j_{0}=q$). The values $i_0$ and $j_0$ are part of the 
$\Phi_{c}^{+}$-dominance, and can be deduced from the \emph{length} of the following two Young diagrams:
\begin{align} \label{eq::Young_diagrams}
\begin{split}
Y_{1}(\uplambda) &:= (\lambda^{1} - \lambda^{p}, \ldots, \lambda^{1} - \lambda^{2}, 0), \\
Y_{2}(\uplambda) &:= (\lambda^{p+1} - \lambda^{m}, \ldots, \lambda^{m-1} - \lambda^{m}, 0).
\end{split}
\end{align}
Indeed, if $\len_{i}(\uplambda) := \operatorname{length}(Y_{1}(\uplambda))$, we have 
$i_0 = \len_1(\uplambda)$ and $j_0 = m - \len_2(\uplambda)$.

The set in \eqref{thisset} is geometrically a half line 
together with a finite set of discrete points that correspond to the zeros of those polynomials. We will
refer to this set, crossed with the real line for $\upalpha$, as the \emph{set of $\even$-unitarity for fixed 
spin and R-symmetry quantum numbers} and denote it as $\Gamma_0^{(a,b)}$. The full set of $\even$-unitarity 
$\Gamma_{0}$ is then defined as the union of all $\Gamma_{0}^{(a,b)}$, which is geometrically a collection 
of half-spaces for fixed spin and R-symmetry quantum numbers, together with some lines. 

The terminology is explained from the usage in the physics literature. In the cases of interest, the Dynkin 
labels of $\Lambda$ with respect to $\su(p)\oplus \su(q)$, which are $a_{i}-a_{i+1}$ for $i=1,\ldots,p-1$,
and $a_{j}-a_{j+1}$ for $j=1,\ldots m-1$, are referred to as \emph{spin quantum numbers}, while those for
$\su(n)$, which are $b_{k}-b_{k+1}$ for $k=1,\ldots, n$ are referred to as \emph{R-symmetry quantum numbers}.
The conformal dimension and superconformal R-charge, which we defined in \eqref{inducesgrading} and 
\eqref{ietoset}, respectively, are given by
\begin{equation}
\label{complicated}
\begin{split}
\Delta = \Lambda(H) &= \frac{2qp}{m} \uplambda + \frac{2q}{m} (a_2+\cdots+a_p) - \frac{2p}{m} (a_{p+1}+\cdots+a_{m-1}) 
\\
r = \Lambda(J) &= 
\begin{cases} 
\begin{array}{l} \frac{mn}{n-m}\upalpha + \frac{n}{n-m} \frac{p-q}{2}\uplambda 
+ \frac{1}{n-m} \bigl(n(a_2+\cdots+a_{m-1}) \\
\qquad\qquad\qquad\qquad\qquad\qquad+ m(b_1+\cdots+b_{n-1})\bigr)
\end{array}
& \text{if $m\neq n$}
\\
n\upalpha + \frac{p-q}2\uplambda + a_2+\cdots+a_{n-1}+b_1+\cdots+b_{n-1}
& \text{if $m=n$} 
\end{cases}
\end{split}
\end{equation}
One of the advantages of the $(\Delta,r)$ parameterization is that the passage to $\psl(n\vert n)$ is accomplished 
simply by imposing $r=0$. This restriction is known in the physics literature as \emph{decoupling of superconformal 
R-charge}. For ease of comparison with \cite{jakobsen1994full}, we will however use Jakobsen's parameterization
\eqref{jakobsenpara}, and not restrict $\upalpha$, even when $m=n$.

\section{Decomposition and Fragmentation}
\label{sec::redrecomb}

In general, unitarizability of the $\even$-module $L_0(\Lambda)$ is not sufficient for $L(\Lambda)$ to be
unitarizable as a $\gg$-supermodule. In fact, even if all constituents of $L(\Lambda)_{\ev}$ are unitarizable 
$\even$-modules, $L(\Lambda)$ might still not be unitarizable as a $\gg$-supermodule. This begs for a concise
description of the \emph{set of $\gg$-unitarity} in weight space. What is more, even if $L(\Lambda)$ is unitarizable, 
the Kac supermodule $K(\Lambda)$ is not necessarily composed solely from unitarizable $\even$-constituents. 
On the other hand, any unitarizable $\gg$-super\-module \emph{decomposes completely} in unitarizable $\even$-modules, 
and it is of interest to describe which ones occur. As another consequence, when $L(\Lambda)$ is unitarizable, but 
atypical, the composition factors of the Jordan--H\"older series of $K(\Lambda)$ as a $\gg$-supermodule are 
all highest weight supermodules, but not necessarily unitarizable. This plays a central r\^ole for the
phenomenon of \emph{fragmentation}, and hence the continuity of our indices, and is therefore important to 
describe as precisely as possible.

\subsection{Dirac operator and decomposition under \texorpdfstring{$\even$}{g0}}
\label{subsec::Dirac}

The analysis of the various issues mentioned above is expedited with the help of the algebraic Dirac operator 
for $\mathfrak{g}$. This operator was introduced in \cite{huang2007dirac,huang2005dirac}. Its use for the problem 
at hand is explained in \cite{SchmidtDirac}, while its generalization is discussed in \cite{SSchmidt_Cubic_Dirac}. 
We will here only highlight the main ideas.

For additional clarity, let us denote the supersymmetric invariant non-degenerate bilinear form $(\cdot,\cdot)$ on 
$\gg$ defined in \eqref{killing} by $\bil(\cdot,\cdot)$, and fix two complementary Lagrangian subspaces of $\odd$ 
with bases $\partial_{i}$ and $x_{i}$, for $1 \leq i \leq mn$, such that $\bil(\partial_{i},x_{j}) = - 
\bil(x_j,\partial_i) = \frac{1}{2} \delta_{ij}$. Letting $T(\odd)$ be the \emph{tensor algebra} 
over the vector space $\odd$ (placed in even degree), the \emph{Weyl algebra} is the quotient 
$\mathscr{W}(\odd) = T(\odd)/I$, where $I$ is the two-sided ideal generated by elements 
$v \otimes w - w \otimes v - 2\bil(v,w)$ for all $v, w \in \odd$. The Weyl algebra can be identified with 
the algebra of differential operators with polynomial coefficients in the variables $x_{1},\dotsc,x_{mn}$,
provided we identify \(\partial_{i}\) with the partial derivative \(\partial/\partial x_{i}\) for all 
$i = 1,\dotsc,mn$. 

With the help of the commutators \([x_{i},x_{j}] = 0\), \([\partial_{i},\partial_{j}] = 0\), and 
\([\partial_{i},x_{j}] = \delta_{ij}\) for all \(1 \leq i, j \leq mn\), one verifies that the 
Lie algebra $\even$ is a subalgebra of $\Weyl$, with Lie algebra morphism $\alpha: \even \to \Weyl$ 
given explicitly by \cite[Equation 2]{xiao2015dirac}:
\begin{multline}
\alpha(X)= \sum_{k,j=1}^{mn}(B(X,[\partial_{k},\partial_{j}])x_{k}x_{j}+B(X,[x_{k},x_{j}])\partial_{k}\partial_{j}) 
\\ -\sum_{k,j=1}^{mn}2\bil(X,[x_{k},\partial_{j}])x_{j}\partial_{k}-\sum_{l=1}^{mn}\bil(X,[\partial_{l},x_{l}]).
\end{multline}
We let $\Omega_\gg\in\UE(\gg)$ denote the quadratic Casimir operator of $\gg$, and $\Omega_\even\in \UE(\gg)\otimes
\Weyl$ the image of the quadratic Casimir of $\even$ under the diagonal embedding induced by
\begin{equation}
\even \to \mathfrak{U}(\gg) \otimes \mathscr{W}(\odd), \qquad X \mapsto X \otimes 1 + 1 \otimes \alpha(X).
\end{equation}
The \textit{Dirac operator} $\Dirac$ for $\gg$ is defined to be the odd element
\begin{equation}
\mathrm{D}=2\sum_{i=1}^{mn}(\partial_{i}\otimes x_{i}-x_{i}\otimes \partial_{i}) \in 
\UE(\gg)\otimes \mathscr{W}(\odd).
\end{equation}
It is independent of the choice of the basis of $\odd$ and is $\gg_0$-invariant under the $\even$-action on 
$\mathfrak{U}(\gg) \otimes \mathscr{W}(\odd)$ induced by the adjoint action on both factors 
\cite[Lemma 10.2.1]{huang2007dirac}, \emph{i.e.}, $[\even, \Dirac] = 0$. Analogously to reductive Lie algebras,
the Dirac operator has a nice square.

\begin{proposition}[{\cite[Proposition 10.2.2]{huang2007dirac}}] \label{SquareDirac} 
The Dirac operator $\Dirac\in \mathfrak{U}(\gg)\otimes \mathscr{W}(\odd)$ satisfies
\[
\Dirac^{2}=-\Omega_{\gg}\otimes 1+\Omega_{\even}-C,
\]
where $C$ is a constant that equals $1/8$ of the trace of $\Omega_{\even}$ on $\odd$.
\end{proposition}

For any $\gg$-supermodule $M$, the Dirac operator acts component wise on $M \otimes M(\odd)$, where 
$M(\odd)=\CC[x_1,\ldots,x_{mn}]$ is the oscillator module over $\Weyl$. If $M=\bigl(\HH,\langle\cdot,\cdot
\rangle_\HH\bigr)$ is an $\omega$-unitarizable $\gg$-supermodule, we equip the $\UE(\gg) \otimes \mathscr{W}
(\odd)$-supermodule $\HH \otimes M(\odd)$ with the positive definite Hermitian form 
\begin{equation}
\langle\cdot,\cdot \rangle_{\HH \otimes M(\odd)} =
\langle \cdot, \cdot \rangle_{\HH} \otimes \langle \cdot, \cdot \rangle_{M(\odd)}\,,
\end{equation}
where $\langle \cdot, \cdot \rangle_{M(\odd)}$ is the Bargmann--Fock Hermitian form on $M(\odd)$. 
Up to multiplication by a real scalar, $\langle\cdot,\cdot \rangle_{\HH \otimes 
M(\odd)}$ is the unique Hermitian form that is $\gg$-anti-contravariant in the first factor, and satisfies 
$x_i^\dagger = \partial_i$ in the second. For a suitable choice of positive system, this is compatible with
the extension of the conjugate-linear anti-involution $\omega$ to $\UE(\gg)$, which implies 
that the Dirac operator is \emph{self-adjoint} on $\HH\otimes M(\odd)$. As a consequence,

\begin{proposition}[\cite{SchmidtDirac}](Parthasarathy's Dirac inequality)
\label{Parthasarathy} 
\ Let $\HH$ be a unitarizable $\gg$-supermodule. Then $\Dirac^{2}\geq 0$ on $\HH\otimes M(\odd)$, 
\emph{i.e.}, $\langle v,\Dirac^{2}v\rangle_{\HH\otimes M(\odd)} \geq 0$ for all $v\in \HH\otimes M(\odd)$. 
\end{proposition}

In conjunction with Proposition \ref{SquareDirac}, this estimate is powerful enough to give a necessary
and sufficient criterion for a simple highest weight $\gg$-supermodule $L(\Lambda)$ to be unitarizable, 
explicitly in terms of its decomposition under $\even$. 

We recall that $L(\Lambda)$ was defined on p.\ \pageref{thm::simple_top} as the unique simple top of the Kac 
supermodule $K(\Lambda)$, and that in general neither of these will split in a direct sum of simple $\even$-modules. 
However, analogously to Theorem 10.4.5.\ in \cite{musson2012lie} (see also Theorem 2.5 and Corollary 2.7 in 
\cite{jakobsen1994full}; we will use similar argumentation in the discussion of fragmentation on p.\ 
\pageref{immediateLemma} below), $K(\Lambda)_\ev$ has a Jordan--H\"older-type filtration as a $\even$-module 
that induces a filtration of $L(\Lambda)_\ev$. The simple quotients are highest weight $\even$-modules 
$L_0(\Lambda-\Sigma_S)$ with highest weight of the form $\Lambda-\Sigma_S$, where $\Sigma_S := \sum_{\alpha\in S}\alpha$ 
is a sum of mutually distinct odd positive roots, for some $S\subseteq\Phi_1^+$. By Prop.\ \ref{SquareDirac}, 
$\Dirac^2$ is semi-simple with respect to this filtration, and we say that the Dirac inequality holds 
\emph{strictly} if $\Dirac^2>0$ on some $\even$-constituent $L_0(\Lambda-\Sigma_S)$.

\begin{theorem}[\cite{SchmidtDirac}] 
\label{UnitarityD2} 
Let $L(\Lambda)$ be a simple highest weight $\gg$-super\-module of highest weight $\Lambda\in \hh^{\ast}$. Then 
$L(\Lambda)$ is unitarizable if and only if all $\even$-constituents $L_{0}(\Lambda-\Sigma_S)$ are unitarizable 
$\even$-modules, and the Dirac inequality holds strictly on each $L_{0}(\Lambda-\Sigma_S)\otimes 1$ with 
$\Sigma_S \neq 0$. In that case, the decomposition as a $\even$-module is complete and reads explicitly
$$
L(\Lambda)_{\ev}\cong L_{0}(\Lambda)\oplus \bigoplus_S L_{0}(\Lambda-\Sigma_S),
$$
where the sum extends at least over all non-empty $S$ that do not contain any odd root $\alpha$ with 
$(\Lambda+\rho,\alpha)=0$. In the notation of eq.\ \eqref{Alambda}, these are all non-empty $S\subseteq 
\Phi_1^+\setminus A_\Lambda$.
\end{theorem}

The intuitive reason that the Dirac inequality is stronger than mere unitarity of the $\even$-constituents is 
that $\Dirac$ itself ``depends on the odd root vectors''. In physics terminology, the definiteness of the one
operator $\Dirac^2$ already guarantees that all supercharges are represented by (anti) self-adjoint operators.
It is worth remarking that subsets of $\Phi_1^+\setminus A_\Lambda$ in general are only a lower bound for 
the extent of the sum.

\begin{corollary}
\label{superCoro}
Let $L(\Lambda)$ be a unitarizable simple highest weight $\gg$-supermodule. Then $L(\Lambda)$ decomposes in a
finite sum of unitarizable simple highest weight $\even$-modules. The number of $\even$-constituents lies
between $2^{\#\Phi_1^+ -\# A_\Lambda}$ and $2^{\dim(\gg_{-1})}$, with maximum attained precisely if 
$\Lambda$ is typical.
\end{corollary}

Following physics, we call an irreducible unitary representation with the maximal number of 
$\even$-constituents a \emph{long supermultiplet}, all others \emph{short}. If the number of $\even$-constituents 
is smaller than $2^{\dim(\gg_{-1})-1}$, the supermultiplet might also be called \emph{ultra-short}. Corollary 
\ref{superCoro} in conjunction with Proposition \ref{PropertiesHermitianForm} thus say that a unitarizable simple 
highest weight supermodule is long precisely if it is isomorphic 
to a Kac supermodule with typical highest weight, and short precisely if it is isomorphic to the quotient of 
a Kac supermodule with atypical highest weight by the radical of the Kac--Shapovalov
form. In this context, non-zero elements of the Kac supermodule belonging to the radical of the Kac--Shapovalov 
form are also referred to as \emph{null vectors}. 

\subsection{Region of \texorpdfstring{$\gg$}{g}-unitarity and unitarity bound} 
\label{subsec::Classification} 

A complete classification of unitarizable simple highest weight $\gg$-supermodules was given by Jakobsen in 
\cite{jakobsen1994full}, and Günaydin and Volin in \cite{Günaydin}. Unitarizable simple highest weight 
$\gg$-supermodules that can be integrated to the Lie supergroup $\SU(p,q\vert n)$ were classified by Furutsu 
and Nishiyama in \cite{furutsu1991classification}. In this section, we sketch the result, assuming $p,q>0$ for 
genericity, and write down the complete classification for $\gg = \su(1,1\vert 1)$.

Following the passage from unitarity for the maximal compact subalgebra $\kk^\CC$ to $\even$-unitarity (see eqs.\ 
\eqref{dominance} to \eqref{thisset}), the first step is to require positivity of the odd roots. With respect 
to the standard ordering from eq.\ \eqref{oddpositive} and the involution $\omega_-$ from eq.\ \eqref{involutions}, 
this translates into the constraints
\begin{equation}
\label{positiveOdd}
\begin{split}
-(\Lambda,\epsilon_i-\delta_k) = -\lambda^i - \mu^k & \ge 0 \quad \text{for $i=1,\ldots,p$ and $k=1,\ldots,n$}
\\
(\Lambda,\epsilon_j-\delta_k) = \lambda^j + \mu^k & \ge 0 \quad \text{for $j=p+1,\ldots,m$ and $k=1,\ldots,n$}
\end{split}
\end{equation}
and yields an interlocking of the two chains from \eqref{inequalities}, \emph{i.e.},
\begin{equation}
\label{interlocking}
\lambda^{p+1}\geq \dotsc\geq \lambda^{m}\geq -\mu^{n}\geq \dotsc\geq -\mu^{1}\geq 
\lambda^{1}\geq \dotsc\geq \lambda^{p}.
\end{equation}
On the Jakobsen parameters from eq.\ \eqref{jakobsenpara}, this translates to the constraints
\begin{equation}
\label{previneq}
\frac{\uplambda}2 \leq \upalpha \leq -b_1-\frac{\uplambda}{2}
\end{equation}
which for fixed $a_i$, $b_k$ corresponds geometrically to a non-degenerate closed half cone in the 
$(\uplambda,\upalpha)$-plane
with vertex at $(\uplambda,\upalpha)=(-b_1,-\frac{b_1}{2})$. In the second step, the semi-definiteness of the 
Kac-Shapovalov determinants leads to further inequalities that in general are stronger except potentially where some of 
the previous inequalities \eqref{previneq} are already saturated. As a result, the \emph{set of $\gg$-unitarity for 
fixed spin and R-symmetry quantum numbers}, which we denote by $\Gamma^{(a,b)}$, consists of a congruent subcone of 
\eqref{previneq}, together with some additional half lines and possibly isolated points. The set of $\gg$-unitarity
$\Gamma\subset\hh^*$ is the union of all $\Gamma^{(a,b)}$. 
Note that in general, there are two different isomorphism classes of unitarizable simple supermodules 
for each $\Lambda\in\Gamma$, related by parity reversal. Under appropriate circumstances, physical considerations 
may fix a section of this $\ZZ/2$ bundle over $\Gamma$.
The (relative, topological) interior of $\Gamma^{(a,b)}$, where $\uplambda$ 
and $\upalpha$ can be varied independently, is a (relatively) open cone. We will denote it by $\calC^{(a,b)}$ and, 
borrowing physics terminology, call it the \emph{region of unitarity for fixed spin and R-symmetry quantum numbers}. 
It consists 
entirely of typical weights in the sense of Def.\ \ref{typatyp}, as can be shown with the help of the Dirac 
inequality Prop.\ \ref{Parthasarathy}.

\begin{lemma}[\cite{SchmidtDirac}] 
\label{lemm::region_C}
Let $\HH$ be a simple highest weight $\gg$-supermodule with highest weight $\Lambda$. If
$(\Lambda+\rho, \epsilon_{p}-\delta_{1})<0$ and $(\Lambda+\rho, \epsilon_{p+1}-\delta_{n})>0$,
then $\HH$ is unitarizable. Moreover, $\HH$ is unitarizable and has typical highest weight if 
and only if $\Lambda$ satisfies both inequalities. Namely,
$$
\calC := \bigsqcup_{(a,b)} \calC^{(a,b)} =
\{ \Lambda \in \Gamma :
 (\Lambda+\rho,\epsilon_{p}-\delta_{1}) < 0, \ (\Lambda+\rho,\epsilon_{p+1}-\delta_{n}) > 0\}.
$$
\end{lemma}

A weight at the (relative, topological) boundary of $\calC$, denoted by $\partial\calC$, is said to live 
\emph{at the unitarity bound}. This is a subset of
\begin{equation}  
\{ \Lambda \in \Gamma : 
(\Lambda+\rho,\epsilon_{p}-\delta_{1}) = 0 \ \text{or} \ (\Lambda+\rho,\epsilon_{p+1}-\delta_{n}) = 0\}.
\end{equation}
but in general not identical to it (see example below). Finally, we set $\overline{\calC}:= \calC \cup \partial 
\calC$, which is a closed subspace of $\Gamma$, and refer to $\Lambda\in\overline{\calC}$ as \emph{within the
unitarity bound}. This, together with Lemma \ref{lemm::region_C}, implies the following statements.

\begin{lemma}
\label{immediateLemma}
Let $\Lambda \in \Gamma$ be the highest weight of a unitarizable highest weight
$\gg$-supermodule.
\begin{enumerate}
\item $\Lambda \in \calC$ if and only if $\Lambda$ is typical. In particular, $K(\Lambda)\cong L(\Lambda)$.
\item If $\Lambda \in \partial \calC$, then $\Lambda$ is either $1$-atypical or $2$-atypical. If $n=1$, 
$\Lambda$ is $1$-atypical. In particular, $K(\Lambda)$ is not simple.
\end{enumerate}
\end{lemma}

Let us make this explicit for the example from the introduction \ref{sec::SQM}, the Lie superalgebra
$\sl(2\vert 1)$, with real form $\su(1,1\vert 1)$. With respect to the basis \eqref{standardRelations} and
the standard ordering, positive root vectors are $Q$, $S$, and $E$ with
roots $\alpha_Q=(-1,1)$, $\alpha_S=(1,1)$, $\alpha_E=(2,0)$. If $\bigl(\HH,\langle\cdot,\cdot\rangle\bigr)$
is a unitarizable highest weight supermodule with highest weight vector $v_\Lambda$ of scaling dimension and R-charge
$\Lambda=(\Delta,r)$, then $Qv_\Lambda=Sv_\Lambda=Ev_\Lambda=0$. Using $\bar S^\dagger= Q$, $\bar Q^\dagger=-S$,
$F^\dagger = -E$,
and commutation relations \eqref{standardRelations}, we find necessary conditions for unitarity,
\begin{equation}
\label{thefirst}
\begin{array}{lcl}
\langle \bar S v_\Lambda,\bar S v_\Lambda\rangle
=\langle v_\Lambda , Q\bar Sv_\Lambda\rangle=
\langle v_\Lambda,(-H-J)v_\Lambda\rangle \ge 0 &\Rightarrow& -\Delta - r \ge 0,
\\[.1cm]
\langle \bar Q v_\Lambda,\bar Qv_\Lambda\rangle 
= \langle v_\Lambda, -S\bar Q v_\Lambda\rangle =
\langle v_\Lambda,(-H+J) v_\Lambda \rangle \ge 0 &\Rightarrow& -\Delta + r \ge 0,
\end{array}
\end{equation}
which using \eqref{complicated} one checks are just eq.\ \eqref{previneq} in this case. The constraint that
has to be added for sufficiency arises from positivity of $\bar S \bar Q$:
\begin{equation}
\begin{split}
\langle \bar S \bar Q v_\Lambda,\bar S \bar Q v_\Lambda \rangle = &
\langle v_\Lambda, -SQ\bar S \bar Q v_\Lambda \rangle 
=
\langle v_\Lambda,\bigl(S(H+J)\bar Q + S\bar S Q\bar Q\bigr)v_\Lambda\rangle \\
\quad  =
\langle v_\Lambda , \bigl(S\bar Q (H&+J-2) + 4 EF\bigr)v_\Lambda\rangle 
=\langle v_\Lambda, \bigl((H-J)(H+J-2) + 4 H \bigr) v_\Lambda \ge 0 
\\[.1cm]
& \Rightarrow (-\Delta-r)(-\Delta+r-2) \ge 0.
\end{split}
\end{equation}
Namely, as long as the first of \eqref{thefirst} is not saturated, \emph{i.e.}, $-\Delta-r> 0$, the second 
is strengthened to $-\Delta+r-2\ge 0$. Otherwise, it is unchanged.
The region of unitarity is 
\begin{equation}
\calC = \{ \Lambda\in\hh^* : (\Lambda+\rho,\alpha_S) = \Delta-r+2 < 0 ,
(\Lambda+\rho,\alpha_Q) = -\Delta-r > 0 \},
\end{equation}
where we have used that $\rho=-\alpha_Q=(1,-1)$ in the standard positive system, and the full set of
unitarity
\begin{equation}
\label{standardset}
\Gamma = \calC \cup 
\{ \Lambda : \Delta = -r, \Delta - r = 2\Delta \le 0 \} 
\cup 
\{ \Lambda : \Delta = r-2, \Delta + r < 0 \}.
\end{equation}
The region and set of unitarity with respect to the non-standard positive system, in which in particular 
$\rho=0$, can be obtained from \eqref{standardset} by an odd reflection \eqref{oddreflection}. We sketch
the state of affairs, and indicate the decomposition of a few representative $L(\Lambda)$ into $\even$-modules,
in Fig.\ \ref{fig:Example}.

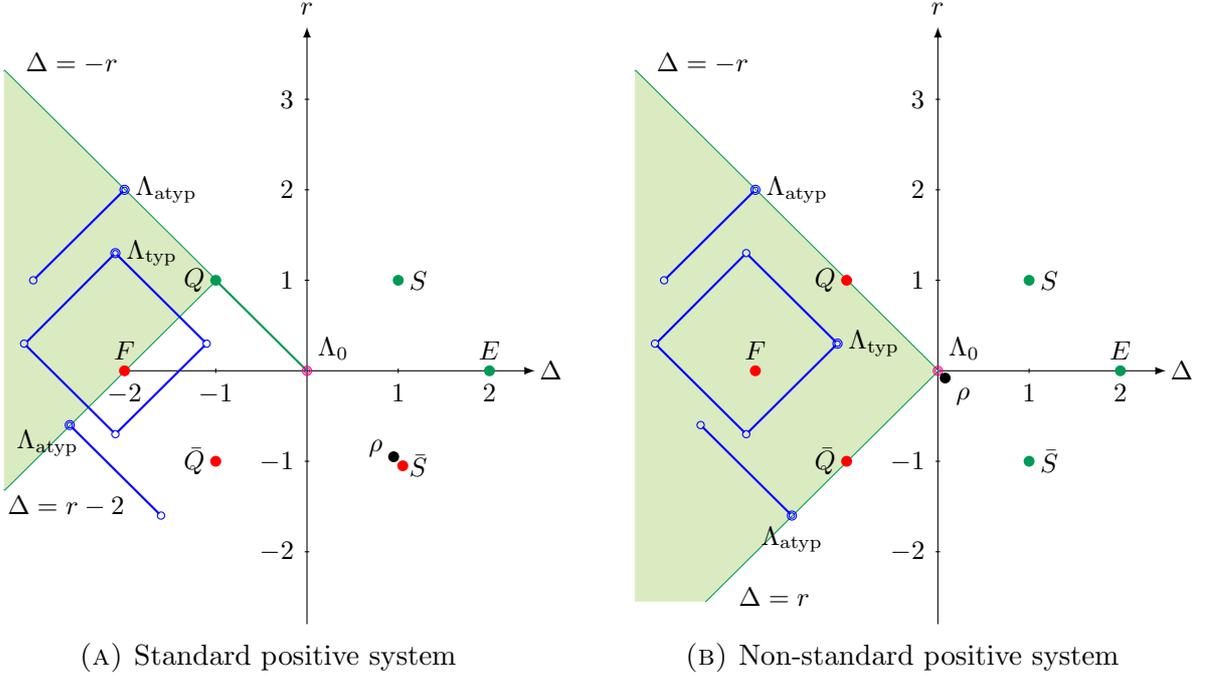
\begin{figure}[ht]
\subfloat[Standard positive system]
{
\begin{tikzpicture}
    [
    x=1cm, y=1cm,  scale=1.2,
    font=\footnotesize,
    >=latex
    ]
\draw[->] (-3.3,0) -- (2.5,0) ;
\foreach \i in {-3,...,-1}
\draw (\i,0.25mm) -- ++ (0pt,-0.5mm) node[below] {$\i$};
\foreach \i in {1,...,2}
\draw (\i,0.25mm) -- ++ (0pt,-0.5mm) node[below] {$\i$};
\draw[->] (0,-2.8) -- (0,3.8);
\foreach \i in {-2,...,-1}
\draw   (0.025,\i) -- ++ (-0.5mm,0pt) node[left] {$\i$};
\foreach \i in {1,...,3}
\draw   (0.025,\i) -- ++ (-0.5mm,0pt) node[left] {$\i$};
\draw[ForestGreen, thick] (-3.32,3.32) -- (0,0);
\draw[ForestGreen, thick] (-3.32,-1.32) -- (-1,1);
\draw (-3.4,-1.5) node[right] {$\Delta = r-2$} ;
\draw (-3.2,3.4) node[right] {$\Delta = -r$} ;
\fill[LimeGreen!30] (-3.32,3.32) -- (-1,1) -- (-3.32,-1.32) -- cycle;
\draw   (-1,1)  node[left] {$Q$};
\draw   (-1,-1)  node[left] {$\bar Q$};
\draw   (1,1)  node[right] {$S$};
\draw   (1,-1.05)  node[right] {$\bar S$};
\draw   (-2,0)  node[above] {${F}$};
\draw   (2,0)  node[above] {$E$};
\draw   (2.55,0)  node {\rlap{$\Delta$}};
\draw   (0,4.0)  node {$r$};
\fill[ForestGreen] (1,1) circle (0.6mm) (-1,1) circle (0.6mm) (2,0) circle (0.6mm);
\fill[red]  (-1,-1) circle (0.6mm) (1.05,-1.05) circle (0.6mm) (-2,0) circle (0.6mm);
\draw (-2.1,1.3) node[right] {$\Lambda_{\rm typ}$} ;
\draw[blue] (-2.1,1.3) circle (0.3mm);
\draw[blue] (-2.1,1.3) circle (0.5mm);
\draw[blue,thick] (-2.08,1.28) -- (-1.12,.32);
\draw[blue] (-1.1,.3) circle (0.4mm);
\draw[blue,thick] (-1.12,.28) -- (-2.08,-.68);
\draw[blue] (-2.1,-.7) circle (0.4mm);
\draw[blue,thick] (-2.12,-.68) -- (-3.08,.28);
\draw[blue] (-3.1,.3) circle (0.4mm);
\draw[blue,thick] (-3.07,.32) -- (-2.12,1.28);
\draw (-2,2) node[right] {$\Lambda_{\rm atyp}$} ;
\draw[blue](-2,2) circle (0.3mm);
\draw[blue] (-2,2) circle (0.5mm);
\draw[blue,thick] (-2.02,1.98) -- (-2.98,1.02);
\draw[blue] (-3,1) circle (0.4mm);

\draw (-3.3,-.8) node[right] {$\Lambda_{\rm atyp}$}; 
\draw[blue](-2.6,-.6) circle (0.3mm);
\draw[blue] (-2.6,-.6) circle (0.5mm);
\draw[blue,thick] (-2.58,-.62) -- (-1.62,-1.58);
\draw[blue] (-1.6,-1.6) circle (0.4mm);

\draw (0,0) node[above right] {$\Lambda_{0}$} ;
\draw[magenta] (0,0) circle (0.5mm);
\draw[magenta] (0,0) circle (0.3mm);

\draw (.95,-.85) node[left] {$\rho$} ;
\fill[black] (.95,-.95) circle (0.6mm);
\end{tikzpicture}
}
\qquad
\subfloat[Non-standard positive system]
{
\begin{tikzpicture}
    [
    x=1cm, y=1cm,  scale=1.2,
    font=\footnotesize,
    >=latex
    ]
\draw[->] (-3.3,0) -- (2.5,0) ;
\foreach \i in {-3,...,-1}
    \draw (\i,0.25mm) -- ++ (0pt,-0.5mm) node[below] {$\i$};
\foreach \i in {1,...,2}
    \draw (\i,0.25mm) -- ++ (0pt,-0.5mm) node[below] {$\i$};
\draw[->] (0,-2.8) -- (0,3.8);
\foreach \i in {-2,...,-1}
    \draw   (0.025,\i) -- ++ (-0.5mm,0pt) node[left] {$\i$};
\foreach \i in {1,...,3}
\draw   (0.025,\i) -- ++ (-0.5mm,0pt) node[left] {$\i$};
\draw[ForestGreen, thick] (-3.32,3.32) -- (0,0);
\draw[ForestGreen, thick] (-2.55,-2.55) -- (0,0);
\draw (-2.3,-2.5) node[right] {$\Delta = r$} ;
\draw (-3.2,3.4) node[right] {$\Delta = -r$} ;
\fill[LimeGreen!30] (-3.32,3.32) -- (0,0) -- (-2.55,-2.55) -- (-3.32,-2.55) -- cycle;
\draw   (-1,1)  node[left] {$Q$};
\draw   (-1,-1)  node[left] {$\bar Q$};
\draw   (1,1)  node[right] {$S$};
\draw   (1,-1)  node[right] {$\bar S$};
\draw   (-2,0)  node[above] {${F}$};
\draw   (2,0)  node[above] {$E$};
\draw   (2.55,0)  node {\rlap{$\Delta$}};
\draw   (0,4.0)  node {$r$};
\fill[ForestGreen]  (1,1) circle (0.6mm) (1,-1) circle (0.6mm) (2,0) circle (0.6mm);
\fill[red]  (-1,1) circle (0.6mm) (-1,-1) circle (0.6mm) (-2,0) circle (0.6mm);
\draw (-1.1,.3) node[right] {$\Lambda_{\rm typ}$} ;
\draw[blue] (-2.1,1.3) circle (0.4mm);
\draw[blue,thick] (-2.08,1.28) -- (-1.12,.32);
\draw[blue] (-1.1,.3) circle (0.3mm);
\draw[blue] (-1.1,.3) circle (0.5mm);
\draw[blue,thick] (-1.12,.28) -- (-2.08,-.68);
\draw[blue] (-2.1,-.7) circle (0.4mm);
\draw[blue,thick] (-2.12,-.68) -- (-3.08,.28);
\draw[blue] (-3.1,.3) circle (0.4mm);
\draw[blue,thick] (-3.07,.32) -- (-2.12,1.28);
\draw (-2,2) node[right] {$\Lambda_{\rm atyp}$} ;
\draw[blue](-2,2) circle (0.3mm);
\draw[blue] (-2,2) circle (0.5mm);
\draw[blue,thick] (-2.02,1.98) -- (-2.98,1.02);
\draw[blue] (-3,1) circle (0.4mm);

\draw (-1.6,-1.6) node[below] {$\Lambda_{\rm atyp}$} ;
\draw[blue](-2.6,-.6) circle (0.4mm);
\draw[blue,thick] (-2.58,-.62) -- (-1.62,-1.58);
\draw[blue] (-1.6,-1.6) circle (0.3mm);
\draw[blue] (-1.6,-1.6) circle (0.5mm);

\draw (0,0) node[above right] {$\Lambda_{0}$} ;
\draw[magenta] (0,0) circle (0.5mm);
\draw[magenta] (0,0) circle (0.3mm);
\draw (0.08,-.28) node[right] {$\rho$} ;
\fill[black] (0.08,-.08) circle (0.6mm);
    \end{tikzpicture}
}
\caption{The set of $\sl(2\vert 1)$-unitarity with respect to two canonical positive systems, together
with the highest weights of the $\sl(2)$-constituents of some representative unitarizable highest weight 
supermodules. The highest weight of each supermodule is indicated by a double circle, and depends on the 
positive system. Also note that the ``trivial'' one-dimensional supermodule with $\Lambda_0=(0,0)$, shown 
in magenta, is 1-atypical in this case, yet has only one $\sl(2)$ constituent. The coordinates are with 
respect to the basis $(H,J)$ of the Cartan subalgebra used in the presentation 
in eq.\ \eqref{standardRelations}.}
\label{fig:Example}
\end{figure}

We conclude this section with the description of the integral points of $\Gamma$, \emph{i.e.}, those 
unitarizable simple supermodules of $\gg$ that integrate to unitary representations of $\SU(p,q) \times 
\SU(n) \times \mathrm{U}(1)$. Indeed, the associated weights of these supermodules are integral as they 
are obtained by the differentiation of a representation of a maximal torus in $\SU(p,q) \times \SU(n) 
\times \mathrm{U}(1)$. To describe them concisely, we define for any integral $\Lambda \in \Gamma_0$ 
the following quantities according to the unitarity conditions in eq.\ \eqref{interlocking}:
\begin{equation}
g_{1}(\Lambda) := \lambda^{m} + \mu^{n}, \qquad
g_{2}(\Lambda) := \lambda^{1} + \mu^{1}, \qquad
g_{3}(\Lambda) := \lambda^{1} - \lambda^{m},
\end{equation}
where the latter value must satisfy
$
g_{3}(\Lambda) \leq -\len_{1}(\Lambda) - \len_{2}(\Lambda)
$
according to eq.\ \eqref{thisset} and our definition of $\len_{1}(\Lambda), \len_{2}(\Lambda)$ in eq.\ 
\eqref{eq::Young_diagrams}. Furthermore, we define $\len_{3}(\Lambda)$ to be the length of the Young diagram
\begin{equation}
Y_{3}(\Lambda) := (\mu^{1} - \mu^{n}, \ldots, \mu^{n-1} - \mu^{n}, 0).
\end{equation}

\begin{theorem}[\cite{furutsu1991classification, SchmidtGeneralizedSuperdimension}] 
\label{thm::classification_integral}
Let $M$ be a simple highest weight $\gg$-supermodule with integral highest weight $\Lambda$ that is $\even$-semisimple. 
Then $M$ is unitarizable if and only if the following two statements hold:
\begin{enumerate}
\item[a)] $\Lambda$ satisfies  
$\lambda^{p+1}\geq \dotsc\geq \lambda^{m}\geq -\mu^{n}\geq \dotsc\geq -\mu^{1}\geq 
\lambda^{1}\geq \dotsc\geq \lambda^{p}.$
\item[b)] $\Lambda$ satisfies one of the following two conditions:
\begin{enumerate}
\item[(i)] $g_{2}(\Lambda) \leq -\len_{1}(\Lambda) - q$ and $g_{1}(\Lambda) \leq \len_{3}(\Lambda)$.
\item[(ii)] $g_{2}(\Lambda) \leq -\len_{1}(\Lambda) - \len_{2}(\Lambda)$ and $g_{1}(\Lambda) = \len_{3}(\Lambda) = 0$.
\end{enumerate}
\end{enumerate}
\end{theorem}

\subsection{Fragmentation and recombination}
\label{fragrecom}

If $L(\Lambda)$ is unitarizable, but $\Lambda$ is atypical (\emph{i.e.}, $\Lambda\in\Gamma\setminus \calC)$, the 
Kac supermodule $K(\Lambda)$ is not simple, but also does not split into a direct sum of simple supermodules. 
Instead, as $K(\Lambda)$ is of highest weight type, it has a Jordan--Hölder series 
$K_0= K(\Lambda) \supsetneq K_1 \supsetneq \cdots \supsetneq K_M=0$ such that each $K_{i-1}/K_{i}$ is isomorphic to
a simple highest weight $\gg$-supermodule $L(\Lambda_i)$, with highest weight $\Lambda_i$ of the form 
$\Lambda-\Sigma_S$ where $\Sigma_S := \sum_{\alpha\in S} \alpha$ is a sum of mutually distinct
odd positive roots, for some $S\subseteq \Phi_1^+$ (\emph{cf.\ } discussion around Theorem \ref{UnitarityD2}).
The $\Lambda_i$ will all be atypical, but it is difficult to describe in general how many and which 
ones occur. Also the composition factors $L(\Lambda_i)=K_{i-1}/K_i$ for $i=2,\ldots,M$ need not be 
unitarizable. It is an instructive exercise to show that this does not happen within the unitarity 
bound.\footnote{In the example, $K(-\frac 12,\frac 12)$ has unitarizable top quotient, but, if 
defined with respect to the standard positive system, $K_1=L(\frac 12,-\frac 12)$, which is not 
unitarizable. On the other hand, $K(\alpha_Q)$ has factors $K_1/K_2= L(0,0)$, $K_2=L(\alpha_F)$, which 
are both unitarizable.}

\begin{lemma} \label{lemm::recombination_rules}
For $n>1$, the simple highest weight supermodules in the Jordan--Hölder series of $K(\Lambda)$ for 
$\Lambda \in \partial \calC$ are 
\begin{enumerate} 
\item $L(\Lambda)$ and $L(\Lambda+\epsilon_{p}-\delta_{1})$ if $\Lambda \in \partial \calC$ satisfies 
only $(\Lambda+\rho,\epsilon_{p}-\delta_{1})=0$.
\item $L(\Lambda)$ and $L(\Lambda-\epsilon_{p+1}+\delta_{n})$ if $\Lambda \in \partial \calC$ satisfies 
only $(\Lambda+\rho,\epsilon_{p+1}-\delta_{n})=0$.
\item $L(\Lambda), L(\Lambda+\epsilon_{p}-\delta_{1})$, $L(\Lambda-\epsilon_{p+1}+\delta_{n})$, 
and $L(\Lambda+(\epsilon_{p}-\delta_{1})+(-\epsilon_{p+1}+\delta_{n}))$ if $\Lambda \in \partial 
\calC$ satisfies both $(\Lambda+\rho,\epsilon_{p}-\delta_{1})=0$ and $(\Lambda+\rho,\epsilon_{p+1}-\delta_{n})=0$.
\end{enumerate}
\end{lemma}
As a consequence, letting
\begin{equation}
\label{assGraded}
\gr K(\Lambda) := \bigoplus_{i=1}^{M} K_{i-1}/K_{i} \cong \bigoplus_{i=1}^{M}L(\Lambda_{i}),
\end{equation}
denote the graded $\gg$-supermodule associated to the Jordan--H\"older series,\footnote{
Note that the parity of the $L(\Lambda_i)$ is determined by the parity of $K(\Lambda)$ and 
the number of odd roots in $\Sigma_S=\Lambda-\Lambda_i$. This is not necessarily equal to the grading 
induced from the Jordan--H\"older series, which in general is neither $\ZZ/2$-graded (nor unique for 
that matter).} we obtain,
\begin{lemma} 
\label{lemm::unitarity_bound_unitary}
Let $\Lambda \in \overline{\calC}$. Then $\gr K(\Lambda)$ is a unitarizable $\gg$-supermodule.
\end{lemma}

If $L(\Lambda)=K_0/K_1$ is unitarizable, then $K_1$ is the radical of the Kac--Shapovalov form on $K(\Lambda)$ 
by Proposition \ref{PropertiesHermitianForm}. As a consequence, assuming the factors $L(\Lambda_i)$
for $i=2,\ldots,M$ remain unitarizable, their Hermitian form is \emph{not} induced from the 
Kac--Shapovalov form on $K(\Lambda)$. Remarkably \cite{kinney2007index}, this can be remedied in a 
natural way in the consideration of \emph{continuous families} of unitary representations in the
sense of section \ref{subsec::assumptions}. Concretely, let $(\Lambda^{(k)})_{k=1,2,\ldots}\subset \calC$ 
be a sequence of typical weights with unitarizable highest weight modules $L(\Lambda^{(k)})$, and 
$\lim_{k\to\infty}\Lambda^{(k)}=\Lambda_0\in\partial\calC$ at the unitarity bound, where the limit
is taken in the usual topology on $\hh^*$. Then, viewing 
the $L(\Lambda^{(k)})=K(\Lambda^{(k)})$ as a sequence of $\gg$-supermodule structures on a fixed 
underlying (pre-)Hilbert space $\HH$, it can be seen that the associated weight space decomposition (Lemma 
\ref{lemm::properties_HWM} (b) and Proposition \ref{PropertiesHermitianForm} (b)) agrees in the limit 
$k\to\infty$ with the weight space decomposition on the (reducible!) super\-module $K(\Lambda_0)$. In 
particular, this induces a non-degenerate Hermitian form on $K(\Lambda_0)\cong\HH$, which by uniqueness
must agree with the Kac--Shapovalov form on the composition factors. In the sense of section 
\ref{subsec::assumptions}, the assignment $\Lambda\mapsto \gr K(\Lambda)$ is continuous on 
$\overline\calC$, and
\begin{equation}
\label{fragDef}
\lim_{\Lambda\to\Lambda_0} \gr K(\Lambda) = \gr K(\Lambda_0).
\end{equation}

Following physics terminology, we think of the ``filling in'' of $\gr K(\Lambda_0)$ in a continuous
family as the \emph{fragmentation of the long supermultiplet $K(\Lambda)$ as $\Lambda\to\Lambda_0$ hits the 
unitarity bound}, or conversely as \emph{recombination of the short constituents when $\Lambda$ moves 
away from it}. It is important to emphasize that this \emph{does not} imply that the constituents of $\gr 
K(\Lambda_0)$ can only appear together in a continuous family. Quite to the contrary, the ``fragments'' 
can move around independently from each other, as long as their highest weight remains in $\Gamma\setminus\calC$.
In principle, they can also undergo further fragmentation in codimension 2, but we leave 
a complete description of this possibility for future work. If $n>1$ and $\Lambda_0\in\overline\calC$, Lemma 
\ref{lemm::recombination_rules}, implies
\begin{equation} 
\lim_{\Lambda \to \Lambda_{0}} \gr K(\Lambda) = \begin{cases}
L(\Lambda_{0}) \oplus L(\Lambda_{0} -\gamma_{1}) &\text{if} \ (\Lambda_{0} + \rho, \gamma_{1}) = 0, \ 
(\Lambda_{0} + \rho, \gamma_{2}) \neq 0, 
\\
L(\Lambda_{0}) \oplus L(\Lambda_{0} - \gamma_{2}) &\text{if} \ (\Lambda_{0} + \rho, \gamma_{1}) \neq 0, \ 
(\Lambda_{0} + \rho, \gamma_{2}) = 0, 
\\
\begin{array}{@{}l}
L(\Lambda_{0}) \oplus L(\Lambda_{0} -\gamma_{1}-\gamma_2)   \\ 
\quad\oplus L(\Lambda_{0} - \gamma_1) \oplus L(\Lambda_0-\gamma_2) 
\end{array}
&\text{if} \ 
(\Lambda_{0} + \rho, \gamma_{1}) = 0, \ (\Lambda_{0} + \rho, \gamma_{2}) = 0,
\end{cases}
\end{equation}
where $\gamma_{1} := -\epsilon_{p} + \delta_{1}$ and $\gamma_{2} := \epsilon_{p+1} - \delta_{n}$, 
and the $\ZZ/2$-grading is left implicit, as usual.

The fact that unitarizable simple $\gg$-supermodules with atypical highest weight (\emph{i.e.,} $\Lambda\in 
\Gamma\setminus\calC$) can not by themselves ``move back into $\calC$'', and specifically that their 
conformal dimension is completely determined by their superconformal R-charge, is expressed in physics 
as the statement that \emph{short supermultiplets are protected}. Supermodules that do not participate at all 
in any fragmentation/recombination process are called \emph{absolutely protected}. 
\label{absProt} By the above, 
a unitarizable simple supermodule is absolutely protected if and only if it does not appear as composition 
factor in the Jordan--H\"older series of $K(\Lambda_0)$ for any $\Lambda_0\in \partial\calC$. More explicitly, we
can describe this as follows.

\begin{lemma}
\label{absProtLemma}
Let $L(\Lambda)$ be a unitarizable simple highest weight $\gg$-supermodule. Then $L(\Lambda)$ is
absolutely protected if and only if $(\Lambda+\rho, \epsilon_{p}-\delta_{1})>0$ and
$(\Lambda+\rho, \epsilon_{p+1}-\delta_{n})<0$ holds.
\end{lemma}

\begin{corollary} \label{cor::maximal_protected_gleich_absolutely_protected}
If $m,n \geq 3$, then maximally atypical unitarizable simple highest weight $\gg$-supermodules $L(\Lambda)$ 
are absolutely protected. 
\end{corollary}
\begin{proof}
Since $m, n \geq 3$, it follows from the assumption that $\at(\Lambda) > 2$. Consider the roots
$\gamma_1 := -\epsilon_p + \delta_1$ and $\gamma_2 := \epsilon_{p+1} - \delta_n$. Then $(\gamma_1, 
\gamma_2) = 0$, showing that $\gamma_1$ and $\gamma_2$ are orthogonal. This implies that the $\gg$-supermodules 
appearing in the Jordan--Hölder series of $K(\Lambda_0)$, for $\Lambda_0 \in \partial \calC$, all share the same 
degree of atypicality — specifically, either $1$ or $2$ (see Lemma \ref{lemm::recombination_rules}). In particular, 
this shows that $L(\Lambda)$ cannot occur as a simple quotient in the Jordan--Hölder series of $K(\Lambda_0)$ for 
any $\Lambda_0 \in \overline{\calC}$.
\end{proof}

\section{Duflo-Serganova Functor} 
\label{SectionDS}

The \emph{Duflo--Serganova (DS) functor} is a symmetric monoidal tensor functor that serves as a powerful tool 
for studying the representation category of Lie superalgebras \cite{duflo2005associated}.  The DS functor  
is associated with any odd square-zero element and relates the $\ZZ/2$-graded 
representation theory of Lie superalgebras of different dimensions in a cohomological manner. In mathematical 
physics, the DS functor appears as a representation theoretic analogue of the so-called twisting procedure 
that reduces a (Lagrangian) supersymmetric field theory to a much simpler \emph{topological} or 
\emph{holomorphic} version \cite{WittenMorse,Witten:1988ze,CostelloNotes,SaberiWilliams}. In fact, the relation
between the two concepts should be much more than an analogy, with the superconformal index as a very
concrete link. We here give a brief introduction to the DS functor, and describe its extension to the unitarizable 
subcategory. Our exposition follows mostly 
\cite{duflo2005associated, gorelik2022duflo, SchmidtGeneralizedSuperdimension}.

\subsection{The algebraic theory}
\label{DScomplex}

The DS functor can be attached quite generally to any odd element $x\in\gg_1$ of a complex Lie superalgebra 
$\gg=\even \oplus \odd$ that ``squares to zero'', in the sense that $x^2 := \frac 12 [x,x]=0$. By the invariance 
of the bracket, any homological constructions will be equivariant with respect to the adjoint action of the 
Lie group $G_0$, attached to $\even$. This makes it natural to adopt a geometric point of view 
\cite{gruson2000ideal}.

\begin{definition}  
\label{nilpotenceDef}
The \emph{self-commuting variety} of a Lie superalgebra $\gg$ is the algebraic variety  
$$  
\YY := \bigl\{ x \in \odd : [x,x] = 0 \bigr\}.  
$$  
\end{definition}  

In the context of Lagrangian field theories with super Poincar\'e symmetry, the self-commuting variety has
also been called \emph{nilpotence variety} \cite{Eager:2018dsx,Elliott:2020ecf}. Its primary purpose was 
to parameterize the possible (topological and holomorphic) twists of such theories, which as mentioned above 
are quite analogous to the DS functor, but by now it has turned out to have much wider applicability also.

The adjoint action of $G_{0}$ preserves $\YY$ and makes it a $G_{0}$-invariant Zariski-closed cone in $\odd$. 
There are in general finitely many $G_0$-orbits, which are in one-to-one correspondence with the orbits of
the Weyl group acting on the set of collections $S\subset \Phi_1$ of mutually orthogonal linearly independent odd 
roots, see ref.\ \cite[Section 4]{duflo2005associated}. This is reflected in a numerical invariant, called 
the \emph{rank}, that we can attach to each $G_{0}$-orbit. To define it, we use that given $x \in \YY$, 
there exists an element $g \in G_{0}$ along with isotropic, mutually orthogonal, and linearly independent 
roots $\alpha_{1}, \dotsc, \alpha_{k}$ such that  
\begin{equation} 
\label{eq::rank}
\Ad_{g}(x) = u_{1} + \dotsb + u_{k},  
\end{equation}
where each $u_{i}$ is a non-zero element of $\gg^{\alpha_{i}}$. The number of roots, $k$, in this representation 
does not depend on the choice of $g$ and thereby defines the \emph{rank} of $x$, which we will denote by $\rk(x)$.
Equivalently, $\rk(x)$ can be interpreted as the rank of $x$ when regarded as a linear operator acting in 
the standard representation.\footnote{This is in general different from the rank of $x$ acting in the adjoint
representation, which is the definition used in the context of nilpotence varieties.} Necessarily, $\rk(x)\le 
\defect(\gg)$ for all $x\in\YY$.

Now to any $\gg$-supermodule $M$, we attach a $G_{0}$-invariant subvariety of $\YY$, called the \emph{associated 
variety} $\YY_{M}$, which provides useful information about atypicality if $M$ is of highest weight type. To this end, 
we note that any element $x \in \YY$ defines an odd endomorphism $x_{M} \in \text{End}_{\CC}(M)$ such 
that $x_{M}^{2} = x_{M} \circ x_{M} = 0$. We then set
\begin{equation}
\label{twistmoddef}
M_{x} := \Ker(x_{M}) / \Im(x_{M}),
\end{equation}
and define the \emph{associated variety} of a $\gg$-supermodule $M$ as the $G_{0}$-invariant subvariety
\begin{equation}
\YY_{M} := \{x \in \YY \ : \ M_{x}\neq 0\}\subset\YY.
\end{equation}
The following statement is similar to Lemma 5.12 and Proposition 5.14 in \cite{coulembier2017homological}.

\begin{proposition}
\label{VarietyTypical} 
Let $M$ be a highest weight $\gg$-supermodule with highest weight $\Lambda$.
\begin{enumerate}
\item The associated variety of $M$ is trivial if and only if the highest weight $\Lambda$ is typical. 
\item If $\alpha$ is an odd positive root satisfying $(\Lambda+\rho,\alpha) = 0$, any associated root 
vector $Q_{\alpha}$ lies in $\YY_{M}$. In particular, $\YY_{M} \neq \{ 0\}$.  
\end{enumerate}
\end{proposition}

Next, to define the DS functor following \cite{duflo2005associated,serganova2011superdimension}, we observe that
when $M=\gg$ the adjoint representation, the space \eqref{twistmoddef}  
\begin{equation}  
\label{twistladef}
\gg_{x} := \Ker(\ad_{x}) / \Im(\ad_{x}),  
\end{equation} 
is a Lie superalgebra for any $x\in\YY$, since $[x, \gg] = \Im(\ad_{x})$ is an ideal in $\Ker(\ad_{x})$. Naturally, 
the defect of $\gg$ and that of $\gg_{x}$ are related through the rank of $x$ \cite[Section 4]{gorelik2022duflo}:  
\begin{equation}
\label{eq85}
\defect(\gg_{x}) = \defect(\gg) - \rk(x).
\end{equation} 
For instance, when $\gg=\slmn$ as in our main example, we have 
\begin{equation}
\label{eq86}
\rk(x)\le \min(m,n) \;\text{ and }\;
\slmn_{x} = \mathfrak{sl}(m-k\vert n-k).
\end{equation}
For reference, we also record that with respect to the given Cartan subalgebra $\hh$ and root system $\Phi$ of $\gg$, 
the Lie superalgebra $\gg_{x}$ possesses a Cartan subalgebra given by  
\begin{equation}  
\label{CSAx}
\hh_{x} := (\Ker(\alpha_{1}) \cap \dotsb \cap \Ker(\alpha_{k})) / (\hh^{\alpha_{1}} \oplus \dotsb \oplus \hh^{\alpha_{k}}),  
\end{equation}  
where $\hh^{\alpha} = [\gg^{\alpha}, \gg^{-\alpha}]$, and the root system associated with $(\gg_{x}, \hh_{x})$ is  
\begin{equation}  
\label{rootsysx}
\Phi_{x} = \{\alpha \in \Phi \mid (\alpha, \alpha_{i}) = 0, \ \alpha \neq \alpha_{i}, \ i = 1, \dotsc, k\}.  
\end{equation}  

Along similar lines, when $M$ is a $\gg$-supermodule, the space $M_x=\Ker(x_M)/\Im(x_M)$ defined in \eqref{twistmoddef}
is naturally a $\gg_{x}$-supermodule, as $\Ker(x_{M})$ is invariant under $\Ker(\ad_{x})$, and the relation 
$[x,\gg] \Ker(x_{M}) \subset \Im(x_{M})$ ensures compatibility with the quotient structure. It is also clear
that this construction is compatible with homomorphisms of supermodules. This yields the following generalization 
of statements proven in the finite-dimensional case in ref.\ \cite[Section 2]{gorelik2022duflo}.

\begin{proposition}
\label{DSTensorFunctor}
Let $\gg$ be a Lie superalgebra, and $x\in\YY$ an element of the self-commuting variety. Then
\begin{enumerate}
\item The map $\DS_{x}: \gsmod \to \gxsmod$, $M\mapsto \DS_{x}(M) := M_{x}=\Ker(x_M)/\Im(x_M)$ defines a functor 
from the category of $\gg$-supermodules to the category of $\gg_{x}$-supermodules, called the \emph{Duflo--Serganova 
functor}. $\DS_x$ is an additive functor and preserves the symmetric monoidal tensor structure. 
\item For any exact sequence $0\to M' \to M \to M''\to 0$
of $\gg$-supermodules, there exists an exact sequence of $\gg_{x}$-supermodules
$$
0\to E \to \DS_{x}(M')\to \DS_{x}(M)\to \DS_{x}(M'')\to \Pi(E)\to 0
$$
for some $\gg_{x}$-supermodule $E$ (and where $\Pi$ is the parity reversing functor). In particular, the DS functor 
is middle exact and satisfies $\DS_{x}(\Pi(M))=\Pi(\DS_{x}(M))$ for any $\gg$-supermodule $M$. 
\end{enumerate}
\end{proposition}

For highest weight supermodules, there is a tight connection between the degree of atypicality and the rank
of $x$, similar to \eqref{eq85}.

\begin{theorem}[{\cite{serganova2011superdimension}}]
\label{DScharacter} 
Let $M=L(\Lambda)$ be a simple highest weight $\gg$-supermodule, of degree of atypicality $\at(M)=\at(\Lambda)$.
Then, for any $x\in\YY$ of rank $\rk(x)$, we have
$$
\at(\DS_x(M)) = \at(M)-\rk(x)
$$
\end{theorem}

\begin{corollary}
\label{DSNM}
Let $x \in \YY$ be an element of the self-commuting variety, and $M$ a highest weight $\gg$-supermodule.
Then
$$
\at(M) < \rk (x) \Rightarrow \DS_{x}(M) = 0
$$
In particular, if $M$ is a highest weight $\gg$-supermodule with typical highest weight, then $\DS_x(M)=0$ 
is trivial for any $x\in\YY\setminus 0$.
\end{corollary}

\subsection{DS functor and unitarity} 
\label{subsec::DS_functor_and_unitarity}

If $\omega$ is a conjugate-linear anti-involution of $\gg$, $\HH$ an $\omega$-unitarizable highest weight 
$\gg$-supermodule, and $x \in \YY$ an element of the self-commuting variety, the next step would seem to be 
to ask for conditions under which $\DS_{x}(\HH)$ is unitarizable. If one thinks that a real form on $\gg_x$ 
suitable for this purpose should be induced naturally if $x$ itself is (anti-)invariant under $\omega$, one 
might well have another think coming. This is because, \emph{self-commutativity 
and reality are incompatible!} For our Lie superalgebra $\gg=\slmn$, which is reductive, this follows on
account of the integral grading \eqref{intgrading} from the fact that for any non-zero $x\in\gg_{+1}$, 
$[x,\omega(x)]\neq 0$, so that no $x\in\YY$ can satisfy $\omega(x)=-x$. In other words, $\YY$ has no
real points in this case. More fundamentally,\footnote{This fact, without which cohomological field theories would 
be much less interesting for physics, seems grossly underappreciated. J.W.\ would like to thank E.\ Rabinovici, 
K.\ Hori, E.\ Witten, S.\ Hellerman, I.\ Saberi, F.\ Hahner, and R.\ Senghaas for some extensive long-ago and 
more recent discussions on this issue.} and without any assumptions on $\gg$, any $x\in\gg^\omega$ acts on $\HH$ as a 
(possibly (anti-)imaginary) self-adjoint operator. This implies that $x^2 = -x\omega(x)$ is (up to a phase positive or 
negative semi-)definite, and vanishes (if and) only if $x$ itself acts by $0$. As a consequence, $\HH_x=\HH$ 
for any such $x$ because the unitary representation theory of $\gg$ factors through $\YY^\omega$.

Here, we will follow the approach outlined in \cite{gorelik2022duflo} to extend the DS functor to the
real form. We note that $\even$ has semisimple elements, as it is reductive, and define the \emph{rank variety}
as
\begin{align}  
\YY^{\text{hom}} := \{ x \in \odd : [x,x] \text{ is semisimple} \},  
\end{align}  
which is stable under $G_{0}$, although not closed in $\odd$. Elements of $\YY^{\text{hom}}$ are referred to as 
\emph{homological}. Clearly $\YY\subset\YY^{\rm hom}$ as a closed subvariety. Fixing $x \in \YY^{\text{hom}}$, we 
set $c := [x, x]$. For a $\gg$-supermodule $M$, we let $M^{c}$ denote the space of $c$-invariants on $M$. Then, 
$x$ defines a square-zero endomorphism on $M^{c}$, allowing us to consider its cohomology:\footnote{In physics,
localization with respect to ``non-square zero complex supercharges'' is known in various incarnations, such as 
twisted masses, central charges, etc., but we are not aware of a systematic, geometric treatment of the associated
twist varieties.}
\begin{equation}  
M_{x} := \Ker (x\vert_{M^{c}}) / \Im (x\vert_{M^{c}}).  
\end{equation} 
This reduces to \eqref{twistmoddef} if $x\in\YY$. Also, generalizing \eqref{twistladef}, we put
$\gg_x=\Ker(\ad_x|_{\gg^c}) / \Im (\ad_x\vert_{\gg^c})$. The following lemma then follows from the same 
reasoning as before \cite{duflo2005associated}.  

\begin{lemma}  
For all $x \in \YY^{\text{hom}}$, $M\in\gsmod$, the following assertions hold: 
\begin{enumerate}  
\item[a)] $\gg_{x}$ has the natural structure of a Lie superalgebra.  
\item[b)] $M_{x}$ is a $\gg_{x}$-supermodule.  
\end{enumerate}  
\end{lemma}  

As a consequence, for any $x \in \YY^{\text{hom}}$, we have a functor  
\begin{equation}  
\DS_{x}(M) := M_{x} := \Ker (x\vert_{M^{c}}) / \Im (x\vert_{M^{c}})  
\end{equation}  
from the category of $\gg$-supermodules to the category of $\gg_{x}$-supermodules as before. This functor is 
also referred to as the \emph{DS functor} with respect to $x$. Moreover, it is a tensor functor. For given
conjugate-linear anti-involution $\omega$, the real locus $\bigl(\YY^{\rm hom}\bigr)^\omega$ is not 
necessarily empty. This allows for the study of unitarity. If $x\in\YY^{\rm hom}$ with $\omega(x)=-x$, then
$\omega$ induces a well-defined conjugate-linear anti-involution $\omega_{x}$ on $\gg_{x}$.

\begin{theorem}[{\cite{SchmidtGeneralizedSuperdimension}}]
\label{thm::unitarity_DS}
Let $\HH$ be a unitarizable $\gg$-supermodule and let $x \in \YY^{\text{hom}}$ such that $\omega(x) = -x$. 
Then, $\DS_{x}(\HH) = \ker x\vert_{\HH^{c}}$ and is an $\omega_{x}$-unitarizable $\gg_{x}$-supermodule. In 
particular, $\DS_{x}(\HH)$ is a semisimple $\gg_{x}$-supermodule.
\end{theorem}

In generalization of \eqref{eq::rank}, given $x\in \YY^{\rm hom}$ with $\omega(x)=-x$, there exists
$g\in G_0$ and roots $\alpha_i$, $i=1,\ldots,k=\rk(x)$ as before such that
\begin{equation}
\Ad_g(x) = u_1 + \cdots + u_k + z_1 v_1 + \cdots + z_k v_k, 
\end{equation}
where $u_i\in \gg^{\alpha_i}\setminus 0$, $v_i\in \gg^{-\alpha_i}\setminus 0$ and $z_i=\pm 1$ determine the 
\emph{signature} of $x$. In the example $\gg=\slmn$, we have \cite[Section 5]{Serganova_Sherman_Localization}
\begin{equation}
\gg_{x}^{\omega_x}  = \bigl(\DS_{x}(\gg)\bigr)^{\omega_x} = \su(p-r,q-s\vert n-k) 
\end{equation}
for some $r\le p$, $s\le q$ with $r+s=k$. Furthermore, the centralizer of $c := [x,x]$ 
in $\gg$ is 
\begin{equation}
    C(c) = \ker \ad_{c} = \mathfrak{sl}(m-k\vert n-k) \oplus \sl(1\vert 1)^{\oplus k}.
\end{equation}
This implies that $c$ cannot be regular; otherwise, $C(c)$ would be abelian.

The $\gg_{x}$-supermodule $\DS_{x}(\HH)$ is generally not simple. In one special case however, which will 
be sufficient for our purposes in section \ref{sec::anindex}, simplicity can be established \cite{kinney2007index}.

\begin{theorem}[\cite{SchmidtGeneralizedSuperdimension}] 
\label{thm::g_Q_simple}
Let $\HH$ be a unitarizable simple $\gg$-supermodule and let $x \in \YY^{\text{hom}}$ such that $\omega(x) = -x$. 
Assume the highest weight $\Lambda$ of $\HH$ satisfies $(\Lambda + \rho, \alpha) = 0$ for some odd root 
$\alpha$, and $x = u_{\alpha} + z_{\alpha} v_{\alpha}$ has rank one with $c := [x,x]  \in \hh$ and $u_{\alpha} 
\in \gg^{\alpha}, v_{\alpha} \in \gg^{-\alpha}$ and $z_{\alpha}=\pm 1$. 
\begin{enumerate}
\item $\DS_{x}(\HH)$ is non-trivial and decomposes into at most two unitarizable simple $\gg_{x}$-super\-modules.
    \item If $\alpha$ is simple, then $\DS_{x}(\HH)$ is a simple 
$\gg_{x}$-supermodule.
\end{enumerate}
\end{theorem}

We use Theorem~\ref{thm::g_Q_simple} to show that $\DS_{x}(\HH)$ decomposes into 
finitely many $\gg_{x}$-super\-modules whenever $\HH$ is simple. To this end, we first observe that 
$\DS_{x}(\HH) = \ker x = \HH^{ic} = \HH^{c}$, since $ic$ is a positive operator. Now, consider an element 
$x = u_1 + u_2 + z_1v_1 + z_2v_2 \in \YY^{\hom}$ of rank $2$, such that $[x, x] \in \hh$. Then we can write 
$c = [x, x] = c_1 + c_2$, where $c_1 = [x_{1}, x_{1}]$ and $c_2 = [x_{2}, x_{2}]$, with 
$x_{i} = u_{i} + z_{i}v_{i}$. In particular, $ic_1$ and $ic_2$ are positive operators, and we have
\[
\DS_{x}(\HH) = \HH^{c_1} \cap \HH^{c_2} = \DS_{x_{2}}(\DS_{x_{1}}(\HH)).
\]
Using this structure inductively, we obtain the following lemma from Theorem~\ref{thm::g_Q_simple}.
\begin{lemma} 
\label{lemm::decomposes_finitely_many}
Let $\HH$ be a unitarizable simple $\gg$-supermodule, and let $x \in \YY^{\text{hom}}$ satisfy 
$\omega(x) = -x$ and $[x,x] \in \hh$. Suppose that $x$ is of rank $k$ and decomposes as 
$x = x_1 + \dots + x_k$ in $\hh$. Then
\[
\DS_{x}(\HH) = \DS_{x_k}(\DS_{x_{k-1}}(\dots \DS_{x_1}(\HH))).
\]
In particular, $\DS_{x}(\HH)$ decomposes into finitely many $\gg_{x}$-supermodules.
\end{lemma}

\subsection{Bogomol'nyi–Prasad–Sommerfield or `BPS'}
\label{subsec::BPSness}

As advertised above, the interplay between self-commutativity and unitarity plays a central r\^ ole in the 
physical constructions of invariants and measures of atypicality. One fairly ubiquitous concept we have not 
yet formalized is the notion of \emph{BPS state}. The underlying intuition is that \emph{multiplet shortening}, 
which as discussed on p.\ \pageref{superCoro} is a consequence of \emph{null vectors} in the Kac supermodule, 
is reflected in the unitarizable simple quotient in non-trivial subspaces that are \emph{annihilated 
by more supercharges than necessary}. To make this more precise, let $M$ be a highest weight $\gg$-supermodule. 
Then, as a consequence of the weight space decomposition $M=\bigoplus_{\lambda\in\hh^*} M^\lambda$ (see Lemma 
\ref{lemm::properties_HWM}), for any vector $v_\lambda \in M^\lambda$, and any odd root $\alpha$, we have that 
(at least) either $Q_\alpha v_\lambda=0$ or $Q_{-\alpha} v_\lambda=0$, where $Q_\alpha$ denotes a generator 
of the root space $\gg^\alpha$. If $M$ is
the standard Verma module, exactly one of these is satisfied for non-zero $v_\lambda$, and any additional 
vanishing is a consequence of atypicality.

\begin{definition}
\label{def::BPSstate}
Let $\HH$ be a unitarizable $\gg$-supermodule. For any $v\in\HH$, we denote by $\ann(v):=\{x\in\gg_1\vert xv=0\}$ 
the \emph{odd annihilator} of $v$. Then a non-zero vector $v \in \HH$ is called a \emph{BPS state} if 
$\dim(\ann(v)) > \frac{\dim \gg_1}{2} = \# \Phi_1^+$.
\end{definition}

\begin{lemma}[\cite{coulembier2017homological}]
\label{lemma::longshort}
Let $\HH$ be a unitarizable simple $\gg$-supermodule.
\begin{enumerate}
\item $v\in\HH$ is a BPS state if and only if there exists an odd positive root $\alpha\in\Phi_1^+$
with non-zero root vector $Q_\alpha$ such that $Q_\alpha v=0$, but 
$v\neq 0 \in \Ker Q_\alpha/\Im Q_\alpha$.
\item If $\HH$ has atypical highest weight $\Lambda$, and $(\Lambda+\rho,\alpha)=0$ for some
odd root $\alpha$, then $\Ker Q_\alpha/\Im Q_{\alpha}\neq 0$ is non-trivial. 
\item $\HH$ is a long supermodultiplet if and only if $\ker Q_\HH = \Im Q_{\HH}$ for all square-zero 
$Q \in \gg_1$.
\end{enumerate}
\end{lemma}

We conclude with the following comparison between various math and physics notions that we have touched 
upon so far.

\begin{proposition}
\label{prop::comparison}
Let $M$ be a unitarizable simple $\gg$-supermodule. Then the following assertions are equivalent:
\begin{enumerate}
\item $\HH$ is a short supermodultiplet.
\item The highest weight of $\HH$ is atypical.
\item $\HH$ contains a BPS state.
\item $\HH$ is protected as a direct summand in families of unitarizable $\gg$-supermodules.
\end{enumerate}
In particular, long unitarizable supermodules do not contain BPS states.
\end{proposition}

One might expect that if $\HH$ contains BPS states, then in particular the highest weight state should be
one. In other words, that $\HH$ is a short supermultiplet if and only if the highest weight state is BPS. 
A quick look at Fig.\ \ref{fig:rootsystem}, which illustrates the weight system of the oscillator supermodule
for $\su(1,1\vert 1)$, shows however that this is not necessarily true and in any case depends on the 
choice of positive system.

In physics, BPS states are further classified by comparing the size of the annihilator with the total
number of supercharges. In other words, if $v\in\HH$ is a BPS state one defines the \emph{(degree of) BPS-ness} by
\begin{equation}
\label{eq::BPSdegree}
\deg_{\rm BPS}(v) = \frac{\dim(\ann(v))- \dim\gg_{+1}}{\dim\gg_{+1}}.
\end{equation}
The intuition is that the more supercharges annihilate a state, the more it is protected against
deformations. Thus, if $\HH$ is a unitarizable simple supermodule, one expects a close relationship 
between the maximal \emph{BPS-ness} of any state in $\HH$ and the degree of atypicality of the 
highest weight (defined as the maximal number of \emph{mutually orthogonal} roots with 
$(\Lambda+\rho,\alpha)=0$). For finite-dimensional supermodules, it is indeed known that the two 
notions are equivalent concepts. We conjecture that this also extends to the infinite-dimensional 
case.

\section{Formal Superdimension}
\label{sec::formal_superdimension}

The aim of this section is to explain the definition of the \emph{formal superdimension}, which is introduced 
more formally in ref.\ \cite{SchmidtGeneralizedSuperdimension}. The idea is to assign an invariant to
unitarizable $\gg$-supermodule that generalizes, on the one hand, the ordinary superdimension 
for finite-dimensional $\gg$-supermodules, and, on the other, what is known as \emph{formal
dimension} or \emph{Harish-Chandra degree} of unitarizable modules over the bosonic subalgebra 
$\even$. In fact, the formal superdimension of a $\gg$-supermodule is simply defined as the alternating 
sum of the formal dimensions of its $\even$-constituents. Such an invariant is expected to be ``protected 
in continuous families'' in a similar sense as atypicality (or BPS-ness) discussed in section \ref{fragrecom}. 
We will realize this expectation in section \ref{sec::anindex}.

Our approach might seem limited by the fact that the formal dimension itself is
defined only for the class of unitarizable $\even$-modules that integrate to 
\emph{relative discrete 
series representations} of a real Lie group $G$ with Lie algebra $\rform$. Here, $\rform$, is the 
reductive real form $\su(p,q) \oplus \su(n) \oplus \u(1)$ of $\even$ (see section \ref{sec::realforms}), 
and $G$ is either the matrix Lie group $\rformg$, or its simply connected covering group by $\sccg$. We 
recall these notions and explain the calculation of the formal dimension in section 
\ref{subsec::relative_discrete_series_formal_dimension}. The definition of the formal 
superdimension follows in section \ref{subsec::superdimension_formula}. It  turns out that all simple 
supermodules for which the formal superdimension is defined and non-zero are maximally atypical, and that
conversely any unitarizable maximally atypical supermodule consists only of $\even$-modules in the (limit of) 
relative holomorphic discrete series. This removes the limitations of the 
formal superdimension to some degree.

\subsection{Relative holomorphic discrete series and formal dimension} 
\label{subsec::relative_discrete_series_formal_dimension} 

We proceed in steps. First, we realize unitarizable highest weight $\even$-modules as 
Harish-Chandra modules of unitary irreducible representations of $G$. Second, we establish the Harish-Chandra condition 
that distinguishes those $\even$-modules that come from $G$-representations that are square-integrable modulo 
the center. Third, we assign a formal dimension to these square-integrable highest weight modules. 
Fourth, we describe the Harish-Chandra parameterization of the relevant highest weights. Fifth, we record the
relation between the formal dimension and the Harish-Chandra character of $L$-packets. Finally, we comment briefly
on some `limit cases`. Our approach closely follows the theory developed in \cite{NeebSquareIntegrable}.

\subsubsection{Unitary representations and Harish-Chandra modules} 

Let $(\pi,\HH)$ be a \emph{unitary representation} of $G$, that is, a complex Hilbert space 
$\bigl(\HH,\langle\cdot,\cdot\rangle\bigr)$
together with a continuous group homomorphism $\pi:G\to U(\HH)$, where $U(\HH)$ is the topological group of 
unitary operators on $\HH$ equipped with the \emph{weak operator topology}.\footnote{%
This is the coarsest topology for which all functions 
\begin{equation}
f_{v,w} : U(\HH) \to \CC, \qquad T \mapsto \langle v, Tw\rangle, \qquad v,w \in \HH
\end{equation}
are continuous. It coincides with the strong operator topology.} A vector $v\in\HH$ is called \emph{smooth} if 
the mapping $G\ni g\mapsto \pi(g)v\in\HH$ is smooth. The subspace $\HH^\infty\subset\HH$ of smooth vectors is
dense and carries a natural action of the Lie algebra $\rform$ by the \emph{derived representation},
\begin{equation} 
\label{conjugation_G}
\mathrm{d}\pi : \rform \to \End(\HH^{\infty}), \qquad \mathrm{d}\pi(X)v := \frac{\mathrm{d}}{\mathrm{d}t} 
\pi(\exp(tX))v.
\end{equation}
Consequently, we may regard $\HH^{\infty}$ as a $\UE(\rform)$-module. Conjugation by elements of $G$ 
in $\End(\HH)$ preserves $\UE(\rform)$, as follows from the functoriality of the exponential map.
\begin{equation} 
\pi(g)\mathrm{d}\pi(X)\pi(g)^{-1} = \mathrm{d}\pi (\Ad(g)(X)). 
\end{equation}

Although $\HH^{\infty}$ has a natural structure, this space is too large to work with effectively in an 
algebraic context. A smaller subspace, which is still dense and a $\UE(\rform)$-module, but more manageable 
for our purposes, can be constructed by decomposing with respect to a maximal compact subgroup 
$K\subset G$, with Lie algebra $\kk\subset\rform$ (see section \ref{sec::realforms}).\footnote{Of some 
importance is the fact that the centralizer in $\gg$ of the center of $\kk$ coincides with $\kk$, meaning 
that $\rform$ is \emph{quasihermitian}.} Formally, we let $\HH^{K}$ be the set of all $v \in 
\HH$ for which $\{\pi(k)v : k \in K\}$ spans a finite-dimensional space. Alternatively, we
can interpret $\HH^{K}$ as the algebraic direct sum of the $K$-isotypic components $\HH_\sigma$
into which $\HH$ decomposes by the Peter-Weyl Theorem. Here, $\sigma$ labels all (equivalences classes 
of) finite-dimensional irreducible representations of $K$. 
By construction, $\HH^K$ consists entirely of smooth vectors and is a $\UE(\rform)$-invariant subspace of 
$\HH^\infty$. As such, it is usually denoted as $\HH^{\infty,K}$. It is at the same time a $\UE(\rform)$-module 
and a representation of $K$ such that the $K$-action preserves $\UE(\rform)$. In other words, $\HH^{\infty,K}$ 
is a $(\rform,K)$-module. 
If $\pi$ is irreducible, then $\dim(\HH_{\sigma}) < \infty$ for any $K$-isotypic component 
$\HH_{\sigma}$, and $\HH^{\infty,K}$ is called the \emph{Harish-Chandra module} underlying $(\pi, 
\HH)$. The Harish-Chandra modules uniquely determine the unitary irreducible representations, see e.g.,
ref.\ \cite{welleda1997general}.

We record that the Harish-Chandra module of an irreducible unitary representation $(\pi,\HH)$ consists 
entirely of analytic vectors. Here, a smooth vector $v \in \HH^{\infty}$ is called \emph{analytic} if 
there exists an $r > 0$ such that the power series 
\begin{equation} 
\label{eq::analytic_vector}
f_{v} : B_{r} \to \HH, \qquad f_{v}(X) = \sum_{k = 0}^{\infty}\frac{1}{k!}\rho(X)^{k}v
\end{equation}
defines a holomorphic function on $B_{r} := \{X \in \even : \lVert X \rVert < r\}$, where $\lVert\cdot\rVert$ 
denotes an $\Ad(H)$-invariant norm on $\even$, $H$ being the Cartan subgroup associated to $\hh$, and we have
extended the derived representation to a representation of the complexification $\even$ on the 
complex vector space $\HH^{\infty,K}$, denoted by the same symbol. For every $r > 0$, 
let $\HH^{a,r}$ denote the set of all analytic vectors for which \eqref{eq::analytic_vector} converges on
$B_{r}$. These spaces are dense in $\HH$.
We are particularly interested in the class of unitary representations $(\pi, \HH)$ for 
which $\HH^{\infty,K}$ is a highest weight module (with respect to a positive system compatible 
with the choice of embedding $K\subset G$. In general, there are other unitary representations
of $\sccg$ with Harish-Chandra module of highest weight with respect to a different positive system. 
We will return to these below.)

\begin{proposition}[{\cite[Theorem 7]{welleda1997general}}]
A simple highest weight $\even$-module $L_{0}(\Lambda)$ is \emph{unitarizable} if and only if there exists 
a unitary representation $(\pi_\Lambda,\HH_\Lambda)$ of $\sccg$ such that $L_0(\Lambda)$ is isomorphic to 
the $\even$-module $\HH^{\infty,K}$ of $K$-finite smooth vectors in $\HH$. In this case, 
$(\pi_\Lambda,\HH_\Lambda)$ is unique and called a \emph{highest weight representation of} $\sccg$.
\end{proposition}

\subsubsection{Relative holomorphic discrete series} 

Let $L_{0}(\Lambda)$ be a unitarizable highest weight $\even$-module with highest weight $\Lambda$, and 
let $(\pi_{\Lambda}, \HH_{\Lambda})$ be the unique unitary irreducible representation of $G$ such that 
$\HH_{\Lambda}^{\infty,K} \cong L_{0}(\Lambda)$. We wish to realize $(\pi_{\Lambda}, \HH_{\Lambda})$ as a 
subspace of square-integrable functions on $Z\backslash G$, where $Z$ denotes the center of $G$. This
is not strictly possible, since the \emph{central character} $\chi_{\pi_{\Lambda}} : Z \to \Uone(1)$, defined
by $\pi_{\Lambda}(z)v = \chi_{\pi_{\Lambda}}(z)v$ for any $z \in Z$ and $v \in \HH_{\Lambda}$ might be 
non-trivial. Instead, we consider the representation\footnote{This can be viewed as space of sections of
a certain line bundle over $Z\backslash G$.}
\begin{equation}
\Gamma_{G}(E_{\pi_{\Lambda}}) := \{ f \in \calC(G, \CC) \ : \ f(zg) = \chi_{\pi_{\Lambda}}(z) f(g) \ 
\text{for all} \ z \in Z \},
\end{equation}
with action $(g \cdot f)(x) = f(xg)$, viewed as a subspace of the right-regular representation with
fixed central character. Since $\rform$ is reductive, $G$ is unimodular, allowing us to fix a $G$-invariant 
Haar measure $\mathrm{d}g$ on $G$, 
which induces a $G$-invariant Haar measure $\mathrm{d}\mu$ on the unimodular group $Z \backslash G$.
Then the space $\Gamma^{2}_{G}(E_{\pi_{\Lambda}})$ of square-integrable elements,
\begin{equation}
\Gamma^{2}_{G}(E_{\pi_{\Lambda}}) := \{ f \in \Gamma_{G}(E) \ : \ \int_{Z \backslash G} \abs{f(g)}^{2} \ 
\mathrm{d}\mu(Zg) < \infty \},
\end{equation}
carries a natural Hermitian form,
\begin{equation}
\langle f, h \rangle := \int_{Z \backslash G} \overline{f(g)} h(g) \ \mathrm{d}\mu(Zg),
\end{equation}
and we denote its Hilbert space completion by the same symbol. The action of $G$ on $\Gamma^{2}_{G}(E_{\pi_{\Lambda}})$ 
then defines a unitary representation of $G$.
 
For fixed $v\in \HH_\Lambda$, $v\neq 0$ the matrix coefficients of $(\pi_{\Lambda}, \HH_{\Lambda})$ 
define a $G$-equivariant map:
\begin{equation}
\label{equiemb}
\Psi_v : \HH_{\Lambda} \to \Gamma_{G}(E_{\pi_{\Lambda}}), 
\qquad 
\Psi_v(w)(g) := \frac{\langle v, \pi_{\Lambda}(g)w \rangle}{\lVert v\rVert }.
\end{equation}
This allows us to compare $(\pi_{\Lambda}, \HH_{\Lambda})$ and $\Gamma_{G}^{2}(E_{\pi_{\Lambda}})$ if the matrix 
coefficients are square-integrable.

\begin{definition} 
\label{def::relholo}
A unitary irreducible highest weight representation $(\pi_{\Lambda},\HH_{\Lambda})$ of $G$ is 
said to belong to the \emph{relative holomorphic 
discrete series} if there exist non-zero $v,w \in \HH_{\Lambda}$ such that the function 
$$
Z \backslash G \to \CC, \qquad Zg \mapsto \vert \langle v, \pi_{\Lambda}(g)w\rangle \vert
$$
is square-integrable. In this case, we also say that $(\pi_{\Lambda},\HH_{\Lambda})$ is a relative holomorphic 
discrete series representation.
\end{definition}

Here, as usual, the epithet ``holomorphic'' refers to the complex structure on the 
quotient $G/K$ that is compatible with the choice of positive system. As such,
relative holomorphic discrete series representations are direct generalizations of holomorphic
discrete series representations of semisimple Lie groups  \cite{harish1954representations} to 
their universal covering groups. Physically, relative holomorphic discrete series
are ``plane wave normalizable in the direction of the center, otherwise normalizable.''

\begin{theorem}[{\cite[Theorem 2.1, Proposition 2.2]{NeebSquareIntegrable}}] 
\label{thm::embedding}
If $(\pi_{\Lambda},\HH_{\Lambda})$ is a relative holomorphic discrete series representation of $G$, there exists a 
(unique) constant $d(\pi_{\Lambda}) > 0$ such that $\sqrt{d(\pi_{\Lambda})}\Psi_v$ is an isometric $G$-equivariant 
embedding $\HH_{\Lambda} \hookrightarrow \Gamma_{G}^{2}(E_{\pi_{\Lambda}})$, for any fixed  non-zero 
$v \in \HH_{\Lambda}$.
\end{theorem}

Harish-Chandra classified in \cite{harish1954representations} all unitary highest weight representation of $G$ that 
belong to the relative holomorphic discrete series by evaluating the integral 
\begin{equation} 
\label{eq::HC_integral}
\int_{Z \backslash G} \vert\langle v_{\Lambda}, \pi_{\Lambda}(g)v_{\Lambda}\rangle\vert^{2} \ \mathrm{d}\mu(Zg),
\end{equation}
where $v_{\Lambda}$ denotes a highest weight vector of the Harish-Chandra module 
$\HH_{\Lambda}^{\infty,K}$. 

\begin{theorem}\cite[Lemma 27, Lemma 29]{harish1954representations} 
\label{thm::HC_condition}
A unitary highest weight representation of $G$ with highest weight $\Lambda$ belongs to the relative holomorphic 
discrete series if and only if, for all (positive non-compact roots) $\beta \in \Phi_{n}^{+}$, we have
$$
(\lambda + \rho_{0}, \beta) < 0.
$$
\end{theorem}

This condition is in the following referred to as the \emph{Harish-Chandra condition}.

\subsubsection{Formal dimension} 
\label{subsec::formal_dimension}

The positive constant $d(\pi_{\Lambda})$ in Theorem \ref{thm::embedding} is independent of the normalization 
of the inner product on $\HH$, as well as the choice of the vector $v$. It can be thought of as a measure of
the ``relative size'' or ``degree'' of the relative holomorphic discrete series representation. This is illustrated
by its appearance in the \emph{Godement--Harish-Chandra orthogonality relations} of matrix 
coefficients (Schur's Lemma).

\begin{proposition}
\label{godement}
Let $(\pi_{1},\HH_{1})$ and $(\pi_{2},\HH_{2})$ be two relative holomorphic discrete series representations 
with the same central character. Then, for all $v_1,w_1\in\HH_1$, $v_2,w_2\in\HH_2$,
$$
\int_{Z \backslash G} 
\overline{\langle v_1,\pi_{1}(g) w_1\rangle}
\langle v_2, \pi_2(g) w_2\rangle \ \mathrm{d}\mu(Zg) = 
\begin{cases}
\displaystyle
\frac{1}{d(\pi_{1})}
\langle v_2,v_1\rangle \langle w_1,w_2\rangle \qquad &\text{if} \ \pi_{1} \cong \pi_{2}, 
\\
0 \qquad &\text{if} \ \pi_{1} \ncong \pi_{2}.
\end{cases}
$$
(Where in the first case a unitary isomorphism is used to identify $\HH_1$ and $\HH_2$)
\end{proposition}

As it stands, $d(\pi_\Lambda)$ depends on the normalization of the Haar measure on $G$. If $G$ were
compact, one would normalize it such that $G$ has volume one. In general, the evaluation of the integral
\eqref{eq::HC_integral} allows for the (geometric) determination of a preferred measure.

\begin{theorem}[{\cite[Theorem 3.17]{NeebSquareIntegrable}}] 
\label{thm::formal_dimension}
With respect to a suitable normalization of the measure $\mathrm{d}\mu$ on $Z \backslash G$, the formal dimension 
$d(\pi_{\Lambda})$ of a relative holomorphic discrete series representation $(\pi_{\Lambda},\HH_{\Lambda})$ is 
\begin{equation}
d(\pi_{\Lambda}) = \prod_{\alpha \in \Phi_{c}^{+}} \frac{(\Lambda+\rho_{c},\alpha)}{(\rho_{c},\alpha)}\prod_{\beta 
\in \Phi_{n}^{+}} \frac{\vert (\Lambda+\rho_{0},\beta)\vert}{(\rho_{0},\beta)}.
\end{equation}
\end{theorem}

The resulting invariant $d(\pi_\Lambda)$ is called \emph{formal dimension} or \emph{formal degree} of the holomorphic
discrete series representation.

\subsubsection{Harish-Chandra parameterization} 

To describe the geometric parameterization of irreducible highest weight modules in the relative holomorphic 
discrete series, let $T$ be the analytic subgroup of $G$ corresponding to the (compact) Cartan subalgebra 
$\hh^{\RR}$. Then $Z \subset T$, and $T/Z$ is compact since $\Ad(G)$ of a reductive 
Lie group is closed. This establishes a relationship between relative holomorphic discrete series representations 
and square-integrable representations of $G/T$, which allows us to parametrize the relative holomorphic discrete series 
in terms of the character group $X^{\ast}(T^{\CC})$ of the complexification $T^{\CC}$ of $T$, and Weyl chambers. 
The coset $X^{\ast}(T^{\CC})+\rho_{0}$ of $X^{\ast}(T^{\CC})\otimes \RR$ is independent of the choice of positive system. 
We call an element $\lambda$ in $X^{\ast}(T^{\CC})\otimes \RR$ \emph{regular} if no dual root is orthogonal to 
$\lambda$; otherwise, $\lambda$ is called \emph{singular}. Let 
$(X^{\ast}(T^{\CC})\otimes \RR)^{\reg}$ denote the set of regular elements in $X^{\ast}(T^{\CC})\otimes \RR$, and set 
$(X^{\ast}(T^{\CC})+\rho_{0})^{\reg} := (X^{\ast}(T^{\CC})+\rho_{0})\cap (X^{\ast}(T^{\CC})\otimes \RR)^{\reg}$. 
A \emph{Weyl chamber} of $G$ is a connected component of $(X^{\ast}(T^{\CC})\otimes \RR)^{reg}$. The Weyl chambers 
are in one-to-one correspondence with the systems of positive roots for $(\even,\hh)$. We define a Weyl chamber $C$ 
to be \emph{holomorphic} if for all $\lambda\in C$ we have $(\lambda,\alpha)<0$ for all $\alpha 
\in \Phi_{n}^{+}$. Then there are exactly $\# W_{c}$ holomorphic Weyl chambers, forming a single orbit for 
the action of $W_{c}$. A \emph{holomorphic Harish-Chandra parameter} is then a pair $(\Lambda := \lambda - \rho_0,C)$, where $C$ is a 
holomorphic Weyl chamber and $\lambda \in (X^{\ast}(T^{\CC})+\rho_{0})^{\reg}\cap C$. There exists a 
bijection between the set of isomorphism classes of relative holomorphic discrete series representations and 
holomorphic Harish-Chandra parameters up to the natural action of the Weyl group $W_c$. 
They correspond to a relatively connected subset of the full set of $\even$-unitarity
described in section \ref{unitarizable}.

\subsubsection{Harish-Chandra characters and \texorpdfstring{$L$}{L}-packets}

If $(\pi,\HH)$ were a finite-dimensional representation of $G$, its dimension could be computed by restricting its
character $\chi:G\to\CC$, $g\mapsto \tr_\HH\pi(g)$ to the identity, $\dim \HH = \chi(e)$. For infinite-dimensional
representations, $\pi(g)$ is generally not trace-class. However, at least for trivial central character, 
\emph{i.e.,} (proper) holomorphic discrete series representations, the formal dimension can be related to a certain
combination of characters, understood in a distributional sense. For this purpose, we replace $Z\backslash G$ 
with $G' := G_{0}' = \SU(p,q) \times \SU(n)$, denote its Haar measure by $\mathrm{d}g$, and let $T'$ be the 
maximal compact Cartan subgroup associated with the real form $\hh^{\RR}$ in $G'$, with complexification 
$(T')^{\CC}$. Furthermore, we let $K' := \SU(p) \times \SU(q) \times \SU(n)$
be the maximal compact subgroup, diagonally embedded in $G'$. We consider the Lie algebra of $G'$, naturally, as 
a subalgebra of $\operatorname{Lie}(G)$, with complexification in $\even$. 

Fix a holomorphic discrete series representation $(\pi_{\Lambda},\mathcal{H}_{\Lambda})$ of $G'$, 
which is uniquely determined by its Harish-Chandra parameter $\Lambda$. 
For a compactly supported smooth function $f \in \mathcal{C}_{c}^{\infty}(G')$, the operator  
\begin{equation}
\pi_{\Lambda}(f)v = \int_{G'} f(g) \pi_{\Lambda}(g)v \, \mathrm{d}g,
\end{equation}  
is trace-class. The mapping $f \mapsto \tr(\pi_{\Lambda}(f))$ defines a continuous linear functional on 
$\mathcal{C}_{c}^{\infty}(G')$. The associated Schwartz distribution,  
\begin{equation}
\Theta(f) := \tr(\pi_{\Lambda}(f)),
\end{equation}  
is referred to as the \emph{Harish-Chandra character} of $\pi_{\Lambda}$. This uniquely determines 
$\pi_{\Lambda}$ up to unitary equivalence \cite{Harish2}. Moreover, Harish-Chandra demonstrated in 
\cite{Harish2} the existence of a locally integrable function 
$\Theta_{\pi_{\Lambda}} : G' \to \mathbb{C}$ such that 
\begin{equation}
\Theta(f) = \int_{G'} f(g) \Theta_{\pi_{\Lambda}}(g) \, \mathrm{d}g,
\end{equation}
where $\Theta_{\pi_{\Lambda}}$ is real analytic on the set of regular elements $G_{\reg}'\subset G'$. 
When referring to the calculation of a representation's character, we typically mean evaluating 
$\Theta_{\pi_{\Lambda}}\big|_{G_{\reg}'}$.

As it stands, $\Theta_{\pi_\Lambda}$ cannot be evaluated at the identity. It becomes continuous, however,
after summing over the so-called \emph{$L$-packet} of $\pi_\Lambda$, which consists of those discrete series 
representations of $G'$ with the same infinitesimal character as $\pi_\Lambda$. These
are proper irreducible subrepresentations of $L^2(G')$ whose Harish-Chandra module is of highest weight 
with respect to a different positive system, and correspond bijectively to Harish-Chandra parameters
in $(X^{\ast}(T^{\CC}) + \rho_{0})^{\reg}/W_{c}$ that are in the same orbit of the full Weyl group $W$ of
$G'$ as the holomorphic parameter $\Lambda$. They are uniquely determined by their Harish-Chandra 
character, denoted $\Theta_{w\Lambda}$ for $w\in W/W_c$, and can also be assigned a formal dimension,
denoted $d_{\pi_{w\Lambda}}$. As it turns out, the combination 
\begin{align}
\tilde{\Theta} = \sum_{w \in W/W_{c}}\Theta_{\pi_{w\lambda}}.
\end{align}
is in fact continuous at the identity, and we have

\begin{proposition}[{\cite{hochs2019orbital}}] 
\label{prop::limit_l_packet}
The following assertion holds: 
$$
 d(\pi_{\Lambda}) = \lim_{\topa{g \to e_{G'}}{g \in T' \cap G_{\reg}'}} \tilde{\Theta}(g) 
$$
Moreover, $d(\pi_{\Lambda}) = d(\pi_{w\Lambda})$ for all $w \in W/W_{c}$. 
\end{proposition}

\subsubsection{Limit of relative holomorphic discrete series} 

A unitary highest weight representation $(\pi_{\Lambda}, \HH_{\Lambda})$ belongs to the relative holomorphic discrete 
series if $(\Lambda+\rho_0,\beta) < 0$ for all $\beta \in \Phi_{n}^{+}$, according to Harish-Chandra's condition. 
One says that $(\pi_{\Lambda},\HH_{\Lambda})$ belongs to the \emph{limit of the relative holomorphic discrete series} 
if $(\Lambda+\rho_0,\beta) = 0$ for some $\beta \in \Phi_{n}^{+}$. Unfortunately, unlike the case of the relative 
holomorphic discrete series, there is no simple intrinsic definition for limits of relative holomorphic discrete 
series. Explicit constructions are provided by Harish-Chandra's character formula and via Zuckerman--Jantzen tensoring.
Highest weights of representations in the limit of the relative holomorphic series are in bijection with $W_c$-orbits 
of \emph{irregular holomorphic Harish-Chandra parameters}. These are pairs $(\Lambda := \lambda - \rho_0,C)$ 
consisting of a holomorphic Weyl chamber $C$ and a weight $\lambda \in\bigl((X^{\ast}(T^{\CC})+\rho_{0})\setminus 
(X^{\ast}(T^{\CC})+\rho_{0})^{\reg} \bigr)\cap C$, with the additional condition that $\lambda$ is not orthogonal 
to any simple compact dual root. Finally, there exist unitary highest weight representations ``beyond 
relative holomorphic discrete series'', whose highest weights satisfy $(\Lambda+\rho_0,\beta)>0$ for some 
$\beta\in\Phi_n^+$. These form a bounded subset of the full set of $\even$-unitarity in the language of 
section \ref{unitarizable}.

\subsection{A superdimension formula} 
\label{subsec::superdimension_formula} 

The proposal for the formal superdimension \cite{SchmidtGeneralizedSuperdimension} arises from 
the natural combination of the Harish-Chandra degree for 
$\even$-modules from section \ref{subsec::formal_dimension}  with the insights from sections \ref{subsec::Dirac} 
and \ref{subsec::Classification}. To fix notation, we recall that if $\HH$ is a non-trivial unitarizable simple 
$\gg$-supermodule, then
\begin{enumerate}
\item $\HH$ is a simple highest weight supermodule with some highest weight $\Lambda \in \Gamma$, the set of 
$\gg$-unitarity described in section \ref{subsec::Classification};
\item $\HH$ is $\even$-semisimple, with each $\even$-constituent $L_{0}(\Lambda_j)$ being a unitarizable highest 
weight $\even$-supermodule concentrated in degree $p(\Lambda_j)=\sum_{k=1}^{n}(\Lambda_j-\Lambda,\delta_{k}) 
\mod 2$ relative to the highest weight vector. Moreover, each $\even$-constituent is a tensor product of a
unitarizable simple $\su(p,q)$-module (which is finite-dimensional iff $p=0$ or $q=0$), a (finite-dimensional) 
simple $\su(n)$-module and a (one-dimensional) simple $\u(1)$-module. 
\item If $p,q\neq 0$, the $\even$-constituents $L_{0}(\Lambda_j)$ belong to the relative holomorphic discrete series 
precisely is they satisfy the Harish-Chandra condition of Theorem \ref{thm::HC_condition},
\begin{equation*}
(\Lambda_j+\rho_{0}, \beta) < 0, \qquad \text{for all} \ \beta \in \Phi_{n}^{+}.
\end{equation*}
In this case they possess a formal dimension $d(\Lambda_j)\in\RR_{+}$ given by Theorem \ref{thm::formal_dimension},
\begin{equation*}
d(\Lambda_j) = \prod_{\alpha \in \Phi_{c}^{+}} \frac{(\Lambda_j+\rho_{c},\alpha)}{(\rho_{c},\alpha)}\prod_{\beta 
\in \Phi_{n}^{+}} \frac{\vert (\Lambda_j+\rho_{0},\beta)\vert}{(\rho_{0},\beta)}.
\end{equation*}
\end{enumerate}

All we have to do is weight by the parity $(-1)^{p(\Lambda_j)}$.

\begin{definition} 
\label{def::superdim}
Let $\HH$ be a unitarizable highest weight $\gg$-supermodule with highest weight $\Lambda$ such that every 
$\even$-constituent $L_{0}(\Lambda_j)$ belongs to the relative holomorphic discrete series of $\sccg$. Then 
the real value 
\begin{equation}
\sdim(\HH) := \sum_{j} \sdim(L_{0}(\Lambda_j)), \qquad \sdim(L_{0}(\Lambda_j)) := (-1)^{\HH+p(\Lambda_j)} 
d(\Lambda_j),
\end{equation}
where the sum runs over the highest weights of all $\even$-constituents $L_{0}(\Lambda_j)$ of $\HH$, is called 
the \emph{formal superdimension} of $\HH$.
\end{definition}

In the following, we will call a unitarizable highest weight $\gg$-supermodule $\HH$ a \emph{(relative)
holomorphic discrete series supermodule} if each $\even$-constituent is a Harish-Chandra supermodule corresponding 
to a (relative) holomorphic discrete series supermodule of $\rformg$ (or $\sccg$, respectively). An explicit 
formula for the superdimension of all these supermodules is provided by Theorem \ref{UnitarityD2} together
with Theorem \ref{thm::formal_dimension}. We give 
here two important properties of the formal superdimension, referring to
\cite{SchmidtGeneralizedSuperdimension} for a complete discussion.

\begin{theorem}
\label{thm::properties_superdimension} 
\begin{enumerate}
\item The superdimension is additive.
\item The superdimension of a relative holomorphic discrete series $\gg$-supermodule $\HH$ is trivial unless the 
highest weight is maximally atypical.
\end{enumerate}
\end{theorem}
The above theorem reduces the consideration of the superdimension to the study of unitarizable simple $\gg$-supermodules 
with maximally atypical highest weight. A direct calculation yields that these consist of $\even$-supermodules 
that belong to the relative holomorphic discrete series or ``limit'' thereof. In fact, if $\gg \neq \mathfrak{sl}(1|n)$, 
all of these have integral highest weight, and therefore trivial central character, so that their
formal dimension can be calculated via Proposition \ref{prop::limit_l_packet}.

\section{Indices for Unitarizable \texorpdfstring{$\slmn$}{sl(m|n)}-Supermodules} 
\label{sec::anindex}

We have reviewed in previous sections that the representation theory of $\gg=\slmn$ with real form 
$\gg^\omega=\supqn$ shares essential features with that of superconformal algebras such as 
$\su(2,2\vert \mathcal N)$, $1\le\mathcal N\le 4$, 
which are of great interest in mathematical physics.\footnote{We expect many of these features to be
present more generally for all basic classical Lie superalgebras.} $\omega$-unitarizable $\gg$-supermodules 
are completely 
reducible. Simple constituents are highest weight modules $L(\Lambda)$ parameterized up to parity reversal 
by $\Lambda\in\Gamma\subset\hh^*$, the ``set of $\gg$-unitarity'' in weight space. The relative interior of
$\Gamma$ is the ``region of $\gg$-unitarity'', $\calC$, that parametrizes all unitarizable highest weight
module with typical highest weight, also known as ``long supermultiplets''. When, in the decomposition of
a \emph{continuous family} of unitarizable $\gg$-supermodules, the highest weight of one of the typical 
constituents ``hits the unitarity bound'', $\partial\calC$, it fragments into several ``short supermultiplets'' 
of atypical 
highest weight. This is signaled by the appearance of BPS states, non-trivial cohomology of some element 
$Q\in\YY$ of the self-commuting variety. Upon further deformation, these can disappear only in 
corresponding combinations. ``Ultra-short'' supermultiplets, with highest weight of higher atypicality, are 
``absolutely protected'' and must always persist under deformation. 

In this context, an \emph{index} is a quantitative measure of those components of a unitarizable $\gg$-supermodule
that are invariant under continuous deformations. The physics literature, following 
\cite{Minwalla,romelsberger2006counting}, shows that any invariant count of short supermultiplets, which we will
call a \emph{KMMR index}, can be recovered from the knowledge of the \emph{superconformal index}, which is defined as
a regularized Witten index \cite{WittenIndex} of any possible $Q\in\YY$. We will here formalize these two notions, 
and establish their equivalence, for general unitarizable $\slmn$-supermodules. Moreover, we show that the formal 
superdimension, discussed in section \ref{sec::formal_superdimension}, is also an index in the sense of this 
definition, and explain its relation to the KMMR and Witten index.

\subsection{KMMR indices} 
\label{subsec::KMMR_index}

For simplicity, we will restrict to unitarizable $\gg$-supermodules that decompose completely into
a countable direct sum of simple $\gg$-supermodules, \emph{cf.} section \ref{subsec::assumptions}. In fact, 
in order to define a ``counting index'', we need to assume that the number of atypical simple $\gg$-supermodules 
in the decomposition is finite. We will denote this category as $\ugsmodp$, and as before suppress the inner product
$\langle\cdot,\cdot\rangle$ and structure map $\gg\to \End(M)$ from the notation for its objects $M\in\ugsmodp$.
We will also restrict the notion of ``continuous family of unitarizable $\gg$-supermodules'' accordingly, to 
mean a continuous map $T:\mathds{H}\to\Homp(\gg^\RR,\uu(\HH))$ into the space of $\omega$-unitarizable 
$\gg$-supermodule structures on a fixed Hilbert space that decompose into a direct sum of simple supermodules, 
finitely many of which are atypical. The topology should be sufficiently fine to 
make the fragmentation process described in section \ref{fragrecom} continuous. It might be interesting to 
observe that in such a topology, $\Homp(\gg^\RR,\uu(\HH))$ will not be closed inside $\Hom(\gg^\RR,\uu(\HH))$. 
There is an obvious tautological map $\Homp(\gg^\RR,\uu(\HH)) \to \ugsmodp$, which by a healthy abuse of notation 
we will indicate as $(\rho,\HH)\mapsto \HH$.

\begin{definition}
\label{def::KMMR}
A \emph{Kinney--Maldacena--Minwalla--Raju index} \cite{kinney2007index}, or \emph{KMMR index} for short, is a
map $I:\ugsmodp\to \ZZ$, $M\mapsto I(M)$ that is
\begin{enumerate}
    \item additive, i.e., $I(M_1\oplus M_2) = I(M_1)+ I(M_2)$, and
    \item such that the induced map $\Homp(\gg^\RR,\uu(\HH))\to\ZZ$, $(\rho,\HH)\mapsto I(\HH)$ is continuous.
\end{enumerate}
\end{definition}

\begin{lemma}
\label{lemm::properties_KMMR}
An additive map $I: \ugsmodp \to \ZZ$ is specified uniquely by its values on simple unitarizable supermodules 
$L(\Lambda)$ for all $\Lambda\in\Gamma$. It is a KMMR index if and only if
\begin{enumerate}
\item $I(K(\Lambda)) = 0 $ for all typical weights $\Lambda\in \calC$.
\item For all weights $\Lambda_0\in\partial\calC$ at the unitarity bound, $\sum_{i=1}^M I(L(\Lambda_i)) = 0$,
where $\gr K(\Lambda_0) = \bigoplus_{i=1}^M L(\Lambda_i)$ is the fragmentation of the Kac supermodule.
\end{enumerate}
\end{lemma}

\begin{proof} Existence and uniqueness are obvious. For a), it suffices to notice that if 
$\Lambda_*\in\calC$ is any typical weight, we have 
\begin{equation}
I\Bigl(\bigoplus_{i=1}^\infty K(\Lambda_*)\Bigr) = 
I\Bigl(\bigoplus_{i=2}^\infty K(\Lambda_*)\Bigr) + I(K(\Lambda_*))
=
I\Bigl(\bigoplus_{i=1}^\infty K(\Lambda_*)\Bigr) + I(K(\Lambda_*)).
\end{equation}
Statement b) is equivalent to the continuity of the fragmentation process described in section \ref{fragrecom}, 
see eqs.\ \eqref{assGraded} and \eqref{fragDef}.
\end{proof}

\begin{lemma}
\label{KMMRModule}
The set of all KMMR indices forms naturally a finitely generated free $\ZZ$-module.\footnote{In
situations such as considered in \cite{kinney2007index}, only supermodules respecting spin-statistics 
are physically relevant. This can reduce $\mathcal{I}$ correspondingly.}
We denote it by $\mathcal{I}$, and allow extending scalars as
\begin{equation}
    \calI_{\QQ} := \calI \otimes_{\ZZ} \QQ, \qquad \calI_{\RR} := \calI \otimes_{\ZZ} \RR.
\end{equation}
\end{lemma}

This identification will be convenient for the comparison with supertraces and superdimensions,
to which we now turn.

\subsection{The character-valued Witten index} 
\label{wittenelab}

For any element $Q$ of the self-commuting variety $\YY$ acting on a fixed (non-trivial) unitarizable 
$\gg$-supermodule $\HH\in \ugsmodp$, its adjoint with respect to the Hermitian inner product 
$\langle\cdot,\cdot\rangle$ is related to its conjugate under the (graded) anti-involution $\omega$ 
defining the real form $\supqn$ of $\gg$ via 
$Q^\dagger=-i\omega(Q)$ (see Definition \eqref{unitarize} and eq.\ \eqref{amendagain}). As a consequence,
$x:= e^{3\pi i/4}(Q+Q^\dagger)$ satisfies $\omega(x)=-x$ and $c:=x^2 := \frac 12 [x,x] $ is semi-simple on
account of $\omega(c)=c^\dagger=-c$. In the notation of section \ref{subsec::DS_functor_and_unitarity}, 
$x\in\bigl(\YY^{\rm hom}\bigr)^\omega$, with $\rk(x)=\rk(Q)$.

\begin{lemma} 
\label{lemm::properties_Xi}
Let $\HH$ be a unitarizable $\gg$-supermodule, $Q\in\YY$. 
\begin{enumerate}
\item $\Xi=i c= [Q,Q^\dagger] $ is a positive operator, \emph{i.e.}, $\langle v, \Xi v\rangle \geq 0$ for all 
$v \in \HH$.
\item $\Xi$ is self-adjoint. In particular, $\HH$ decomposes completely in eigenspaces for $\Xi$,
$$
\HH = \bigoplus_{\xi} \HH(\xi), \qquad \HH(\xi) := \{ v \in \HH : \Xi v = \xi v\},
$$
and each eigenvalue $\xi$ is a non-negative real number.
\item There exists a Cartan subalgebra $\tt$ of $\gg$ with $\Xi \in \tt$, and for any such Cartan subalgebra,
the $\tt$-weight spaces of $\HH$ are eigenspaces of $\Xi$.
\end{enumerate}
\end{lemma}
\begin{proof} 
The contravariance of the Hermitian form implies the following for all $v \in \HH$:
\begin{equation}
\langle v, \Xi v \rangle = \langle v,  [Q,Q^{\dagger}]v \rangle = \langle v, QQ^{\dagger}v \rangle + \langle v, 
Q^{\dagger}Qv \rangle = \langle Qv, Qv \rangle + \langle Q^{\dagger}v, Q^{\dagger} v \rangle \geq 0,
\end{equation}
where we used $(Q^{\dagger})^{\dagger} = Q$, and the last inequality follows from the positive definiteness of 
the Hermitian form. This implies a). b) follows from $\omega(c)=-c$, together with the assumption that
$\HH$ decomposes discretely under $\gg$ and the spectral theorem. c) is then obvious.
\end{proof}

In physics, the positivity of $\Xi$ is referred to as a \emph{BPS-bound associated to the supercharge $Q$}, and 
one expects to study the \emph{BPS states} (elements of $\HH(0)$, see section \ref{subsec::BPSness}) by considering 
the supertrace of the operator $e^{-\beta\Xi}$ for positive real $\beta$ \cite{WittenIndex}. The formula 
\begin{equation}
e^{-\beta\Xi}v:= \sum_{k=0}^{\infty} \frac{(-\beta)^{k}}{k!} \Xi^{k}v, \qquad v \in \HH.
\end{equation}
for this operator is well-defined by \eqref{eq::analytic_vector}, as $\Xi$ is an even operator and $\HH$ decomposes
completely in unitarizable simple highest weight $\even$-supermodules. The supertrace $\str_{\HH}(e^{-\beta \Xi})$ 
however, is not. A mathematically clever but physically naive assumption for a remedy would be that $e^{-\beta \Xi}$ 
is trace class (it is obviously bounded), but the convergence of the trace requires that the eigenspaces of $\Xi$ 
be finite-dimensional. This assumption is, however, too strong, although it might hold for some special 
situations (such as the example of $\su(1,1\vert 1)$ from section \ref{sec::SQM}). Instead, one resolves the 
infinite degeneracy of $\HH(0)$ by considering a \emph{character-valued Witten index}. To clarify this notion, 
we let $\tt$ be a Cartan subalgebra of $\gg$ with $\Xi\in\tt$, and recall from our discussion of the 
$\omega$-compatible Duflo-Serganova functor in section \ref{subsec::DS_functor_and_unitarity} that 
$\gg_{x} = \DS_x(\gg) \cong \sl(m-k \vert n-k)$ if $\rk(Q) = k$ with real form $\gg_{x}^{\omega_x}  = 
\su(p-r,q-s\vert n-k)$ for some $r\le p$, $s\le q$ with $r+s=k$. Viewing $\gg_{x}$ as a Lie subsuperalgebra of
$\gg$, the Cartan subalgebra $\tt$ induces a Cartan subalgebra $\tt_{x}$ of $\gg_{x}$. We denote the associated 
analytic complex subgroup by $T_{x} \subset G_0$. The set of regular elements in $\tt_{x}$ is denoted by 
$\tt_{x}^{\reg}$, while the set of regular elements in $T_{x}$ is denoted by $T_{x}^{\reg}$. Note that 
$T_{x}^{\reg} \subset T_{x}$ is open and dense. Linear coordinates on $\tt_{x}$ are referred to in physics 
as \emph{fugacities}. We also fix a positive root system $\Phi_x^+$ for $\gg_x$ (\emph{cf.} eq.\ 
\eqref{rootsysx}).

\begin{proposition} 
\label{lemm::trace_class_twist} 
Let $V$ be a unitarizable highest weight $\gg_{x}$-supermodule. Assume $X \in \tt_{x}^{\reg}$ satisfies 
$\alpha(X) > 0$ for all $\alpha \in \Phi^{+}_x$. Then $e^{X}$ is trace class and 
$$
\str_{V} e^{X} \in \RR.
$$
\end{proposition}

\begin{proof}
Let $\Lambda \in \tt_{x}^{\ast}$ denote the highest weight of $V$. We consider the triangular decomposition 
$\gg_{x} = \nn_{x}^{-} \oplus \tt_{x} \oplus \nn_{x}^{+}$, where $\nn_{x}^{\pm} := \sum_{\alpha \in 
\Phi_{x}^{\pm}}\gg_{x}^{\alpha}$. Then, as a $\tt_{x}$-module, the supermodule $V$ is a quotient of 
the $\tt_{x}$-module 
$$
\UE(\nn^{-}_{x}) \otimes \CC_{\Lambda} \cong \twedge(\nn_{x,0}^{-}) \otimes 
S(\nn_{x,1}^{-}) \cong \calS(\nn_{x}^{-}) \otimes \CC_{\Lambda}.
$$
Here, $\calS(\nn_{x}^{-})$ is the symmetric superalgebra over $\nn_{x}^{-}$.
Now, it is enough to show that $e^{X}$ is trace class on $\UE(\nn^{-}_{x}) \otimes \CC_{\Lambda}$, as quotients 
have smaller multiplicities. For that, we consider the following decomposition of the $\tt_x$-module 
$\calS(\nn_{x}^{-})$:
$$
\calS(\nn_{x}^{-}) \cong \bigotimes_{\alpha \in \Phi_{x}^{+}} \calS(\gg^{-\alpha}) \cong \bigotimes_{\alpha \in 
\Phi_{x}^{+}} \calS(\CC_{-\alpha}).
$$
The trace of $e^{X}$ on any $\calS(\CC_{-\alpha})$ is 
$$
\tr_{\calS(\CC_{-\alpha})} e^{X} = 
\begin{cases}
\sum_{k = 0}^{\infty}e^{-k\alpha(X)} = \frac{1}{1-e^{-\alpha(X)}}, \quad &\alpha \in \Phi_{x,0}, \\
1 + e^{-\alpha(X)}, \ &\alpha \in \Phi_{x,1}.
\end{cases}
$$
We conclude
$$
\tr_{\calS(\nn_{x}^{-}) \otimes \CC_{\Lambda}} e^{X} = e^{\Lambda(X)} 
\prod_{\alpha \in \Phi_{x,1}^{+}}(1+e^{-\alpha(X)})
\prod_{\alpha \in \Phi_{x,0}} \frac{1}{1-e^{-\alpha(X)}} < \infty,
$$
\emph{i.e.}, $e^{X}$ is trace class on $V$. With respect to the weight space decomposition 
$V=\bigoplus_{\lambda\in P_V} V^\lambda$ for some countable set $P_V$, the supertrace
$$
\str_{V} e^{X}= (-1)^V e^{\Lambda(X)} 
\sum_{\lambda \in P_{V}}(-1)^{\lambda-\Lambda} m(\lambda)e^{-\Lambda(X)+\lambda(X)}, 
$$
where $m(\lambda)=\dim V^\lambda$, is dominated by $\tr_V e^X$, and real.
\end{proof}

Proposition \ref{lemm::trace_class_twist} defines for any unitarizable highest weight $\gg_{x}$-supermodule 
$V$ a conjugation invariant continuous function on 
$T_{x}^{\reg, +} := \{e^{X} \in T_{x}^{\reg} : \alpha(X) > 0 \ \forall \alpha \in \Phi^{+}_x\}$,
\begin{equation}
\label{supercharacters}
\chi^{x}_{V} : T_{x}^{\reg, +} \to \RR, \qquad 
\chi^x_V (e^X) := \str_V e^X
\end{equation}
which we refer to as the \emph{supercharacter} of $V$. The set of supercharacters forms a ring, the 
\emph{supercharacter ring} $X^{\ast}(T_{x}^{\reg,+})$. This motivates the following definition.

\begin{definition} 
\label{superWitten}
Let $M$ be a unitarizable highest weight $\gg$-supermodule, and let $Q$ be an element of the self-commuting 
variety. Using the decomposition $\DS_{x}(M) = \bigoplus_{i}V_{i}$ into finitely many $\gg_x$-supermodules
(see Lemma \ref{lemm::decomposes_finitely_many}), we define the $Q$-\emph{Witten index} of $M$ as the 
$\gg_x$-supercharacter
$$
I_{M}^{W}(Q, \cdot) := \sum_{i} \chi_{V_{i}}^{x}(\cdot) = \str_{\DS_{x}(M)}(\cdot) \in 
X^{\ast}(T_{x}^{\reg,+}). 
$$
\end{definition}

The fact that the Witten index detects short (or protected) unitarizable supermodules is reflected
in the following statement.

\begin{lemma} 
\label{lemm::QWitten_protected}
Let $M$ be a unitarizable highest weight $\gg$-supermodule with highest weight $\Lambda$, and $Q\in\YY$. 
Assume $\at(\Lambda) < \rk(Q)$. Then $I^{W}_{M}(Q) = 0$. In particular, if $\Lambda \in \calC$, then 
$I^{W}_{M}(Q) = 0$ for any $Q \in \YY$.
\end{lemma}
\begin{proof}
This follows at once from Corollary \ref{DSNM}.
\end{proof}

Lemma \ref{lemm::QWitten_protected} allows us to extend the $Q$-Witten index additively to any unitarizable 
supermodule 
$\HH\in\ugsmodp$ with finitely many atypical simple components in a well-defined fashion. We can then relate
this definition to the formulation as a supertrace over the entire supermodule, which is standard in the 
physics literature (such as specifically for $\psu(2,2\vert 4)$ in \cite{kinney2007index}).

\begin{proposition} 
\label{prop::QWitten_à_la_physik}
Let $\HH\in \ugsmodp$, and $Q\in\YY$. Then the $Q$-\emph{Witten index} of $\HH$ satisfies
$$
I^{W}_{\HH}(Q,X) = \str_{\HH}e^{-\beta \Xi + X}
$$
for any  $X \in \tt_{x}^{\reg,+} = \{X \in \tt_{x}^{\reg} : 
\alpha(X) > 0 \ \forall \alpha \in \Phi^{+}_x\}$, and any positive $\beta$.
\end{proposition}

\begin{proof}
With respect to the decomposition in Lemma \ref{lemm::properties_Xi} (b), the supertrace 
vanishes outside of the zero eigenspace $\HH(0)=\ker Q \cap \ker Q^{\dagger}$ of $\Xi$, which is identified 
with the DS cohomology $\DS_{x}(\HH)$ according to Theorem \ref{thm::unitarity_DS}.
This follows from a standard argument that is akin to the Hodge decomposition, and mostly formal under our 
assumptions. Note that $[X,\Xi]=0$. On $\HH(0)$, the supertrace is well-defined by Lemma 
\ref{lemm::trace_class_twist}, and equal to the Witten index by definition.
Note that in particular, the supertrace is independent of $\beta$ as advertised.
\end{proof}

It is interesting to observe that the supertrace formula for the Witten index could be extended to a larger class
of unitarizable $\gg$-supermodules, provided the sum of supercharacters over atypical components converges in the
appropriate topology. We finish with the equivalence between the $Q$-Witten index and the KMMR index.

\begin{proposition}
\label{wittentoKMMR}
For any $Q\in\YY$, $X \in \tt_{x}^{\reg,+}$, the assignment $\ugsmodp\to \RR$, 
$M\mapsto I_{M}^{W}(Q,X)$ is a real KMMR index. 
\end{proposition}

\begin{proof}
The vanishing on typical components following from Lemma \ref{lemm::QWitten_protected}, we explain ``continuity in
the $\gg$-supermodule structure on $\HH$'', \emph{i.e.,} of the Witten index evaluated at $X\in \tt_x^{\reg,+}$, 
viewed as a function
$$
I^{W}_{\HH}(Q,X) : \Hom'(\gg^{\RR}, \mathfrak{u}(\HH)) \to \RR
$$
Topologically, $\Hom'(\gg^{\RR}, \mathfrak{u}(\HH))$ is $(\hh^{\ast})^{\ZZ_{+}}/S_{\infty}$, while  $\Hom'(\gg^{\RR}_{x}, 
\mathfrak{u}(\DS_{x}(\HH))) \cong (\tt_{x}^{\ast})^{\ZZ_{+}}/S_{\infty}$ if $\Xi \in \tt$ (see section 
\ref{subsec::assumptions}). Focusing on a single fragmentation process starting from $\Lambda\in\calC$, 
we have a convergent sum
$$
\str_{K(\Lambda)} e^{-\beta \Xi + X} = \sum_{\lambda \in P_{K(\Lambda)} \subset \hh^{\ast}} 
(-1)^{\lambda-\Lambda} m(\lambda) e^{\lambda(-\beta\Xi + X)}.
$$
This is clearly continuous on $\Hom'(\gg^{\RR},\mathfrak{u}(K(\Lambda))$, where $K(\Lambda)$ is viewed as a fixed 
Hilbert space (\emph{cf.} section \ref{fragrecom}). This finishes the proof.
\end{proof}

\begin{theorem} 
\label{thm::KMMR_gleich_Witten}
Any KMMR index is a real linear combination of Witten indices. 
\end{theorem}

\begin{proof}
Let $I : \ugsmodp \to \ZZ$ denote a KMMR index. By our standing assumptions and the additivity 
of KMMR and Witten indices, it suffices to prove that for any unitarizable simple $\gg$-supermodule 
$L(\Lambda)$ with $I(L(\Lambda)) \neq 0$, there exists a $Q$-Witten index such that
$$
I(L(\Lambda)) = c_{X, \Lambda} \cdot I^{W}_{L(\Lambda)}(Q, X).
$$
for some $X \in \tt_{x}^{\reg,+}$, $c_{X, \Lambda} \in \RR$.
By Lemma~\ref{lemm::properties_KMMR}, we know that $\Lambda$ is atypical; that is, there exists an odd root $\alpha 
\in \Phi_{1}^{+}$ such that $(\Lambda + \rho, \alpha) = 0$. Define $Q$ to be an associated root vector, and set 
$x := e^{3\pi i/4}(Q + Q^{\dagger})$ so that
$\Xi = \frac{i}{2}[x,x] \in \hh$. Then $\DS_{x}(L(\Lambda))$ decomposes into either 
one or two unitarizable $\gg_{x}$-supermodules (see Theorem~\ref{thm::g_Q_simple}).

The claim follows if we can show that $I^{W}_{L(\Lambda)}(Q, X) \neq 0$ for at least one $X \in \tt_{x}^{\reg,+}$. 
However, $I_{L(\Lambda)}^{W}(Q,\cdot)$ is either the supercharacter of a nontrivial unitarizable simple 
$\gg_{x}$-supermodule, or the sum of two nontrivial and non-equivalent unitarizable simple $\gg_{x}$-supermodules. 
In either case, the supercharacter is nonzero. This completes the proof.
\end{proof}

\subsection{Formal superdimension and Witten index} 
\label{formalwitten}

In this section, we establish the connection between the Witten index and the formal superdimension for
(relative) holomorphic 
discrete series $\gg$-supermodules in the case where $m, n \geq 3$ (see Lemma 
\ref{cor::maximal_protected_gleich_absolutely_protected} and Theorem \ref{thm::properties_superdimension}).
To begin, we record that the formal superdimension $\sdim(\cdot)$ also serves as a KMMR index on the subspace 
of all relative holomorphic discrete series $\gg$-supermodules. This gives rise to a new KMMR index that, to our 
knowledge, has not previously appeared in the literature.

\begin{proposition}
If $m,n \geq 3$, the formal superdimension is an element of $\calI_{\RR}$ on the subspace of relative holomorphic 
discrete series $\gg$-supermodules.
\end{proposition}

\begin{proof}
By Theorem \ref{thm::properties_superdimension}, the superdimension is trivial on all unitarizable simple 
$\gg$-supermodules except the maximally atypical ones. These do not belong to $\overline{\calC}$ by Corollary 
\ref{cor::maximal_protected_gleich_absolutely_protected} and are isolated. Consequently, $\sdim(\cdot)$ is 
continuous on the subspace of relative holomorphic discrete series $\gg$-supermodules.
\end{proof}

Next, one may motivate the search for a relation between the $Q$-Witten index and the formal
superdimension from the formula (Proposition \ref{prop::QWitten_à_la_physik})  
\begin{equation}
I^{W}_{\HH}(Q,X) = \str_{\HH} e^{-\beta \Xi + X}, \qquad \text{(where $Q\in\YY$, $\Xi = [Q,Q^\dagger]$,
$X \in \tt_{x}^{\reg,+}$)},
\end{equation}
which bears a strong resemblance to the Weyl character formula for finite-dimensional representations of 
Lie groups, where taking the limit $X \to 0$ recovers the dimension of the representation. For 
infinite-dimensional representations, this will work best in the framework of Harish-Chandar characters and $L$-packets,
as reviewed in section \ref{subsec::relative_discrete_series_formal_dimension}.

If $M$ is a relative holomorphic discrete series $\gg$-supermodule with highest weight $\Lambda$, and $Q,x$ as
before, by Lemma \ref{lemm::decomposes_finitely_many}, we can decompose its DS-twist
\begin{align}
\DS_{x}(M) = \bigoplus_{i} L_x(\Lambda_i)\,, \qquad \Lambda_{i} \in \tt_{x}^{\ast}
\end{align}
into finitely many relative holomorphic discrete series $\gg_x$-supermodules
($L_x(\Lambda_i)=V_i$ in the notation of section \ref{wittenelab}). It turn, we write
\begin{align}
L_x(\Lambda_i)_{\ev} = \bigoplus_{j} L_{x,0}(\Lambda_{i;j})
\end{align}
for the decomposition of the $L_x(\Lambda_i)$ under $\gg_{x,0}$. If $X \in \tt_{x}^{\reg,+}$,
by Proposition \ref{lemm::trace_class_twist} the operator $e^{X}$ is trace class on any $L_x(\Lambda_{i})$, 
and hence on any $L_{x,0}(\Lambda_{i; j})$, such that we can express the trace of $e^{X}$ on any 
$L_{x,0}(\Lambda_{i; j})$ by \cite[Lemma 1.5]{NeebSquareIntegrable}:
\begin{equation}
\label{somethingmissing}
\tr_{L_{x,0}(\Lambda_{i;j})} e^{X} = d(\pi_{\Lambda_{i; j}}) \int_{Z_{x} \backslash \tilde{G}_{x,0}^\RR} \langle 
\pi_{\Lambda_{i; j}}(g^{-1})v,e^{X}\pi_{\Lambda_{i; j}}(g^{-1})v\rangle \ \mathrm{d}\mu(Z_{x}g)
\end{equation}
for some fixed element $v \in L_{x,0}(\Lambda_{i; j})$ with $\lVert v\rVert=1$. Then, setting
\begin{equation} 
\label{eq::form_c}
c(X; \Lambda_{i; j}) := \int_{Z_{x} \backslash \tilde{G}_{x,0}^\RR} \langle 
\pi_{\Lambda_{i; j}}(g^{-1})v,e^{X}\pi_{\Lambda_{i; j}}(g^{-1})v\rangle \ \mathrm{d}\mu(Z_{x}g).
\end{equation}
we conclude that the Witten index of $M$ is the formal dimension of the $\gg_{x,0}$-constituents of $\DS_{x}(M)$ 
with summands weighted by the $c(X; \Lambda_{i; j})$. 

\begin{lemma} 
\label{vacuousRelation}
The $Q$-Witten index of $M=L(\Lambda)$ is 
\[
I^{W}_{M}(Q,X) = \sum_{i,j} (-1)^{\Lambda -\Lambda_{i; j}} 
d(\pi_{\Lambda_{i; j}}) c(X;\Lambda_{i; j}),\quad .
 \] 
\end{lemma}

This allows in principle an analytic study of the $Q$-Witten index in the limit $X\to 0$ by evaluating 
$c(X; \Lambda_{i; j})$. Alternatively, returning to
\begin{equation}
\label{eq::tr_expression_Witten}
I^{W}_{M}(Q,X) = \sum_{i,j} (-1)^{\Lambda - \Lambda_{i; j}} 
\tr_{L_{0,x}(\Lambda_{i;j})} e^X
\end{equation}
we may interpret $\tr_{L_{0,x}(\Lambda_{i;j})} e^X$ (in a distributional sense) as the Harish-Chandra character of 
$\pi_{\Lambda_{i; j}}$, based on the fact that $e^{X} \in {T'_{x}}^{\reg}$.
To take the limit $X\to 0$, we construct for each $\gg_{x,0}$-constituent the associated $L$-packet by 
summing over the Weyl group orbit.
Writing 
\begin{equation}
\widetilde{\Theta}_{\DS_{x}(M)} := \sum_{i,j} (-1)^{\Lambda-\Lambda_{i;j}} 
\sum_{w\in W_x/W_{x,c}} \Theta_{\pi_{w\Lambda_{i;j}}}
\end{equation}
and
\begin{equation}
\tilde{I}_M^W(Q,X) := \widetilde{\Theta}_{\DS_x(M)} (e^X)
\end{equation}
in combination with Proposition \ref{prop::limit_l_packet}, the definition of the superdimension and eq.\ 
\eqref{eq::tr_expression_Witten}, we obtain our final result.

\begin{theorem} 
\label{finalBig}
Let $M$ be a discrete series $\gg$-supermodule, $Q,x$ as above. Then,
\[
\sdim(\DS_{x}(M)) = \lim_{X \to 0} \tilde I^{W}_{M}(Q, X) 
\]
\end{theorem}



\printbibliography

\end{document}